\documentclass[12pt, a4paper]{amsart}
\usepackage{amscd,amsmath,amssymb,amsthm,amsfonts}
\usepackage[alphabetic]{amsrefs}
\usepackage{mathrsfs}
\usepackage[shortlabels]{enumitem}
\usepackage[all]{xy}

\usepackage{xcolor}
\usepackage[hypertexnames=true,colorlinks = true, allcolors=blue,linktoc = all, pdffitwindow = false, urlbordercolor = white]{hyperref}%

\usepackage{tikz}
 
\usetikzlibrary{knots}
\usetikzlibrary{decorations.markings}

\usepackage{graphicx} 

\def\max{\operatorname{max}}

\newcommand{\IC}[0]{\mathbb{C}}

 \newcommand{\IN}[0]{\mathbb{N}}

 \newcommand{\IZ}[0]{\mathbb{Z}}

 \newcommand{\CF}[0]{\mathcal{F}}

\newcommand{\CS}[0]{\mathcal{S}} \newcommand{\CT}[0]{\mathcal{T}}

\newcommand {\stab}{{\rm Stab}}

\newtheorem{theorem}{Theorem}[section]
\newtheorem*{theorem*}{Theorem}
\newtheorem*{proposition*}{Proposition}
\newtheorem{proposition}[theorem]{Proposition}
\newtheorem{lemma}[theorem]{Lemma}
\newtheorem*{lemma*}{Lemma}

\newtheorem{definition}[theorem]{Definition}

\newtheorem*{question}{Question}
\newtheorem{corollary}[theorem]{Corollary}
\newtheorem{remark}[theorem]{Remark}

\newtheorem{convention}[theorem]{Convention}

\numberwithin{equation}{section}

\begin{document}
\title[On $\mathcal{F}$ and maximal subgroups of $F$]{On the $3$-colorable subgroup $\mathcal{F}$ and maximal subgroups of Thompson's group $F$}
\author{Valeriano Aiello} 
\address{Valeriano Aiello,
Mathematisches Institut, Universit\"at Bern,  Alpeneggstrasse 22, 3012 Bern, Switzerland
}\email{valerianoaiello@gmail.com}
\author{Tatiana Nagnibeda} 
\address{Tatiana Nagnibeda,
Section de Math\'  ematiques, Universit\' e de Gen\` eve, 2-4 rue du Li\` evre, Case Postale 64,
1211 Gen\` eve 4, Switzerland
}\email{tatiana.smirnova-nagnibeda@unige.ch}

\begin{abstract}
In his work on representations of Thompson's group $F$, Vaughan Jones defined and studied the $3$-\emph{colorable subgroup} $\mathcal{F}$ of $F$. Later, Ren showed that it is isomorphic with the Brown-Thompson group $F_4$.
In this paper we continue with the study of the $3$-colorable  subgroup and prove that the quasi-regular representation of $F$ associated with the $3$-colorable subgroup is irreducible.
We show moreover that the preimage of $\mathcal{F}$ under a certain injective endomorphism of $F$ is contained in three (explicit) maximal subgroups of $F$ of infinite index. These subgroups are different from the previously known infinite index maximal subgroups of $F$, namely the parabolic subgroups that fix a point in $(0,1)$, (up to isomorphism) the Jones' oriented subgroup $\vec{F}$,
and the explicit examples found by Golan.
\end{abstract}

\maketitle

\centerline{\it Dedicated to the memory of Vaughan F. R. Jones}\bigskip


\section*{Introduction}

In \cite{Jo14}, Vaughan Jones initiated a research program revolving around Thompson's groups and their  unitary representations.
In particular, he introduced two big families of such representations, one arising from planar algebras \cite{jo2} and one arising from the Pythagorean C$^*$-algebra \cite{BJ}.
The stabilizers of a canonical vector (the so-called vacuum vector) in these representations turn out to be    interesting subgroups of Thompson's groups. 
For example, for Thompson's group $F$,  
 the parabolic subgroups $\stab(t)\leq F$ (with $t\in (0,1)$) arise in this way from the  Pythagorean representations \cite{BJ}.
 Similarly, a certain representation related to the Temperley-Lieb planar algebra gives rise to the so-called oriented subgroup $\vec{F}\leq F$, which corresponds to the
 oriented links in Jones's encoding of  knots and links by elements of $F$, see \cite{Jo14, A}.

Another interesting example of similar origin is the so-called  $3$--colorable subgroup $\CF\leq F$, \cite{Jo16, Ren}. 
It is the main object of this paper 
and it is worth recalling its story.
 Roughly speaking, one of the original aims of Jones's project was to obtain representations of ${\rm Diff}^+(S^1)$
as limits of representations of Thompson's group $T$, seen as a group of homeomorphisms of $S^1$.
However, in \cite{Jo16} Jones discovered that, in general, this is not possible. 
In fact,   he defined two families of  unitary representations for Thompson's groups by means of the planar algebras of quantum $SO(3)$,  one for $F$ and one for $T$ (see \cite{MPS} for more information on these planar algebras). 
Then, he computed the weak limit of the rotations  by angle $2^{-n}$ in the representations of $T$,    when $n$ tends to infinity, and saw
that contrary to what one would hope, the limit is not the identity.
The $3$-colorable subgroup $\CF$ arises as the stabiliser of the vacuum vector in one of these representations of $F$, and we will now explain
its construction. 

Denote by $\CT_k$  the set of rooted planar finite $k$-regular trees. It is well known that the elements of Thompson's group $F$ can be viewed as equivalence classes of pairs of finite planar binary trees. 
Brown \cite{Brown} introduced a family of groups that share many properties with $F$. These groups, sometimes called Brown-Thompson groups $F_k$, $k\geq 2$, are groups of piecewise linear homeomorphisms of the interval $[0,1]$ with slopes powers of $k$ and points of non-continuity of the derivative $k$-adic rationals. Their elements can be described as equivalence classes of pairs of trees from $\CT_k$, $k\geq 2$. 
Besides the Thompson's groups $F$ and $T$, a third group was also introduced by Richard Thompson, namely $V$. 
This group can be described both as a group of homeomorphisms of the Cantor set and in terms of binary trees. 
The groups $F$ and $T$ are proper subgroups of $V$.
 After the work of Thompson and before that of Brown, Higman \cite{HIG} introduced a family of groups generalising $V$. 
 These are the groups $V_{n,r}$ and they form an 
  infinite family of finitely presented, infinite, simple  groups, 
with the additional parameter $r$ for the number of roots of the $n$-ary trees.

In \cite{Ren} Ren introduced the following construction of subgroups of $F$ isomorphic with $F_k$.
Let $T\in \CT_2$ be a rooted planar binary tree with $k$ leaves.
Define an injective map $\alpha_T: F_k \to F$:
given a pair of rooted $k$-regular trees $(T_+,T_-)\in F_k$, replace any vertex of degree $k+1$ with the tree $T$.

For a tree $T\in \CT_2$, denote by $\ell_T(0)$ the number of left edges in the path from the left-most leaf to the root and by $\ell_T(k-1)$ the number of right edges in the path from the right-most leaf to the root.
Recall that   the abelianization map  $\pi: F\to F/[F,F]=\mathbb{Z}\oplus \mathbb{Z}$ can be described as $\pi(f)=(\log_2 f'(0),\log_2 f'(1))$, see \cite{CFP}. 
If $f\in F$ is represented by a pair of trees $(T_+,T_-)$, then $\log_2 f'(0)$ is equal to 
$\ell_{T_+}(0)-\ell_{T_-}(0)$.
Similarly, $\log_2 f'(1)$ is equal to $\ell_{T_+}(k-1)-\ell_{T_-}(k-1)$.

Bleak and Wassink considered in \cite{BW}, for any $a$, $b\in \mathbb{N}$,  the rectangular subgroups of $F$ defined as
$$
K_{(a,b)}:=\{f\in F\; | \;  \log_2 f'(0)\in a \mathbb{Z},  \log_2 f'(1)\in b \mathbb{Z} \}
$$
All these subgroups are finite index subgroups  of $F$ isomorphic with it.  
It is not difficult to see that the subgroup $\alpha_T(F_k)$ sits inside   $K_{(\ell_T(0),\ell_T(k-1))}$.

The first example of this construction is Jones's oriented subgroup $\vec{F}$. It corresponds to the case $k=3$ where there is essentially the unique map $\alpha_T: F_3 \rightarrow F$ depicted on Figure \ref{fig-ren-map} and its image is precisely the oriented subgroup $\vec{F}$. 
In fact, Golan and
Sapir were the first to prove that $\vec F$ is isomorphic with $F_3$, see
 \cite{GS}.  

\begin{figure}
\phantom{This text will be invisible} 
\[\begin{tikzpicture}[x=.75cm, y=.75cm,
    every edge/.style={
        draw,
      postaction={decorate,
                    decoration={markings}
                   }
        }
]

\draw[thick] (0,0)--(.5,.5)--(1,0);
\draw[thick] (0.5,0.75)--(.5,0);
\node at (0,-1.2) {$\;$};
\node at (1.75,0.25) {$\scalebox{1}{$\mapsto$}$};

\end{tikzpicture}
\begin{tikzpicture}[x=.75cm, y=.75cm,
    every edge/.style={
        draw,
      postaction={decorate,
                    decoration={markings}
                   }
        }
]

\draw[thick] (0.5,0.75)--(.5,.5);
\draw[thick] (0.5,0)--(.75,.25);
\draw[thick] (0,0)--(.5,.5)--(1,0);
\node at (0,-1.2) {$\;$};
\end{tikzpicture}
\]
 \caption{Ren's map 
 for $\vec{F}$.} 
  \label{fig-ren-map}
\end{figure}
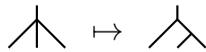 
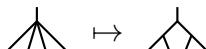
\begin{figure}
\[\begin{tikzpicture}[x=.75cm, y=.75cm,
    every edge/.style={
        draw,
      postaction={decorate,
                    decoration={markings}
                   }
        }
]

\draw[thick] (0,0)--(.5,.5)--(1,0);
\draw[thick] (0.5,0.5)--(.5,.75);
\draw[thick] (0.5,0.5)--(.35,0);
\draw[thick] (0.5,0.5)--(.65,0);
\node at (0,-1.2) {$\;$};
\node at (1.75,0.25) {$\scalebox{1}{$\mapsto$}$};

\end{tikzpicture}
\begin{tikzpicture}[x=.75cm, y=.75cm,
    every edge/.style={
        draw,
      postaction={decorate,
                    decoration={markings}
                   }
        }
]

\draw[thick] (0.5,0.75)--(.5,.5);
\draw[thick] (0.65,0)--(.75,.25);
\draw[thick] (0.35,0)--(.25,.25);
\draw[thick] (0,0)--(.5,.5)--(1,0);
\node at (0,-1.2) {$\;$};
\end{tikzpicture}
\]
\caption{The map 
 for  $\CF$.}
  \label{fig-ren-map-2}
\end{figure} 

The oriented subgroup $\vec{F}$ was originally defined as the set of pairs of trees $(T_+,T_-)$ such that the value of the chromatic polynomial ${\rm Chr}_{\Gamma(T_+,T_-)}(2)$ is non-zero, where $\Gamma(T_+,T_-)$ is 
a certain graph associated
with $(T_+,T_-)$, 
 see \cite{Jo14}. Since $\Gamma(T_+,T_-)$ is always connected, this specialisation of the chromatic polynomial takes only two values: $0$ and $2$.
In passing, we briefly describe how Ren proved that $\vec{F}$ can be realised as the image of $\alpha_T$ described in Figure \ref{fig-ren-map}, \cite{Ren}. 
This is done with  an inductive argument on the number of the leaves. 
First, by investigating the $2$-colorings of $\Gamma(T_+,T_-)$,
 he showed that every element of $\vec{F}$ is generated by the \emph{positive}\footnote{See Section \ref{sec1} for a definition of positive elements of $F$.} elements in $\vec{F}$. 
Then he proved that each positive element in $\vec{F}$ contains a binary tree of the form depicted in Figure \ref{fig-ren-map}, and so by multiplying this element by $(x_ix_{i+1})^{-1}$, for a
suitable $i$, it was possible to cancel this subtree (thus reducing the number of leaves and ending the inductive argument).  

More generally, by using the Temperley-Lieb planar algebra, for every $t\in \{4\cos^2(\pi/n)\; | \; n\geq 4\}\cup [4,\infty)$ one can define 
a family of representations $(\pi_t)_t$ of $F$ such that 
$$
\langle \pi_t(T_+,T_-)\Omega,\Omega\rangle=\frac{{\rm Chr}_{\Gamma(T_+,T_-)}(t)}{t(t-1)^{n-1}} 
$$
where $n$ is the number of leaves of $T_+$ and $\Omega$ is a canonical unit vector which is usually called the vacuum vector.
The stabiliser of $\Omega$ consists of the elements $(T_+,T_-)$ in $F$ such that
$$
{\rm Chr}_{\Gamma(T_+,T_-)}(t)=t(t-1)^{n-1} 
$$
For non-integer values of $t$, the computation of the chromatic polynomial is often more demanding. However, it was shown in  \cite{ABC} that for $t>4.63$ these subgroups are all trivial. In passing, we mention that several other knot and graph invariants can also be used to define unitary representations of Thompson's groups $F$ and $\vec{F}$, see e.g. \cite{Jo14} and \cite{AiCo1, AiCo2, ACJ, ABC}.

Similarly, the $3$-colorable subgroup $\CF$ is the image of the map $\alpha_T:F_4\rightarrow F$ depicted in Figure  \ref{fig-ren-map-2}.
It can also  be defined in terms of a specialisation of the chromatic polynomial of the following graph. Take an element $(T_+,T_-)$ in $F$, join the two roots by an edge and draw its dual graph $\tilde \Gamma(T_+,T_-)$. The elements of  $\mathcal{F}$ are exactly those for which the chromatic polynomial of $\tilde \Gamma(T_+,T_-)$ evaluated at $3$ is non-zero
(and in this case the only possible value is $6$). See Section \ref{sec1} for more details. 

It turns out that the subgroups arising as stabilizers of the vacuum vector in these various representations are also interesting 
from the viewpoint of understanding the maximal subgroups in Thompson's groups.
Recall that parabolic subgroups are natural examples of maximal subgroups of infinite index in $F$, \cite{Sav, Sav2}. The oriented subgroup provided, up to an isomorphism, the first explicit example of a maximal subgroup 
of infinite index of $F$ without fixed points in the open unit interval $(0,1)$, as proven by Golan and Sapir in \cite{GS2}. 
Moreover, in the same paper a method for potentially producing more examples of maximal subgroups of infinite index in $F$ was introduced (see Section 4).
Later, three further explicit examples of maximal infinite index subgroups without fixed points appeared in \cite[Section 10.3]{G}. 
One might wonder whether all maximal subgroups of infinite index arise as stabilizers of suitable subsets.
However, this is not the case, at least when we restrict to the subset of dyadic rationals.
Indeed, Golan showed  in \cite{G} that one of the aforementioned examples of maximal subgroups acts transitively on this set.
A more subtle problem is to   determine the isomorphism classes of infinite index maximal subgroups.
Despite being all distinct, the parabolic subgroups actually reduce to only three isomorphism classes as shown in \cite{GS3} 
(see also the recent paper \cite{DF} for a similar
classification in the wide context of groups with micro-supported actions).
The general problem of describing and classifying the maximal subgroups in $F$ remains very much open.

Maximal subgroups are also of interest in the study of unitary representations by means of quasi-regular representations. By a classical result of Mackey \cite{Ma}, 
the quasi-regular representations associated with a subgroup is irreducible if the subgroup coincides with its commensurator.
For a maximal subgroup it is then enough to exclude its commensurator being equal to the whole group
to conclude that the quasi-regular representations is irreducible.
It is an important problem to determine whether the unitary representations of Thompson's group defined by Jones are irreducible.
Positive evidence for that includes Golan and Sapir's work \cite{GS},
where they showed that $\vec{F}$ coincides with its commensurator,
as well as the recent paper of Jones \cite{Jo19}.

Let us now say a few words on the results of our paper. 
The first   result is that, 
 the $3$-colorable subgroup can be described as the intersection of the stabilizers of certains subsets of dyadic rationals, namely  of
$$
S_i:=\{t\in (0,1)\cap \IZ[1/2]\; | \; \omega(t)= i\} \qquad i=0, 1, 2\; 
$$
where 
 $\omega(t)$ is some natural function on binary words with values in $\{0,1,2\}$ (see Section \ref{sec2} for a precise definition).
 In fact, in 
 Lemma \ref{lemma-inclu} 
 we prove that   $\CF$ is equal to the intersection of the subgroups ${\rm Stab}(S_i)$, for $i=0, 1, 2$.
This allows us to show in Theorem \ref{theo2} that $\CF$   coincides with its commensurator and, in turn, this implies that the quasi-regular representation of $F$ associated with $\CF$ is irreducibile (see Corollary \ref{cor-irred}).
The second main result of this paper is contained in Theorem 
 \ref{thmM1}, where we exhibit three subgroups $M_0$, $M_1$, $M_2$ between $\CF$ and the rectangular subgroup $K_{(2,2)}$.  These turn out to be maximal subgroups of infinite index in $K_{(2,2)}$.
By means of 
an isomorphism between $F$ and $K_{(2,2)}$, we obtain three infinite index maximal subgroups of $F$, which are shown to be distinct from
all previously known examples of maximal infinite index subgroups of $F$:
 the parabolic subgroups, the oriented subgroup, as well as the examples in \cite[Section 3.2]{GS} and 
\cite[Section 10.3.B]{G}.
We also provide partial evidence that the only subgroups between $\CF$ and $K_{(2,2)}$ are $M_0$, $M_1$, $M_2$. 

We end this introduction with the structure of the paper. 
In Section \ref{sec1} we recall the definitions of Thompson's group $F$ and of the $3$-colorable subgroup, along with some of their main properties. 
In Section \ref{sec2} we provide a description of $\CF$ as the homeomorphisms preserving certain subsets of dyadic rationals. 
We exploit this description 
 to prove the irreducibility of the 
quasi-regular representation of $F$ associated with $\CF$.
In Section \ref{sec4} we exhibit an explicit isomorphism $\theta$ between $F$ and the rectangular subgroup $K_{(2,2)}$ consisting of the homeomorphisms whose derivatives at $0$ and $1$ are in $2^{2\IZ}$.
In Section \ref{sec4b} we first show that $\CF$ is contained in $K_{(2,2)}$, then we exhibit  three infinite index maximal subgroups of $K_{(2,2)}$.  
In Section \ref{sec5} we compare $\theta^{-1}(M_0)$, $\theta^{-1}(M_1)$, and $\theta^{-1}(M_2)$ to other infinite index maximal subgroups of $F$ that have been identified before.

\section{Preliminaries and notation}\label{sec1}
In this section we recall the definitions and basic properties of Thompson's group $F$ and of Jones's $3$-colorable subgroup. 
The interested reader is referred to \cite{CFP} and \cite{B} for a general introduction on Thompson’s groups and its basic properties, 
to \cite{Ren} for further information on the $3$-colorable subgroup.

Thompson's group $F$ is the group of all piecewise linear homeomorphisms of the unit interval $[0,1]$ that are differentiable everywhere except at finitely many dyadic rationals numbers and such that on the intervals of differentiability the derivatives are powers of $2$. We adopt the standard notation: $f\cdot g(t)=g(f(t))$.

Thompson's group   has the following infinite presentation
$$
F=\langle x_0, x_1, \ldots \; | \; x_nx_k=x_kx_{n+1} \quad \forall \; k<n\rangle\, .
$$
The monoid generated by $x_0, x_1, \ldots$ is denoted by $F_+$. Its elements are said to be positive. Note that $x_0$ and $x_1$ are enough to generate $F$ (see Figure \ref{genThompsonF} for their description in terms of pairs of binary trees).
\begin{figure}
\phantom{This text will be invisible} 
\[
\begin{tikzpicture}[x=.35cm, y=.35cm,
    every edge/.style={
        draw,
      postaction={decorate,
                    decoration={markings}
                   }
        }
]

\node at (-1.5,0) {$\scalebox{1}{$x_0=$}$};
\node at (-1.25,-3) {\;};

\draw[thick] (0,0) -- (2,2)--(4,0)--(2,-2)--(0,0);
 \draw[thick] (1,1) -- (2,0)--(3,-1);

 \draw[thick] (2,2)--(2,2.5);

 \draw[thick] (2,-2)--(2,-2.5);

\end{tikzpicture}\qquad
\;\;
\begin{tikzpicture}[x=.35cm, y=.35cm,
    every edge/.style={
        draw,
      postaction={decorate,
                    decoration={markings}
                   }
        }
]

\node at (-3.5,0) {$\scalebox{1}{$x_1=$}$};
\node at (-1.25,-3.25) {\;};

\draw[thick] (2,2)--(1,3)--(-2,0)--(1,-3)--(2,-2);

\draw[thick] (0,0) -- (2,2)--(4,0)--(2,-2)--(0,0);
 \draw[thick] (1,1) -- (2,0)--(3,-1);

 \draw[thick] (1,3)--(1,3.5);
 \draw[thick] (1,-3)--(1,-3.5);

\end{tikzpicture}
%
%
%
%
%
%
%
\]
\caption{The generators of $F=F_2$.}\label{genThompsonF}
\end{figure}
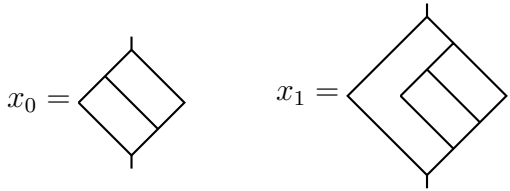
 Similarly,
for any $k\geq 2$, the Brown-Thompson group $F_k$ admits the following presentation  
$$
\langle y_0, y_1, \ldots \; | \; y_ny_l=y_ly_{n+k-1} \quad \forall \; l<n\rangle\, .
$$
The elements $y_0, y_1, \ldots , y_{k-1}$ generate $F_k$ (see Figure \ref{genThompsonF4} for the generators of $F_4$). 
 \begin{figure}
\phantom{This text will be invisible} 
\[
\begin{tikzpicture}[x=1.25cm, y=1.25cm,
    every edge/.style={
        draw,
      postaction={decorate,
                    decoration={markings}
                   }
        }
]

\draw[thick] (0,0)--(.5,.5)--(1,0);
\draw[thick] (0.5,0.5)--(.35,0);
\draw[thick] (0.5,0.5)--(.65,0);
\node at (0,-1.2) {$\;$};
\node at (-.75,0) {$\scalebox{1}{$y_0=$}$};

\draw[thick] (1.25,1.5)--(1.25,1.25);
\draw[thick] (1.25,-1.5)--(1.25,-1.25);

\draw[thick] (0.5,0.5)--(1.25,1.25);
\draw[thick] (1.25,0)--(1.25,1.25);
\draw[thick] (1.75,0)--(1.25,1.25);
\draw[thick] (2.5,0)--(1.25,1.25);

\draw[thick] (2.5,0)--(1.25,-1.25)--(0,0);
\draw[thick] (1.25,-1.25)--(0.35,0);
\draw[thick] (1.25,-1.25)--(0.65,0);

\draw[thick] (1.75,-.75)--(1,0);
\draw[thick] (1.75,-.75)--(1.25,0);
\draw[thick] (1.75,-.75)--(1.75,0);

\end{tikzpicture}
\qquad
\begin{tikzpicture}[x=1.25cm, y=1.25cm,
    every edge/.style={
        draw,
      postaction={decorate,
                    decoration={markings}
                   }
        }
]

\draw[thick] (0.5,0)--(1,.5)--(1.5,0);
\draw[thick] (1,0.5)--(.85,0);
\draw[thick] (1,0.5)--(1.05,0);
\node at (0,-1.2) {$\;$};
\node at (-.75,0) {$\scalebox{1}{$y_1=$}$};

\draw[thick] (0,0)--(1.25,1.25);
\draw[thick] (1,.5)--(1.25,1.25);
\draw[thick] (1.75,0)--(1.25,1.25);
\draw[thick] (2.5,0)--(1.25,1.25);

\draw[thick] (2.5,0)--(1.25,-1.25)--(0,0);
\draw[thick] (1.25,-1.25)--(0.5,0);
\draw[thick] (1.25,-1.25)--(0.85,0);

\draw[thick] (1.75,-.75)--(1.05,0);
\draw[thick] (1.75,-.75)--(1.5,0);
\draw[thick] (1.75,-.75)--(1.75,0);

\draw[thick] (1.25,1.5)--(1.25,1.25);
\draw[thick] (1.25,-1.5)--(1.25,-1.25);

\end{tikzpicture}
\]
\[
\begin{tikzpicture}[x=1.25cm, y=1.25cm,
    every edge/.style={
        draw,
      postaction={decorate,
                    decoration={markings}
                   }
        }
]

\draw[thick] (1,0)--(1.5,.5)--(2,0);
\draw[thick] (1.5,0.5)--(1.35,0);
\draw[thick] (1.5,0.5)--(1.55,0);
\node at (0,-1.2) {$\;$};
\node at (-.75,0) {$\scalebox{1}{$y_2=$}$};

\draw[thick] (0,0)--(1.25,1.25);
\draw[thick] (.5,0)--(1.25,1.25);
\draw[thick] (1.5,.5)--(1.25,1.25);
\draw[thick] (2.5,0)--(1.25,1.25);

\draw[thick] (2.5,0)--(1.25,-1.25)--(0,0);
\draw[thick] (1.25,-1.25)--(0.5,0);
\draw[thick] (1.25,-1.25)--(1,0);

\draw[thick] (1.75,-.75)--(1.35,0);
\draw[thick] (1.75,-.75)--(1.55,0);
\draw[thick] (1.75,-.75)--(2,0);

\node at (1.75,.75) {$\;$};

\draw[thick] (1.25,1.5)--(1.25,1.25);
\draw[thick] (1.25,-1.5)--(1.25,-1.25);

\end{tikzpicture}
\qquad
\begin{tikzpicture}[x=1cm, y=1cm,
    every edge/.style={
        draw,
      postaction={decorate,
                    decoration={markings}
                   }
        }
]

\draw[thick] (.75,1.75)--(.75,2);
\draw[thick] (.75,-1.75)--(.75,-2);

\draw[thick] (0,0)--(.5,.5)--(1,0);
\draw[thick] (0.5,0.5)--(.35,0);
\draw[thick] (0.5,0.5)--(.65,0);
\node at (1.75,-.75) {$\;$};
\node at (-1.75,0) {$\scalebox{1}{$y_3=$}$};

\draw[thick] (0.5,0.5)--(1.25,1.25);
\draw[thick] (1.25,0)--(1.25,1.25);
\draw[thick] (1.75,0)--(1.25,1.25);
\draw[thick] (2.5,0)--(1.25,1.25);

\draw[thick] (-1,0)--(.75,1.75)--(1.25,1.25);
\draw[thick] (-1,0)--(.75,-1.75)--(1.25,-1.25);
\draw[thick] (-.65,0)--(0.75,1.75);
\draw[thick] (-.25,0)--(0.75,1.75);
\draw[thick] (-.65,0)--(0.75,-1.75);
\draw[thick] (-.25,0)--(0.75,-1.75);

\draw[thick] (2.5,0)--(1.25,-1.25)--(0,0);
\draw[thick] (1.25,-1.25)--(0.35,0);
\draw[thick] (1.25,-1.25)--(0.65,0);

\draw[thick] (1.75,-.75)--(1,0);
\draw[thick] (1.75,-.75)--(1.25,0);
\draw[thick] (1.75,-.75)--(1.75,0);

\end{tikzpicture}
\]
\caption{The generators of $F_4$.}\label{genThompsonF4}
\end{figure}
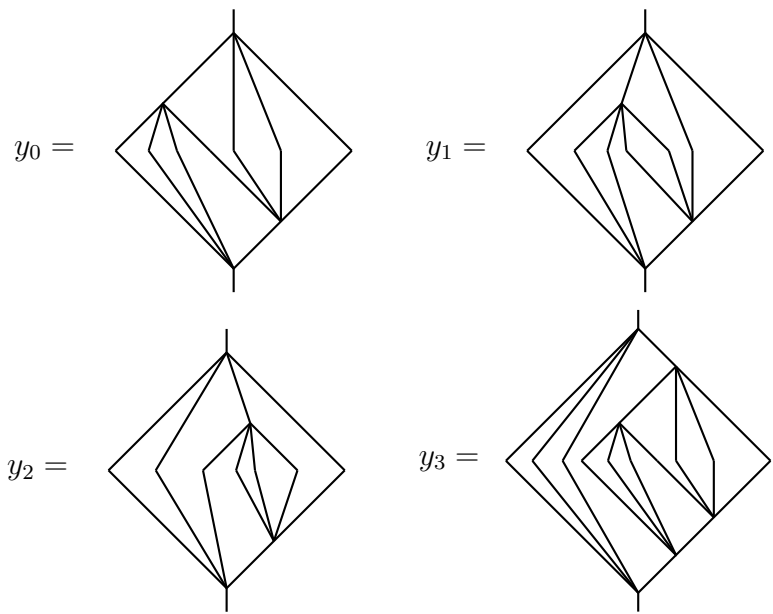

The projection of $F$ onto its abelianisation is denoted by $\pi: F\to F/[F,F]=\IZ\oplus \IZ$ and it admits a nice interpretation when $F$ is seen as group of homeomorphisms: $\pi(f)=(\log_2 f'(0),\log_2 f'(1))$. 
 
There is still another description of $F$  which is relevant to this paper: the elements of $F$ can be seen as pairs $(T_+,T_-)$ of planar binary rooted trees (with the same number of leaves). 
We draw one tree upside down on top of the other; $T_+$ is  the top tree, while $T_-$ is the bottom tree.
Any pair of trees $(T_+,T_-)$ represented in this way is called a tree diagram. %
Two pairs of trees are said to be equivalent if they differ by pairs of opposing carets, namely
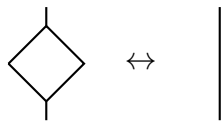
\[\begin{tikzpicture}[x=.5cm, y=.5cm,
    every edge/.style={
        draw,
      postaction={decorate,
                    decoration={markings}
                   }
        }
]

 \draw[thick] (0,0)--(1,1)--(2,0)--(1,-1)--(0,0); 
 \draw[thick] (1,1.5)--(1,1); 
 \draw[thick] (1,-1.5)--(1,-1); 
\node at (0,-1.2) {$\;$};
\node at (3.5,0) {$\scalebox{1}{$\leftrightarrow$}$};

\end{tikzpicture}
\begin{tikzpicture}[x=.5cm, y=.5cm,
    every edge/.style={
        draw,
      postaction={decorate,
                    decoration={markings}
                   }
        }
]

  \draw[thick] (1,1.5)--(1,-1.5); 
\node at (0,-1.2) {$\;$};
 
\end{tikzpicture}
\]
Every equivalence class of pairs of trees (i.e. an element of $F$) gives rise to exactly one tree diagram which is reduced, in the sense that 
the number of its vertices is minimal, \cite{B}.

\begin{convention}\label{conventiondrawings}
We make a convention about how we draw trees on the plane.
The roots of our planar binary trees are drawn as vertices of degree $3$ and hence each tree diagram has the uppermost and lowermost vertices of degree  $1$, which lie respectively on the lines   $y=1$ and $y=-1$.
 The leaves of the trees sit on the $x$-axis, precisely on the non-negative integers. 
\end{convention}

Any tree diagram partitions the strip bounded by the lines $y=1$ and $y=-1$ in regions.
This strip may or may not be $3$-colorable, i.e., it may or may not be 
 possible to assign the colors $\IZ_3=\{0,1, 2\}$ to the regions of the strip in such a way that if two regions share an edge, they   have different colors. 
  By convention, we assign the following colours to the regions near the roots 
  \begin{eqnarray}\label{convention-colors}
\begin{tikzpicture}[x=.5cm, y=.5cm,
    every edge/.style={
        draw,
      postaction={decorate,
                    decoration={markings}
                   }
        }
] 

 \draw[thick] (0,0) -- (1,1)--(2,0);
 \draw[thick] (1,1) -- (1,2);



\node at (0,1) {$0$};
\node at (2,1) {$1$};
\node at (1,0.2) {$2$};

 \draw[thick] (-1,2) -- (3,2);


\end{tikzpicture}
\qquad\qquad 
\begin{tikzpicture}[x=.5cm, y=.5cm,
    every edge/.style={
        draw,
      postaction={decorate,
                    decoration={markings}
                   }
        }
] 

 \draw[thick] (-1,0) -- (3,0);

 \draw[thick] (0,2) -- (1,1)--(2,2);
 \draw[thick] (1,1) -- (1,0);



\node at (0,1) {$0$};
\node at (2,1) {$1$};
\node at (1,1.8) {$2$};

\end{tikzpicture}
\end{eqnarray}
Once we make this convention, if the strip is $3$-colourable, 
there exists a unique colouring. 
The \textbf{$3$-colorable subgroup} $\CF$  consists of the elements of $F$ for which the corresponding  strip is $3$-colorable. 
For example, this is the strip corresponding to $x_0$ (which is not $3$-colorable)
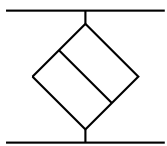
\[
\begin{tikzpicture}[x=.35cm, y=.35cm,
    every edge/.style={
        draw,
      postaction={decorate,
                    decoration={markings}
                   }
        }
]

 \draw[thick] (-1,2.5)--(5,2.5);
 \draw[thick] (-1,-2.5)--(5,-2.5);

 \node at (-1.25,-3) {\;};

\draw[thick] (0,0) -- (2,2)--(4,0)--(2,-2)--(0,0);
 \draw[thick] (1,1) -- (2,0)--(3,-1);

 \draw[thick] (2,2)--(2,2.5);

 \draw[thick] (2,-2)--(2,-2.5);

\end{tikzpicture}
\]
The $3$-colorable subgroup subgroup was introduced by Jones and studied by Ren, who proved the following result.
\begin{theorem}[\cite{Ren}] \label{generatoriCF}
The map $\alpha_T$ with $T$ depicted in Figure \ref{fig-ren-map}
is an isomorphism between $F_4$ and the
 $3$-colorable subgroup. 
In particular, the  $3$-colorable subgroup
is generated by the following elements
$w_0:=x_0^2x_1x_2^{-1}$,
$w_1:=x_0x_1^2x_0^{-1}$,
$w_2:=x_1^2x_3x_2^{-1}$,
$w_3:=x_2^2x_3x_4^{-1}$ (see Figure \ref{genCF}), which are the images of the generators of $F_4$. 
\end{theorem}
 \begin{figure}
\phantom{This text will be invisible} 
\[
\begin{tikzpicture}[x=.75cm, y=.75cm,
    every edge/.style={
        draw,
      postaction={decorate,
                    decoration={markings}
                   }
        }
]

\node at (-1.5,-.25) {$\scalebox{1}{$w_0=$}$};

\draw[thick] (0,0)--(2,2)--(4,0);
\draw[thick] (0,0)--(2,-2)--(4,0);
\draw[thick] (1,0)--(1.5,0.5)--(2,0);
\draw[thick] (1.5,0.5)--(1,1);
\draw[thick] (3,0)--(1.5,1.5);
\draw[thick] (3,0)--(2.5,-.5); 
\draw[thick] (4,0)--(2,2);
\draw[thick] (4,0)--(3.5,-.5);
\draw[thick] (1,0)--(2.5,-1.5);
\draw[thick] (2,0)--(3,-1);

\draw[thick] (2,2)--(2,2.25);
\draw[thick] (2,-2)--(2,-2.25);

\node at (0,-1.2) {$\;$};
\end{tikzpicture}
\qquad 
\begin{tikzpicture}[x=.75cm, y=.75cm,
    every edge/.style={
        draw,
      postaction={decorate,
                    decoration={markings}
                   }
        }
]

\node at (-1.5,-.25) {$\scalebox{1}{$w_1=$}$};

\draw[thick] (0,0)--(2,2)--(4,0);
\draw[thick] (0,0)--(2,-2)--(4,0);

\draw[thick] (1,0)--(1.5,0.5)--(2,0);
\draw[thick] (2,0)--(3,-1);
\draw[thick] (1,0)--(.5,-.5);
\draw[thick] (1.5,.5)--(2,1)--(3,0);
\draw[thick] (3,0)--(3.5,-.5);
\draw[thick] (2,1)--(1.5,1.5);

\draw[thick] (2,2)--(2,2.25);
\draw[thick] (2,-2)--(2,-2.25);

\node at (0,-1.2) {$\;$};
\end{tikzpicture}
\]

\[
\begin{tikzpicture}[x=.75cm, y=.75cm,
    every edge/.style={
        draw,
      postaction={decorate,
                    decoration={markings}
                   }
        }
]

\node at (-2.5,-.25) {$\scalebox{1}{$w_2=$}$};
 
\draw[thick] (-1,0)--(1.5,2.5)--(2,2);
\draw[thick] (-1,0)--(1.5,-2.5)--(2,-2);

\draw[thick] (0,0)--(2,2)--(4,0);
\draw[thick] (0,0)--(2,-2)--(4,0);

\draw[thick] (2,0)--(2.5,.5)--(3,0);

\draw[thick] (1,0)--(.5,.5);
\draw[thick] (2.5,.5)--(1.5,1.5);

\draw[thick] (3,0)--(3.5,-.5);
\draw[thick] (2,0)--(1.5,-.5)--(1,0);
\draw[thick] (1.5,-.5)--(2.5,-1.5);

\draw[thick] (1.5,2.5)--(1.5,2.75);
\draw[thick] (1.5,-2.5)--(1.5,-2.75);

\node at (0,-1.2) {$\;$};
\end{tikzpicture}
\qquad 
\begin{tikzpicture}[x=.75cm, y=.75cm,
    every edge/.style={
        draw,
      postaction={decorate,
                    decoration={markings}
                   }
        }
]

\node at (-3.5,-.25) {$\scalebox{1}{$w_3=$}$};
 
\draw[thick] (-1,0)--(1.5,2.5)--(2,2);
\draw[thick] (-1,0)--(1.5,-2.5)--(2,-2);

\draw[thick] (-2,0)--(1,3)--(1.5,2.5);
\draw[thick] (-2,0)--(1,-3)--(1.5,-2.5);

\draw[thick] (0,0)--(2,2)--(4,0);
\draw[thick] (0,0)--(2,-2)--(4,0);
\draw[thick] (1,0)--(1.5,0.5)--(2,0);
\draw[thick] (1.5,0.5)--(1,1);
\draw[thick] (3,0)--(1.5,1.5);
\draw[thick] (3,0)--(2.5,-.5); 
\draw[thick] (4,0)--(2,2);
\draw[thick] (4,0)--(3.5,-.5);
\draw[thick] (1,0)--(2.5,-1.5);
\draw[thick] (2,0)--(3,-1);

\draw[thick] (1,3)--(1,3.25);
\draw[thick] (1,-3)--(1,-3.25);

\node at (0,-1.2) {$\;$};
\end{tikzpicture}
\]
 \caption{The generators of $\CF$. 
 }\label{genCF}
\end{figure} 

The next couple of easy lemmas will come in handy in the other sections of the paper.
\begin{lemma}\label{lemmasigma}
Let $\sigma: F\to F$ be the order $2$ automorphism  obtained by reflecting tree diagrams about a vertical line.
For any $g\in F$, we have $g\in\CF$ if and only if $\sigma(g)\in\CF$.
\end{lemma}
\begin{proof}
Clearly the strip associated with $g$ is $3$-colorable if and only if the strip associated with $\sigma(g)$ is $3$-colorable (the second strip is obtained by the first after a reflection about a vertical line). 
\end{proof}

\begin{lemma}\label{lemmashift}
Consider the shift homomorphism  $\varphi : F\to  F$ defined graphically  as
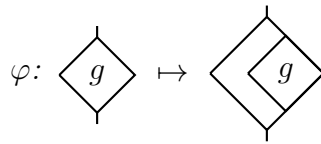
\[
\begin{tikzpicture}[x=1cm, y=1cm,
    every edge/.style={
        draw,
      postaction={decorate,
                    decoration={markings}
                   }
        }
]

\node at (-.45,0) {$\scalebox{1}{$\varphi$:}$};

\draw[thick] (0,0)--(.5,.5)--(1,0)--(.5,-.5)--(0,0);
\node at (1.5,0) {$\scalebox{1}{$\mapsto$}$};

\node at (.5,0) {$\scalebox{1}{$g$}$};
\node at (.5,.-.75) {$\scalebox{1}{}$};

 \draw[thick] (.5,.65)--(.5,.5);
 \draw[thick] (.5,-.65)--(.5,-.5);
 
\end{tikzpicture}
\begin{tikzpicture}[x=1cm, y=1cm,
    every edge/.style={
        draw,
      postaction={decorate,
                    decoration={markings}
                   }
        }
]
 
\node at (2.5,0) {$\scalebox{1}{$g$}$};

\draw[thick] (1.5,0)--(2.25,.75)--(3,0)--(2.25,-.75)--(1.5,0);
\draw[thick] (2,0)--(2.5,.5)--(3,0)--(2.5,-.5)--(2,0);
 
 \node at (1.5,.-.75) {$\scalebox{1}{}$};


 \draw[thick] (2.25,.75)--(2.25,.9);
 \draw[thick] (2.25,-.75)--(2.25,-.9);

\end{tikzpicture}
\]
Then $g\in\CF$ if and only if $\varphi(g)\in\CF$.
\end{lemma}
\begin{proof}
The claim is clear after drawing the strips corresponding to $g$ and $\varphi(g)$.
\end{proof}
In terms of homeomorphisms of $[0,1]$, the shift homomorphism $\varphi$ takes elements of $F$ and squeezes them onto $[1/2,1]$.
In particular,  $\varphi$ maps $x_i$ to $x_{i+1}$ for all $i\geq 0$. We observe that the range of $\varphi$ is the subgroup of elements of $F$ that act trivially on $[0, 1/2]$.

\section{The $3$-colorable subgroup as a stabiliser subgroup}\label{sec2}
The goal of this section is to provide a description of $\CF$ as the intersection of stabilisers of certains subsets of the dyadic rationals.
This result is analogous to the one obtained by Golan and Sapir for $\vec{F}$, where it was realised as the stabiliser of a subset of dyadic rationals, namely that consisting of elements with an even number of digits  equal to $1$ in the binary expansion. Our approach is similar, but the proof is more involved as our subset of dyadic rationals is more complicated and in order to describe it we introduce a suitable \emph{weight} function $\omega$ on dyadic rationals viewed as paths from the root of the tree of standard dyadic intervals, which is intimately  related to the coloring of the regions in the strips associated  with the elements $\CF$.  
  
Consider a   rooted binary planar tree $T$. We draw it in the upper-half plane with its leaves on the $x$-axis   and 
the highest vertex (of degree $1$) on the line $y=1$.
Given a vertex of  a tree, there exists a unique minimal path from the root of the tree to the vertex. This path is made by a collection of left, right edges, and may be represented by a word  in the letters $\{0,1\}$ ($0$ stands for a left edge, $1$ for a right edge). 
An easy inductive argument on the number of vertices shows that
the strip bounded by the line $y=1$ and the $x$-axis is always $3$-colourable. We adopt the same convention as in \eqref{convention-colors}, that is the region to the left of the root is colored with $0$, the region to the right with $1$ and the region below with $2$. After this choice, there is a unique coloring of the strip.

Given a vertex $v$ of $T$, we denote by $\omega(v)$ the color of the region to the left of $v$. 
We call $\omega(\cdot)$ the weight associated with $T$.
The weight $\omega$ can be actually defined on the infinite binary rooted planar tree (where every vertex has two descendants). 
In fact, given a rooted subtree $S$ of $T$, the restriction of the weight of $T$ to the vertices of $S$ is the weight of $S$. 
This allows one to consider the weight of the infinite rooted regular binary tree.
If we denote by $W_2$ the set of finite binary words, 
this yields a function $\omega: W_2\to \IZ_3$.
When we consider a tree diagram $(T_+,T_-)$, we denote by $\omega_+(\cdot)$ and $\omega_-(\cdot)$ the weights associated with the top tree and the bottom tree (after a reflection about the $x$-axis), respectively.

The following result follows at once from the definitions.
\begin{proposition}\label{prop-descr-3col}
It holds
$$
\CF=\{(T_+,T_-)\in F\; |\; \omega_+(i)=\omega_-(i) \; \forall i\geq 0\}
$$
where $\omega_+(i)$ and $\omega_-(i)$  
stand for the colors associated with the $i$-th leaf of the trees $T_+$ and $T_-$, respectively.
\end{proposition}

The next couple of lemmas are instrumental for obtaining a description of $\CF$ as a stabiliser subgroup.
\begin{lemma}\label{lemma1} 
For all $\alpha$, $\beta\in W_2$,
it holds 
\begin{align}
& \omega(\alpha 00 \beta)=\omega(\alpha \beta) \label{formula2}\\
& \omega(\alpha 11 \beta)=\omega(\alpha \beta)\label{formula3}\\
& \omega(\alpha 0)=\omega(\alpha) \label{formula4}
\end{align}
\end{lemma}
\begin{proof}
As \eqref{formula2} and \eqref{formula3},
the figure below shows that after two consecutive left and right edges, the colours labelling the regions surrounding a vertex are the same (the black letters denote the colours of the regions, whereas the red ones describe the path from the root to the vertex).
\[
\begin{tikzpicture}[x=.75cm, y=.75cm,
    every edge/.style={
        draw,
      postaction={decorate,
                    decoration={markings}
                   }
        }
]
\draw[thick] (0,0) -- (1,1)--(2,0);
 \draw[thick] (1,1) -- (1,2);
 \draw[thick] (-1,-1)--(0,0)--(1,-1);
 \draw[thick] (-2,-2)--(-1,-1)--(0,-2);
 
  \fill (-1,-1)  circle[radius=1.5pt];

\node at (0,1) {$i$};
\node at (2,1) {$j$};
\node at (1,0.2) {$k$};
\node at (0,-.75) {$j$};
\node at (-1,-1.75) {$k$};

\node[red] at (.6,1) {$\alpha$};
\node[red] at (-1.75,-1) {$\alpha 00$};


\end{tikzpicture}
\qquad 
\begin{tikzpicture}[x=.75cm, y=.75cm,
    every edge/.style={
        draw,
      postaction={decorate,
                    decoration={markings}
                   }
        }
]
\draw[thick] (0,0) -- (1,1)--(2,0);
 \draw[thick] (1,1) -- (1,2);
 \draw[thick] (1,-1)--(2,0)--(3,-1);
  \draw[thick] (2,-2)--(3,-1)--(4,-2);

  \fill (3,-1)  circle[radius=1.5pt];

\node[red] at (.6,1) {$\alpha$};
\node[red] at (3.75,-1) {$\alpha 11$};

\node at (0,1) {$i$};
\node at (2,1) {$j$};
\node at (1,0.2) {$k$};
\node at (2,-.75) {$i$};
\node at (3,-1.75) {$k$};


\end{tikzpicture}
\]
Formula \eqref{formula4} is obvious.
\end{proof}

\begin{lemma}\label{lemma2}
For any $n\in\IN$, it holds 
\begin{align}\label{formula1}
& \omega((10)^n)=\left\{ \begin{array}{cc} 
0 & \text{if } n\equiv_3 0\\
2 & \text{if }  n\equiv_3 1\\
1 & \text{if }  n\equiv_3 2
\end{array}
\right. 
\qquad 
\omega((01)^n)=\left\{ \begin{array}{cc} 
0 & \text{if } n\equiv_3 0\\
1 & \text{if }  n\equiv_3 1\\
2 & \text{if }  n\equiv_3 2
\end{array}
\right. 
\end{align}
where, for a finite word $w$ in $0$ and $1$, $w^n$ denotes the word obtained by concatenating $n$ copies of $w$, and the symbol $\equiv_3$ denotes the equivalence modulo $3$.
\end{lemma}
\begin{proof}
We begin with the formula for $\omega((10)^n)$. 
The next figure shows that the region to the left of the marked vertices are $2$, $1$, $0$ for the cases $n=1$, $2$, $3$, respectively.
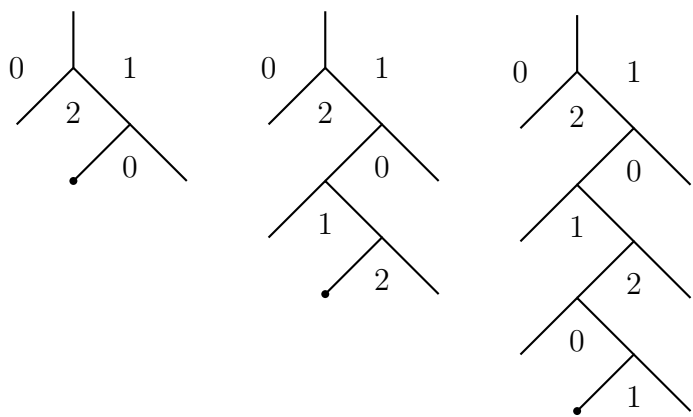
\[
\begin{tikzpicture}[x=.75cm, y=.75cm,
    every edge/.style={
        draw,
      postaction={decorate,
                    decoration={markings}
                   }
        }
]
\draw[thick] (0,0) -- (1,1)--(2,0);
 \draw[thick] (1,1) -- (1,2);
 \draw[thick] (1,-1)--(2,0)--(3,-1);
 
  \fill (1,-1)  circle[radius=1.5pt];

\node at (0,1) {$0$};
\node at (2,1) {$1$};
\node at (1,0.2) {$2$};
\node at (2,-.75) {$0$};

\node at (1,-5) {$\;$};

\end{tikzpicture}
\qquad
\begin{tikzpicture}[x=.75cm, y=.75cm,
    every edge/.style={
        draw,
      postaction={decorate,
                    decoration={markings}
                   }
        }
]
\draw[thick] (0,0) -- (1,1)--(2,0);
 \draw[thick] (1,1) -- (1,2);
 \draw[thick] (1,-1)--(2,0)--(3,-1);
 \draw[thick] (0,-2)--(1,-1)--(2,-2);
 \draw[thick] (1,-3)--(2,-2)--(3,-3);
 
  \fill (1,-3)  circle[radius=1.5pt];

\node at (0,1) {$0$};
\node at (2,1) {$1$};
\node at (1,0.2) {$2$};
\node at (2,-.75) {$0$};
\node at (1,-1.75) {$1$};
\node at (2,-2.75) {$2$};

\node at (1,-5) {$\;$};

\end{tikzpicture}
\qquad
\begin{tikzpicture}[x=.75cm, y=.75cm,
    every edge/.style={
        draw,
      postaction={decorate,
                    decoration={markings}
                   }
        }
]
\draw[thick] (0,0) -- (1,1)--(2,0);
 \draw[thick] (1,1) -- (1,2);
 \draw[thick] (1,-1)--(2,0)--(3,-1);
 \draw[thick] (0,-2)--(1,-1)--(2,-2);
 \draw[thick] (1,-3)--(2,-2)--(3,-3);
  \draw[thick] (0,-4)--(1,-3)--(2,-4);
  \draw[thick] (1,-5)--(2,-4)--(3,-5);

  \fill (1,-5)  circle[radius=1.5pt];

\node at (0,1) {$0$};
\node at (2,1) {$1$};
\node at (1,0.2) {$2$};
\node at (2,-.75) {$0$};
\node at (1,-1.75) {$1$};
\node at (2,-2.75) {$2$};
\node at (1,-3.75) {$0$};
\node at (2,-4.75) {$1$};

\end{tikzpicture}
\]
Since the regions of the regions surrounding the vertex $101010$ have the same colours as those near the root, the former cases suffice to prove the formula.

We now take care of the formula for $\omega((01)^n)$. The next figure displays the cases $n=1, 2, 3$ which are enough to prove our claim.
\[
\begin{tikzpicture}[x=.75cm, y=.75cm,
    every edge/.style={
        draw,
      postaction={decorate,
                    decoration={markings}
                   }
        }
]
\draw[thick] (0,0) -- (1,1)--(2,0);
 \draw[thick] (1,1) -- (1,2);
 \draw[thick] (-1,-1)--(0,0)--(1,-1);
 
  \fill (1,-1)  circle[radius=1.5pt];

\node at (0,1) {$0$};
\node at (2,1) {$1$};
\node at (1,0.2) {$2$};
\node at (0,-.75) {$1$};

\node at (1,-5) {$\;$};

\end{tikzpicture}
\qquad
\begin{tikzpicture}[x=.75cm, y=.75cm,
    every edge/.style={
        draw,
      postaction={decorate,
                    decoration={markings}
                   }
        }
]
\draw[thick] (0,0) -- (1,1)--(2,0);
 \draw[thick] (1,1) -- (1,2);
 \draw[thick] (-1,-1)--(0,0)--(1,-1);
\draw[thick] (0,-2) -- (1,-1)--(2,-2);
 \draw[thick] (-1,-3)--(0,-2)--(1,-3);
 
  \fill (1,-3)  circle[radius=1.5pt];

\node at (0,1) {$0$};
\node at (2,1) {$1$};
\node at (1,0.2) {$2$};
\node at (0,-.75) {$1$};
\node at (1,-1.75) {$0$};
\node at (0,-2.75) {$2$};

\node at (1,-5) {$\;$};

\end{tikzpicture}
\qquad
\begin{tikzpicture}[x=.75cm, y=.75cm,
    every edge/.style={
        draw,
      postaction={decorate,
                    decoration={markings}
                   }
        }
]
\draw[thick] (0,0) -- (1,1)--(2,0);
 \draw[thick] (1,1) -- (1,2);
 \draw[thick] (-1,-1)--(0,0)--(1,-1);
\draw[thick] (0,-2) -- (1,-1)--(2,-2);
 \draw[thick] (-1,-3)--(0,-2)--(1,-3);
\draw[thick] (0,-4) -- (1,-3)--(2,-4);
 \draw[thick] (-1,-5)--(0,-4)--(1,-5);
 
  \fill (1,-5)  circle[radius=1.5pt];

\node at (0,1) {$0$};
\node at (2,1) {$1$};
\node at (1,0.2) {$2$};
\node at (0,-.75) {$1$};
\node at (1,-1.75) {$0$};
\node at (0,-2.75) {$2$};
\node at (1,-3.75) {$1$};
\node at (0,-4.75) {$0$};

\node at (1,-5) {$\;$};

\end{tikzpicture}
\]

\end{proof}

\begin{remark}
The properties \eqref{formula2}, \eqref{formula3}, \eqref{formula4}, \eqref{formula1}  completely determine the function $\omega: W_2 \to\IZ_3$.  
\end{remark}
The next result provides a simple formula for the weight $\omega$. 
\begin{proposition}\label{prop-form-weight}
For a binary word $a_1a_2\ldots a_n$  it holds 
$$\omega(a_1a_2\ldots a_n)\equiv_3 \sum_{i=1}^n (-1)^ia_i \, .$$
\end{proposition}
\begin{proof}
We define the function $f: W_2\to \IZ_3$ by the formula $f(a_1\ldots n_n):=\sum_{i=1}^n (-1)^ia_i$. It is easy to check that $f$ satisfies the properties of Lemmas \ref{lemma1} and \ref{lemma2}. Moreover, $f(1)=-1\equiv_3 2=\omega(1)$, $f(101)=-2\equiv_3 1=\omega(101)$, $f(10101)=-3\equiv_3 0=\omega(10101)$, $f(01)=1=\omega(01)$,  $f(0101)=2=\omega(0101)$,   $f(010101)=3\equiv_3 0=\omega(010101)$ and so the functions $f$ and $\omega$ coincide.
\end{proof}
Note that the symbols $0$ and $1$ appearing in the formula for $\omega(\cdot)$ have
several meanings, namely those in the word $a_1a_2\ldots a_n$ are related to left/right edges, while those in the output of $\omega$ are interpreted as elements of $\IZ_3$ and denote the color of the  regions.
\begin{lemma}
Let $(T_+,T_-)$ be an element in $\CF$. Then,
for any leaf, the parity of the length of the paths to the roots is the same.
\end{lemma}
\begin{proof}
The claim is true for the generators of $\CF$, see Figure \ref{genCF}. 
By using the multiplication algorithm described in 
\cite{B}, one can easily see that this property is preserved by multiplication (it suffices to check that the 
the addition/deletion of carets preserves the property) and we are done.
\end{proof}

There exists a  bijection $\rho$ between 
the finite sequences of $0$ and $1$ and
the dyadic rationals in the open unit interval, namely the map $\rho(a_1\ldots a_n):=\sum_{i=1}^n a_i 2^{-i}$, where $.a_1\ldots a_n$ is word in $W_2$.
For $i\in\IZ_3$ consider the subsets $S_i$ of the dyadic rationals consisting of the numbers whose corresponding binary word has 
weight $i$, namely
$$
S_i:=\{t\in (0,1)\cap \IZ[1/2]\; | \; \omega(t)= i\} \qquad i\in \IZ_3\; .
$$
\begin{lemma}\label{lemma-inclu}
The $3$-colorable subgroup $\CF$ coincides with $\cap_{i\in\IZ_3}{\rm Stab}(S_i)$.
\end{lemma}
\begin{proof}
The inclusion $\CF\subset\cap_{i\in\IZ_3}{\rm Stab}(S_i)$ follows easily by checking that the generators of $\CF$, namely 
$w_0:=x_0^2x_1x_2^{-1}$,
$w_1:=x_0x_1^2x_0^{-1}$,
$w_2:=x_1^2x_3x_2^{-1}$,
$w_3:=x_2^2x_3x_4^{-1}$ 
 preserve the sets $S_i$ for all $i\in\IZ_3$.
For the converse inclusion, let $f=(T_+,T_-)$ be an element of $\cap_{i\in\IZ_3}\stab(S_i)$.
Denote by  $\alpha_+(k)$ and $\alpha_-(k)$ the words associated with the $k$-th leaf of the top and bottom trees, respectively.
Since $f\in \cap_{i\in\IZ_3}\stab(S_i)$, it follows that $\omega_+(\alpha_+(k))\equiv_3 \omega_-(\alpha_-(k))$ for all $k$. 
By Proposition \ref{prop-descr-3col} we have that $f\in\CF$.
\end{proof}
\begin{remark}
We observe that the  previous lemma implies that $\CF\leq K_{(1,2)}=\{f\in F\; | \; \log_2f'(1)\in 2\IZ\}$. 
Indeed, given $f=(T_+,T_-)\in\CF$, let $\alpha_+$ and $\alpha_-$ be the words corresponding to the right most leaves of $T_+$ and $T_-$ (with $k_\pm=|\alpha_\pm|$).
These words do not contain any occurrences of the letter $0$. Therefore, $\omega(\alpha_\pm)\in\{-1,0\}$ (as integers). More precisely, 
$\omega(\alpha_\pm)=-1$ if and only if $k_\pm \in 2\IZ +1$. Similarly, $\omega(\alpha_\pm)=0$ if and only if $k_\pm \in 2\IZ$. 
Since $\log_2 f'(1)=k_+-k_-$, this implies that $\log_2f'(1)\in 2\IZ$.
\end{remark} 
\begin{remark}
All three stabilizers are distinct from one another.
Indeed, we have
$x_0x_1\in {\rm Stab}(S_0)\setminus ({\rm Stab} (S_1)\cup  {\rm Stab}(S_2))$,
$x_1x_0^{-2}\in {\rm Stab}(S_1)\setminus ({\rm Stab} (S_0)\cup  {\rm Stab}(S_2))$,
$x_0x_2(x_0^2x_1)^{-1}\in {\rm Stab}(S_2)\setminus ({\rm Stab} (S_0)\cup  {\rm Stab}(S_1))$.
\end{remark}
As a consequence of the previous lemma, we obtain the following theorem. 
\begin{theorem}\label{theo2}
The $3$-colorable subgroup coincides with its commensurator.
\end{theorem}
\begin{corollary}\label{cor-irred}
The quasi-regular representation of $F$ associated with $\CF$ is irreducible.
\end{corollary}
The irreducibility of the quasi-regular representation associated with $\CF$ follows from \cite{Ma}.
The proof of the Theorem is similar to that for the oriented subgroup $\vec{F}$ done by Golan and Sapir    \cite[Theorem 4.15]{GS}.

We now recall  the natural action of an element $f\in F$ on the numbers in $[0,1]$
expressed in binary expansion. 
A number $t$ enters into the top of the tree diagram, follows a path towards the root of the bottom tree according to the rules portrayed in 
Figure \ref{compute}. What emerges at the bottom is the image of $t$ under the homeomorphism represented by the tree diagram, \cite{BM}. 
Note that there is a change of direction only when the number comes across a vertex of degree $3$ (i.e., the number is unchanged when it comes across a leaf). 

\begin{figure}
\phantom{This text will be invisible} 
 \[
 \begin{tikzpicture}[x=1cm, y=1cm,
    every edge/.style={
        draw,
      postaction={decorate,
                    decoration={markings}
                   }
        }
]

\draw[thick] (0,0) -- (.5,1)--(1,0);
 \draw[thick] (.5,1)--(.5,1.5);

\fill (0,0)  circle[radius=.75pt];
 \fill (1,0)  circle[radius=.75pt];
\fill (.5,1.5)  circle[radius=.75pt];

\node at (.5,1.85) {$\scalebox{1}{$.0\alpha$}$};
\node at (0,-.25) {$\scalebox{1}{$.\alpha$}$};

\end{tikzpicture}
\qquad\qquad
\begin{tikzpicture}[x=1cm, y=1cm,
    every edge/.style={
        draw,
      postaction={decorate,
                    decoration={markings}
                   }
        }
]

\draw[thick] (0,0) -- (.5,1)--(1,0);
 \draw[thick] (.5,1)--(.5,1.5);

\fill (0,0)  circle[radius=.75pt];
 \fill (1,0)  circle[radius=.75pt];
\fill (.5,1.5)  circle[radius=.75pt];

\node at (.5,1.85) {$\scalebox{1}{$.1\alpha$}$};
\node at (1,-.25) {$\scalebox{1}{$.\alpha$}$};

\end{tikzpicture}
\qquad
\qquad
\begin{tikzpicture}[x=1cm, y=1cm,
    every edge/.style={
        draw,
      postaction={decorate,
                    decoration={markings}
                   }
        }
]

\draw[thick] (0,1.5) -- (.5,.5)--(1,1.5);
 \draw[thick] (.5,0)--(.5,.5);

\fill (0,1.5)  circle[radius=.75pt];
\fill (.5,0)  circle[radius=.75pt];
\fill (1,1.5)  circle[radius=.75pt];
 
\node at (.5,-.25) {$\scalebox{1}{$.0\alpha$}$};
\node at (0,1.75) {$\scalebox{1}{$.\alpha$}$};
\node at (1,-.25) {$\scalebox{1}{}$};

\end{tikzpicture}
\qquad\qquad
\begin{tikzpicture}[x=1cm, y=1cm,
    every edge/.style={
        draw,
      postaction={decorate,
                    decoration={markings}
                   }
        }
]

\draw[thick] (0,1.5) -- (.5,.5)--(1,1.5);
 \draw[thick] (.5,0)--(.5,.5);

\fill (0,1.5)  circle[radius=.75pt];
\fill (.5,0)  circle[radius=.75pt];
\fill (1,1.5)  circle[radius=.75pt];
 
\node at (.5,-.25) {$\scalebox{1}{$.1\alpha$}$};
\node at (1,1.75) {$\scalebox{1}{$.\alpha$}$};
\node at (1,-.25) {$\scalebox{1}{}$};

\end{tikzpicture}
\]
\caption{The local rules for computing the action of $F$ on numbers expressed in binary expansion.}\label{compute}
\end{figure}
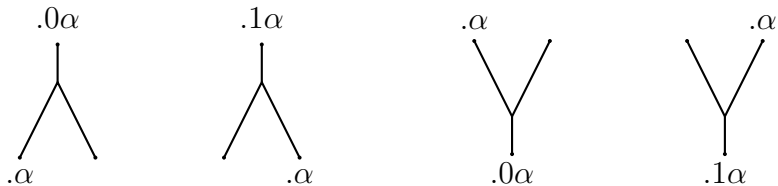

The next result is analogous to \cite[Lemma 4.14]{GS}.
\begin{lemma}\label{lemma3}
Let $g\in F$. Then there exists $m\in\IN$ such that
\begin{enumerate} 
\item if $\log_2 g'(0)\in 2\IZ$, for any $t\in (0,1/2^m)\cap S_i$ and any $i\in \IZ_3$, $g(t)\in S_i$;
\item if $\log_2 g'(0)\in 2\IZ +1$, for any $t\in (0,1/2^m)\cap S_0$, $g(t)\in S_0$;
\item if $\log_2 g'(0)\in 2\IZ +1$, for any $t\in (0,1/2^m)\cap S_1$, $g(t)\in S_2$; 
\item if $\log_2 g'(0)\in 2\IZ +1$, for any $t\in (0,1/2^m)\cap S_2$, $g(t)\in S_1$.
\end{enumerate}
\end{lemma}
\begin{proof}
By definition, $g(t)=2^l t$ for all $t\in I=[0,2^{-r}]$, where $r\in\IN$, $l\in\IZ$.
If $l\leq 0$ $g(t)$ simply adds $l$ zeros to the beginning of the binary form of $t$. Therefore, we have two cases depending on whether $\log_2 g'(0)$ is in $2\IZ$ or  $2\IZ +1$. In the first case $\omega(t)=\omega(g(t))$, while in the second $\omega(t)\equiv_3 -\omega(g(t))$. In both cases, it suffices to take $m=r$.

If $l>0$, take $m=\max\{r,l\}$. Since $m\geq r$, the binary word for $t$ begins with at least $l$ zeroes and $g$ erases $l$ of them.  We have two cases to deal with: $\log_2 g'(0)$ is in $2\IZ$ or in  $2\IZ +1$. In the first case $\omega(t)=\omega(g(t))$, while in the second $\omega(t)\equiv_3 -\omega(g(t))$. 
\end{proof} 

\begin{proof}[Proof of Theorem \ref{theo2}]
Let $h\in F\setminus\mathcal{F}$ and set $I:=|\CF:\CF\cap h\CF h^{-1}|$. If $I<\infty$, then there is an $r\in \IN$ such that $w_0^{-r}\in h\CF h^{-1}$, or equivalently $h^{-1}w_0^{-r}h\in\CF$ (here $w_0$ is one of the generators of $\CF$ depicted in Figure \ref{genCF}).
We will show that for all $n$ big enough, $h^{-1}w_0^{-n}h\not\in\CF$ (and thus reach a contradiction). 

If   $\log_2 h'(0)\in 2\IZ$, by Lemma \ref{lemma3}-(1), there is an $m$ such $\omega(h(t))\equiv_3\omega(t)$ for all $t\in [0,2^{-m}]\cap S_i$.
Since $h\not\in\CF$, there exists $t\in S_i$ (for some $i\in \IZ_3$) such that $t_1:=h^{-1}(t)\not\in S_i$. We observe that for all $l\in\IN$, we have $w_0^{-l}(t_1)\not\in S_i$.
There exists an $n\in\IN$ such that $w_0^{-n}(t_1)<2^{-m}$. 
Indeed, 
$$
w_0^{-1}(t)=\left\{ 
\begin{array}{ll}
.000\alpha & \text{ if } t=.0\alpha\\
.0010\alpha & \text{ if } t=.10\alpha\\
.0011\alpha & \text{ if } t=.1100\alpha\\
.01\alpha & \text{ if } t=.1101\alpha\\
.1\alpha & \text{ if } t=.111\alpha\\
\end{array}
\right.
$$

We observe that $h(w_0^{-n}(t_1))\not\in S_i$  by Lemma \ref{lemma3}-(1). Then, $h^{-1}w_0^{-n}h(t)=h(w_0^{-n}(t_1))\not\in S_i$ and so $h^{-1}w_0^{-n}h\not\in {\rm Stab}(S_i)$.

If $\log_2 h'(0)\in 2\IZ+1$ and $\log_2 h'(1)\in 2\IZ$, then we may apply the automorphism $\sigma$, which by Lemma \ref{lemmasigma}  preserves both $\CF$  and the index (that is, $|\CF:\CF\cap h\CF h^{-1}|$ is finite if and only if $|\CF:\CF\cap \sigma(h)\CF \sigma(h)^{-1}|$). 
Therefore, we may assume that $\log_2 h'(0)$ and $\log_2 h'(1)$ have the same parity (otherwise, replace $h$ with $\sigma(h)$ and repeat the argument of the previous paragraph).

We are left to consider the case where $\log_2 h'(0), \log_2 h'(1) \in 2\IZ+1$. We claim that there is a $t\in S_0$ such that $t_1:=h^{-1}(t)\in S_2$.
By definition, there exist a standard dyadic partitions of $[0,1]$ such that, on the standard dyadic interval of the domain containing $1$, the function $h^{-1}$ reads as $h^{-1}(.1_p \alpha)=.1_q\alpha$, for any binary word $\alpha$. Set $t:=.1_p \alpha$ and $t_1:=h^{-1}(.1_p\alpha )=.1_q\alpha$.
If $p\in 2\IN_0+1$ and $q\in 2\IN_0$, we have $\omega(t_1)=\omega(h^{-1}(t))=\omega(h^{-1}(.1_p\alpha ))=\omega(.1_q \alpha)=\omega(.\alpha)$,
$\omega(t)=\omega(.1_p \alpha)=-1-\omega(.\alpha)$, then $\omega(t_1)=-1-\omega(t)$. In particular, if $\omega(t)=0$, we have $\omega(t_1)=2$.

If $p\in 2\IN_0$ and $q\in 2\IN_0+1$, we have $\omega(t_1)=\omega(h^{-1}(t))=\omega(h^{-1}(1_p\alpha ))=\omega(.1_q \alpha)=-1-\omega(.\alpha)$,
$\omega(t)=\omega(.1_p \alpha)=\omega(.\alpha)$, then $\omega(t_1)=-1-\omega(t)$. In particular, if $\omega(t)=0$, we have $\omega(t_1)=2$.

Therefore, we may assume that  $\log_2 h'(0), \log_2 h'(1) \in 2\IZ+1$ and that there is a $t\in S_0$ such that $t_1:=h^{-1}(t)\in S_2$.
We have $h(w_0^{-n}(t_1))\not\in S_0$  by Lemma \ref{lemma3}-(3) and $h^{-1}w_0^{-n}h(t)=h(w_0^{-n}(t_1))\not\in S_0$, so $h^{-1}w_0^{-n}h\not\in {\rm Stab}(S_0)$ and we are done.
%
%
%
%
\end{proof} 
In general, Jones's representations are not well understood.
The previous result joins a series of investigations \cite{GS, ABC, Jo19, AJ, TV}, where suitable families of representations were studied.

\section{The rectangular subgroup $K_{(2,2)}$} \label{sec4}
The rectangular subgroups of $F$ were introduced in \cite{BW} as
\begin{align*}
K_{(a,b)}&:=\{f\in F\; | \; \log_2f'(0)\in a\IZ, \log_2f'(1)\in b\mathbb{Z}\} \qquad a, b\in\IN
\end{align*}
These are the only finite index subgroups isomorphic with $F$ \cite[Theorem 1.1]{BW}.

As mentioned in the Introduction, the 3-colorable subgroup $\CF$ sits inside the rectangular subgroup $K_{(2,2)}$.
In this section we are going to define an explicit isomorphism between $F$ and $K_{(2,2)}$ that will on one hand, provide us with a 
pair of elements generating 
$K_{(2,2)}$, 
and on the other hand, will later help to identify new maximal subgroup in $F$ by finding the maximal subgroups in $K_{(2,2)}$ that contain $\CF$.
First we recall from \cite{GS2} that the subgroup $K_{(1,2)}$ is generated by $x_0x_2$ and $x_1x_2$ (see Figure \ref{genK12}).
Moreover, the following map is an isomorphism
\begin{align*}
 \beta: \; & F\to K_{(1,2)}\\
& x_0\mapsto x_0x_2\\
& x_1\mapsto x_1x_2
\end{align*}
Recall that there is an order $2$ automorphism $\sigma: F\to F$ obtained by reflecting tree diagrams about a vertical line.
Clearly, $\sigma(K_{(1,2)})=K_{(2,1)}$ and, therefore, the subgroup $K_{(2,1)}$ is generated by $\sigma(x_0x_2)=x_0x_1x_0^{-3}$ and $\sigma(x_1x_2)=x_0x_1^2x_0^{-3}$ (see Figure \ref{genK21}). We define the following isomorphism
\begin{align*}
\alpha: \; & F\to K_{(2,1)}\\
& x_0\mapsto x_0x_1x_0^{-3}\\
& x_1\mapsto x_0x_1^2x_0^{-3}
\end{align*}
The isomorphism $\beta$ first appeared in \cite{GS2}, though we use a different notation from Golan and Sapir: their map $\Psi$ coincides with our map $\beta$.

\begin{figure}
\phantom{This text will be invisible} 
\[
\begin{tikzpicture}[x=.75cm, y=.75cm,
    every edge/.style={
        draw,
      postaction={decorate,
                    decoration={markings}
                   }
        }
]

\node at (-.5,-.25) {$\scalebox{1}{$x_0x_2=$}$};

\draw[thick] (1,0)--(3,2)--(5,0)--(3,-2)--(1,0);
\draw[thick] (3,2)--(3,2.25);
\draw[thick] (3,-2)--(3,-2.25);

\draw[thick] (1.5,.5)--(3.5,-1.5);
\draw[thick] (4,1)--(3,0)--(4,-1);
\draw[thick] (3.5,.5)--(4.5,-.5);

\node at (0,-1.2) {$\;$};
\end{tikzpicture}
\qquad 
\begin{tikzpicture}[x=.75cm, y=.75cm,
    every edge/.style={
        draw,
      postaction={decorate,
                    decoration={markings}
                   }
        }
]

\node at (-.5,-.25) {$\scalebox{1}{$x_1x_2=$}$};

\draw[thick] (1,0)--(3,2)--(5,0)--(3,-2)--(1,0);
\draw[thick] (3,2)--(3,2.25);
\draw[thick] (3,-2)--(3,-2.25);

\draw[thick] (3.5,1.5)--(2,0)--(3.5,-1.5);
\draw[thick] (3,1)--(4.5,-.5);
\draw[thick] (3.5,.5)--(3,0)--(4,-1);

\node at (0,-1.2) {$\;$};
\end{tikzpicture}
\]
 \caption{The generators of $K_{(1,2)}$.}\label{genK12}
\end{figure}
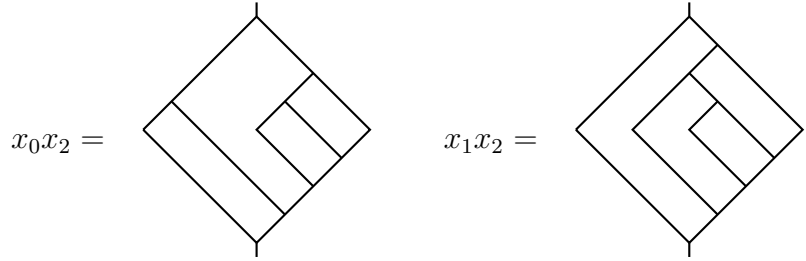

\begin{figure}
\phantom{This text will be invisible} 
\[
\begin{tikzpicture}[x=.75cm, y=.75cm,
    every edge/.style={
        draw,
      postaction={decorate,
                    decoration={markings}
                   }
        }
]

\node at (-.5,-.25) {$\scalebox{1}{$\sigma(x_0x_2)=$}$};

\draw[thick] (1,0)--(3,2)--(5,0)--(3,-2)--(1,0);
\draw[thick] (3,2)--(3,2.25);
\draw[thick] (3,-2)--(3,-2.25);

\draw[thick] (2.5,-1.5)--(4.5,.5);
\draw[thick] (2,1)--(3,0)--(2,-1);
\draw[thick] (1.5,-.5)--(2.5,.5);

\node at (0,-1.2) {$\;$};
\end{tikzpicture}
\qquad 
\begin{tikzpicture}[x=.75cm, y=.75cm,
    every edge/.style={
        draw,
      postaction={decorate,
                    decoration={markings}
                   }
        }
]

\node at (-.5,-.25) {$\scalebox{1}{$\sigma(x_1x_2)=$}$};

\draw[thick] (1,0)--(3,2)--(5,0)--(3,-2)--(1,0);
\draw[thick] (3,2)--(3,2.25);
\draw[thick] (3,-2)--(3,-2.25);

\draw[thick] (2.5,1.5)--(4,0)--(2.5,-1.5);
\draw[thick] (1.5,-.5)--(3,1);
\draw[thick] (2.5,.5)--(3,0)--(2,-1);

\node at (0,-1.2) {$\;$};
\end{tikzpicture}
\]
 \caption{The generators of $K_{(2,1)}$.}\label{genK21}
\end{figure}
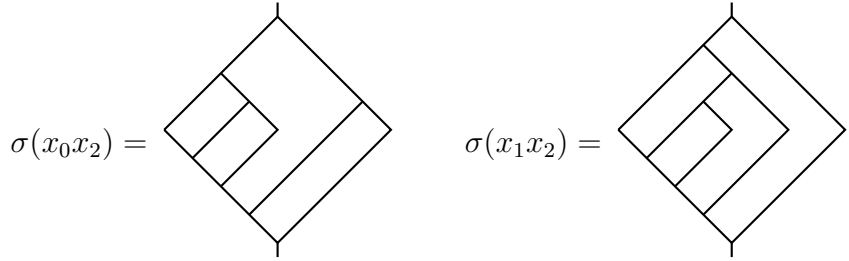

In the following result we give an explicit isomorphism between $F$ and $K_{(2,2)}$ which will allow us, in Section 4, to 
obtain maximal subgroups of $F$ from those of $K_{(2,2)}$.
\begin{proposition}
Let $\theta: F\to F$ be the monomorphism 
$\beta\circ \alpha$.  
Then, the image of $\theta$ is  $K_{(2,2)}$.
\end{proposition}
\begin{proof}
Easy computations yield the following formulas for $\theta(x_0)$ and $\theta(x_1)$
\begin{align*}
\theta(x_0) &= \beta(x_0x_1x_0^{-3})=x_0x_2x_1x_2x_2^{-1}x_0^{-1}x_2^{-1}x_0^{-1}x_2^{-1}x_0^{-1}\\
&=x_0x_2x_1x_0^{-1}x_2^{-1}x_0^{-1}x_2^{-1}x_0^{-1}=x_0x_1x_3x_0^{-1}x_2^{-1}x_3^{-1}x_0^{-2}\\
&=x_0x_1x_0^{-1}x_2x_0x_0^{-1}x_2^{-1}x_3^{-1}x_0^{-2}=x_0x_1x_0^{-1}x_3^{-1}x_0^{-2}=x_0x_1x_4^{-1}x_0^{-3}\\
&= x_0x_1x_0^{-3}x_1^{-1}=\alpha(x_0)x_1^{-1}\\
\theta(x_1) &= \beta(x_0x_1^2x_0^{-3})=x_0x_2x_1x_2x_1x_2x_2^{-1}x_0^{-1}x_2^{-1}x_0^{-1}x_2^{-1}x_0^{-1}\\
&= x_0x_2x_1x_2x_1 x_0^{-1}x_2^{-1}x_0^{-1}x_2^{-1}x_0^{-1}\\
&= x_0x_1x_3x_1x_3 x_0^{-1}x_2^{-1}x_0^{-1}x_2^{-1}x_0^{-1}\\
&= x_0x_1^2x_4x_3 x_0^{-1}x_2^{-1}x_0^{-1}x_2^{-1}x_0^{-1}\\
&= x_0x_1^2x_4(x_0^{-1}x_2x_0) x_0^{-1}x_2^{-1}x_0^{-1}x_2^{-1}x_0^{-1}\\
&= x_0x_1^2x_4x_0^{-1}x_0^{-1}x_2^{-1}x_0^{-1}\\
&= x_0x_1^2x_0^{-1}x_0^{-1}x_2x_2^{-1}x_0^{-1}\\
&= x_0x_1^2x_0^{-3}=\alpha(x_1)
\end{align*}
In particular, we have $\theta(F)\leq K_{(2,2)}$. We now want to show that these two subgroups actually coincide.

First, we observe that $|F: \theta(F)|=4$. Indeed, $|F:K_{(1,2)}|=|F:K_{(2,1)}|=|K_{(1,2)}:K_{(2,2)}|=2$. In particular,   $F=K_{(1,2)}\cup x_0 K_{(1,2)}=K_{(2,1)}\cup x_0^{-1}K_{(2,1)}$, $K_{(1,2)}=K_{(2,2)}\cup gK_{(2,2)}$ for some $g\in K_{(1,2)}$. 
It holds $F= \alpha(F)\cup x_0^{-1} \alpha(F)$ and $K_{(1,2)}=\beta(F)= \beta(\alpha(F))\cup \beta(x_0^{-1}) \beta(\alpha(F))= \theta(F)\cup \beta(x_0^{-1}) \theta(F)$. Then
\begin{align*}
F&= K_{(1,2)}\cup x_0 K_{(1,2)} =  \theta(F)\cup \beta(x_0^{-1})\theta(F)\cup x_0\theta(F)\cup x_0\beta(x_0^{-1})\theta(F)
\end{align*}
 
 Since $|F: K_{(2,2)}|=|F:\theta(F)|=4$ and $\theta(F)\leq K_{(2,2)}$, we have that $\theta(F)=K_{(2,2)}$.
 \end{proof}
 \begin{proposition}
 The rectangular subgroup $K_{(2,2)}$ is generated by $x_0x_1x_4^{-1}x_0^{-3}= x_0x_1x_0^{-3}x_1^{-1}$ and $x_0x_1^2x_0^{-3}$. Moreover, its elements have  normal form of even length and  an even number of occurrences of $x_0^{\pm 1}$.
 \end{proposition}
 \begin{remark}
 Since $\sigma(K_{(2,2)})=K_{(2,2)}$, the elements $x_0^3x_4x_1^{-1}x_0^{-1}=\sigma(x_0x_1x_0^{-3}x_1^{-1})$ and $x_1x_2=\sigma(x_0x_1^2x_0^{-3})$ generate $K_{(2,2)}$ as well
 \begin{align*}
 \sigma(x_0)&=x_0^{-1}\\
 \sigma(x_1)&=x_0x_1x_0^{-2}\\
 \sigma(x_0x_1x_0^{-3}x_1^{-1})&=x_0^{-1}x_0x_1x_0^{-2}x_0^3x_0^2x_1^{-1}x_0^{-1}\\
 &=x_1x_0^3x_1^{-1}x_0^{-1}=x_0^3x_4x_1^{-1}x_0^{-1}\\
 \sigma(x_0x_1^2x_0^{-3})&=x_0^{-1}x_0x_1x_0^{-2}x_0x_1x_0^{-2}x_0^3\\
 &=x_1x_0^{-1}x_1x_0=x_1x_2\\
 \end{align*}
 \end{remark}
 Therefore
 \begin{align*}
 K_{(2,2)}& =\langle x_0x_1x_0^{-3}x_1^{-1}, x_0x_1^2x_0^{-3}\rangle\\
&=\langle x_0x_1x_0^{-3}x_1^{-1}, x_1x_2\rangle
 \end{align*}
 
\section{Maximal infinite index subgroups of $F$ containing an isomorphic image of $\CF$} \label{sec4b}
We have already mentioned in Section \ref{sec1} that the $3$-colorable subgroup $\CF$ is generated by 
$w_0=x_0^2x_1x_2^{-1}$,
$w_1=x_0x_1^2x_0^{-1}$,
$w_2=x_1^2x_3x_2^{-1}$,
$w_3=x_2^2x_3x_4^{-1}$,
 \cite{Ren}.
See Figure \ref{genCF} for the tree diagrams of the generators of $\CF$.

\begin{proposition}\label{propKCF}
The group $\CF$ sits inside 
$$
K_{(2,2)}=\{f\in F\; | \; \log_2f'(0), \log_2f'(1)\in 2\mathbb{Z}\}\cong F \, .
$$
Moreover, it holds $\pi(\CF)=2\IZ\oplus 2\IZ$.
\end{proposition}
\begin{proof}
In order to prove the claim it suffices to compute the images of
the generators of $\CF$ under the map $\pi$ and 
this can be done as explained in the Introduction. 
Thanks to Figure \ref{genCF}, we see that
$\pi(w_0)=(2,-2)$, $\pi(w_1)=(0,-2)$, $\pi(w_2)=(0,-2)$, $\pi(w_3)=(0,-2)$.
\end{proof}
Observe that $\CF$ is a proper subgroup of $K_{(2,2)}$ as, for example, $K_{(2,2)}$ is isomorphic with $F$, whereas $\CF$ is isomorphic with $F_4$.
This can also be shown by checking that $x_0^2$ is in $K_{(2,2)}$, but not in $\CF$.
In \cite[Section 3.2]{GS2} Golan and Sapir pointed out that every proper subgroup of $F$ that projects onto $F/[F,F]$ is contained in a maximal infinite index  
subgroup of $F$. Let us implement this idea on the example of $\CF$
and investigate in which infinite index maximal subgroups of $F$ 
the $3$-colorable subgroup is contained, see Corollary \ref{cor41}, Corollary \ref{cor42}, Theorem \ref{teo43}.

Recall that every finite-index subgroup contains a normal subgroup. 
Since  every normal subgroup of $F$ contains   
contains the commutator subgroup \cite[Theorem 4.3]{CFP},
the map $\pi: F\to \IZ\oplus \IZ$ induces a bijective correspondence between the finite index subgroups of $F$ and the finite index subgroups $\IZ\oplus \IZ$. 
As mentioned in Section \ref{sec4}, the only finite index subgroups isomorphic with $F$ are precisely $K_{(a,b)}=\pi^{-1}(a\IZ\oplus b\IZ)$, \cite[Theorem 1.1]{BW}.
Here are some easy consequences of this discussion. 
\begin{lemma}\label{Lemmamax}
Let $G$ be a subgroup of $F$ such that $\pi(G)=a\IZ\oplus b\IZ$ and $G\neq K_{(a,b)}$. Then, the index of $G$ in $F$ is infinite. 
\end{lemma}
\begin{corollary}
The index of $\CF$ in $F$ is infinite. 
\end{corollary}
\begin{proof}
The claim follows from the fact that, thanks to Proposition \ref{propKCF}, $\pi(K_{(2,2)})=\pi(\CF)$ and $K_{(2,2)}\neq \CF$ . 
The two groups are distinct because $\CF\cong F_4$, while $K_{(2,2)}\cong F$. 
\end{proof}
\begin{remark}
There is also another way to prove the previous result. All the irreducible finite dimensional representations of $F$ are one dimensional (this follows from \cite{DM}). This means that if $\CF$ were of finite index in $F$, then
for the quasi-regular representation of $F$ associated with $\CF$ we would have $\ell_2(F/\CF)=\IC$, that is $F=\CF$, which is absurd.  
\end{remark}
\begin{corollary}
There exists a maximal infinite index subgroup $M$ of $K_{(2,2)}$  containing $\CF$. 
\end{corollary}
\begin{proof}
Thanks to the Zorn Lemma there exists a maximal subgroup $M$ contained in $K_{(2,2)}$ containing $\CF$. 
Since $\pi(\CF)=2\IZ\oplus 2\IZ\leq\pi(M)\leq 2\IZ\oplus 2\IZ$, we have $\pi(M)=2\IZ\oplus 2\IZ$.
The subgroup $M$ cannot have finite index by Lemma \ref{Lemmamax}.
\end{proof}

We are now ready to exhibit \textbf{three   infinite index maximal subgroups of  $K_{(2,2)}$ containing $\CF$}: 
$$
M_0:=\langle \CF, x_0^2\rangle \qquad M_1:=\langle \CF, x_1^2\rangle \qquad 
M_2:=\langle \CF, \sigma(x_1)^2\rangle
$$

\begin{theorem}\label{thmM1}
For any $g\in K_{(2,2)}\setminus M_i$, the group $\langle M_i, g\rangle$ contains $K_{(2,2)}$, with $i=0, 1, 2$. In particular, $M_0$, $M_1$ and $M_2$ are maximal infinite index subgroups of $K_{(2,2)}$.
\end{theorem}

We will adopt the same strategy as the one deployed in \cite{GS2} to prove the maximality of the oriented subgroup $\vec{F}$ in $K_{(1,2)}$. 
We start with the proof of Theorem \ref{thmM1} for $M_0$. Let us start recalling some   definitions and present a couple of preliminary lemmas.
We will need the notion of closure of a subgroup in $F$.
It  first appeared in \cite{GS2}, but  we will need 
the equivalent description from \cite{NR}.
It will allow us to prove that certain inclusions of subgroups are strict.

Let $(T_+,T_-)$ be an element in $F$.
First direct all the edges in $T_+$ and $T_-$ away from the roots. 
Now let us consider the graph obtained by identifying the leaves pairwise.
This directed graph with two roots associated with an element of $F$ is called the element's diagram. 
Such diagram is said to be reduced if 
the corresponding tree diagram is reduced.
The vertices coming from the leaves are the only $2$-valent vertices of this graph.
Note that for any of these vertices there is exactly one directed path from each root to it.

\begin{definition}\label{defNR}
The \textbf{core} of a finitely generated subgroup $H=\langle g_1, \ldots , g_k\rangle \leq F$ is a vertex-labeled directed graph constructed in the following way. Begin with  the reduced  diagrams for $g_1, \ldots , g_k$ and identify all the roots of  these diagrams together. Proceed 
with identifying other vertices 
 according to   the following two rules  for as long as possible
\begin{enumerate}
\item If two vertices are identified, identify their left children and left incident edges, along with their right children and right  incident edges.
\item If two vertices have their left children and their right children identified respectively, then identify the vertices and the edges that connect them to their children.
\end{enumerate}
As there are only finitely many edges and vertices, this process will eventually end.
\end{definition}
See Figure \ref{CORElabel} for a graphical description of these steps.
An example with $M_0$ will be provided in the proof of Proposition \ref{lemmaH}.
This definition can be extended to arbitrary subgroups of $F$,  \cite{GS2}.
An element $g$ of $F$ is said to be \textbf{accepted by the core} of a subgroup $H$
 if there exists a homomorphism  
 of labeled directed graphs from the 
 core of $\langle g\rangle$ 
 to the core of $H$. 
The subset of $F$ accepted by the core of a   subgroup $H$ is actually a subgroup of $F$ (see \cite[Lemma 18]{NR}) and is denoted by Cl$(H)$. 
This subgroup is called the \textbf{closure} of $H$. A subgroup is called \textbf{closed} if it coincides with its closure.

 \begin{figure}
\phantom{This text will be invisible} 
\[
\begin{tikzpicture}[x=.75cm, y=.75cm,
    every edge/.style={
        draw,
      postaction={decorate,
                    decoration={markings}
                   }
        }
]

\draw[thick] (0,0)--(1,1)--(2,0);
\draw[thick] (1,2)--(1,1);
\node at (0.5,.4) {\rotatebox[origin=tr]{-135}{$\scalebox{.75}{$>$}$}};
\node at (1.5,.5) {\rotatebox[origin=tr]{-45}{$\scalebox{.75}{$>$}$}};
\node at (1,1.5) {\rotatebox[origin=tr]{-90}{$\scalebox{.75}{$>$}$}};
 \node[red] at (.8,1.1) {$\scalebox{.75}{$a$}$};
 \node[red] at (-.3,0) {$\scalebox{.75}{$b$}$};
 \node[red] at (2.3,0) {$\scalebox{.75}{$c$}$};

\draw[thick] (4,0)--(5,1)--(6,0);
\draw[thick] (5,2)--(5,1);
\node at (4.5,.4) {\rotatebox[origin=tr]{-135}{$\scalebox{.75}{$>$}$}};
\node at (5.5,.5) {\rotatebox[origin=tr]{-45}{$\scalebox{.75}{$>$}$}};
\node at (5,1.5) {\rotatebox[origin=tr]{-90}{$\scalebox{.75}{$>$}$}};
 \node[red] at (4.8,1.1) {$\scalebox{.75}{$a$}$};
 \node[red] at (3.7,0) {$\scalebox{.75}{$d$}$};
 \node[red] at (6.3,0) {$\scalebox{.75}{$e$}$};

\node at (8,.5) {$\scalebox{1}{$\Rightarrow$}$};
\node at (11,.5) {$\scalebox{1}{$b=d \qquad c=e$}$};


\node at (0,-1.2) {$\;$};
\end{tikzpicture}
\]\[
\begin{tikzpicture}[x=.75cm, y=.75cm,
    every edge/.style={
        draw,
      postaction={decorate,
                    decoration={markings}
                   }
        }
]

\draw[thick] (0,0)--(1,1)--(2,0);
\draw[thick] (1,2)--(1,1);
\node at (0.5,.4) {\rotatebox[origin=tr]{-135}{$\scalebox{.75}{$>$}$}};
\node at (1.5,.5) {\rotatebox[origin=tr]{-45}{$\scalebox{.75}{$>$}$}};
\node at (1,1.5) {\rotatebox[origin=tr]{-90}{$\scalebox{.75}{$>$}$}};
 \node[red] at (.8,1.1) {$\scalebox{.75}{$a$}$};
 \node[red] at (-.3,0) {$\scalebox{.75}{$c$}$};
 \node[red] at (2.3,0) {$\scalebox{.75}{$d$}$};

\draw[thick] (4,0)--(5,1)--(6,0);
\draw[thick] (5,2)--(5,1);
\node at (4.5,.4) {\rotatebox[origin=tr]{-135}{$\scalebox{.75}{$>$}$}};
\node at (5.5,.5) {\rotatebox[origin=tr]{-45}{$\scalebox{.75}{$>$}$}};
\node at (5,1.5) {\rotatebox[origin=tr]{-90}{$\scalebox{.75}{$>$}$}};
 \node[red] at (4.8,1.1) {$\scalebox{.75}{$b$}$};
 \node[red] at (3.7,0) {$\scalebox{.75}{$c$}$};
 \node[red] at (6.3,0) {$\scalebox{.75}{$d$}$};

\phantom{\node at (11,.5) {$\scalebox{1}{$b=d \qquad c=e$}$};}
\node at (8,.5) {$\scalebox{1}{$\Rightarrow$}$};
\node at (10,.5) {$\scalebox{1}{$a=b$}$};


\node at (0,-1.2) {$\;$};
\end{tikzpicture}
\]
 \caption{Rules for labelling the vertices.}\label{CORElabel}
\end{figure}
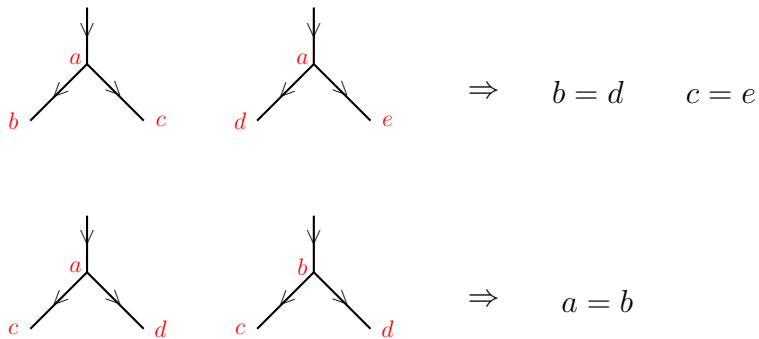

\begin{remark}
It was shown in 
\cite[Theorem 5.01]{GS2} that Cl$(H)$ 
consists of functions in $F$, where every linear piece of $f$ is a restriction of some function from $H$.
It follows from Lemma \ref{lemma-inclu} that $\CF$ is closed.
\end{remark}

 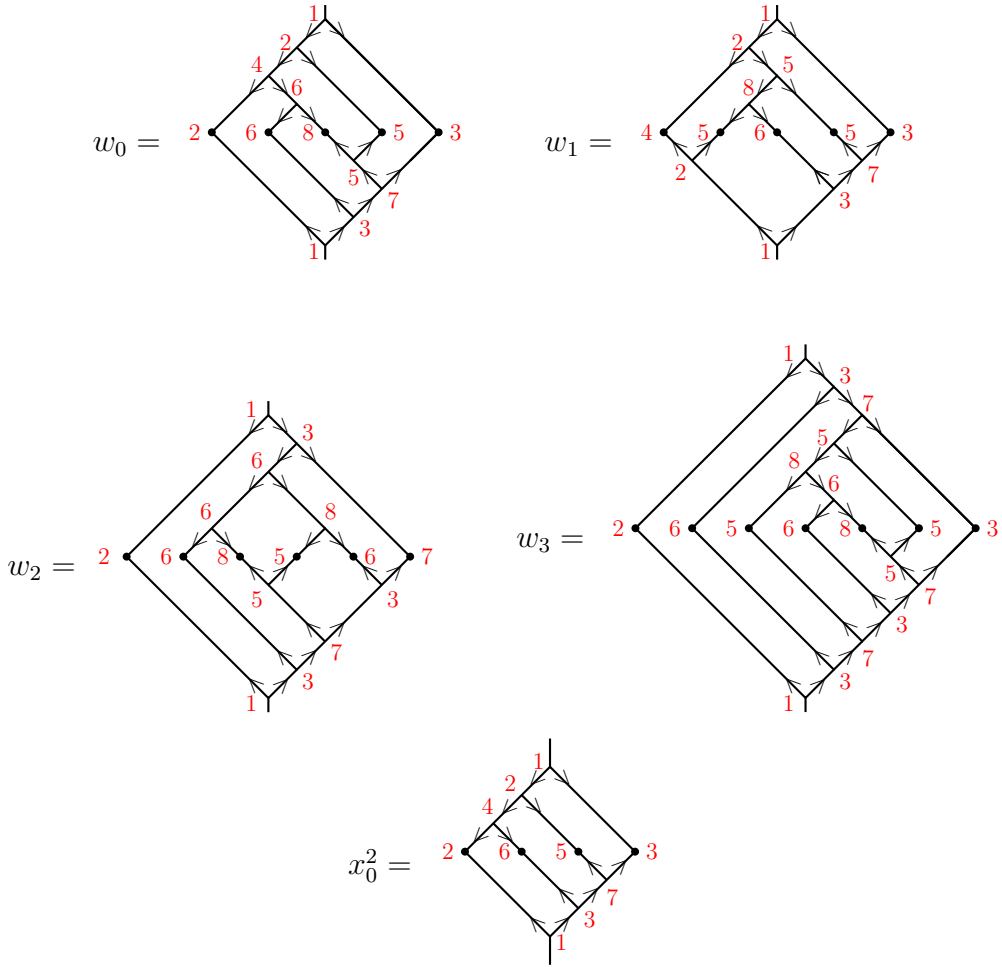
\begin{figure}
\phantom{This text will be invisible} 
\[
\begin{tikzpicture}[x=.75cm, y=.75cm,
    every edge/.style={
        draw,
      postaction={decorate,
                    decoration={markings}
                   }
        }
]

\node at (-1.5,-.25) {$\scalebox{1}{$w_0=$}$};

\draw[thick] (0,0)--(2,2)--(4,0);
\draw[thick] (0,0)--(2,-2)--(4,0);
\draw[thick] (1,0)--(1.5,0.5)--(2,0);
\draw[thick] (1.5,0.5)--(1,1);
\draw[thick] (3,0)--(1.5,1.5);
\draw[thick] (3,0)--(2.5,-.5); 
\draw[thick] (4,0)--(2,2);
\draw[thick] (4,0)--(3.5,-.5);
\draw[thick] (1,0)--(2.5,-1.5);
\draw[thick] (2,0)--(3,-1);

\draw[thick] (2,2)--(2,2.25);
\draw[thick] (2,-2)--(2,-2.25);

\node at (1.75,1.65) {\rotatebox[origin=tr]{-135}{$\scalebox{.75}{$>$}$}};
\node at (1.25,1.15) {\rotatebox[origin=tr]{-135}{$\scalebox{.75}{$>$}$}};
\node at (.75,.65) {\rotatebox[origin=tr]{-135}{$\scalebox{.75}{$>$}$}};
\node at (1.25,.75) {\rotatebox[origin=tr]{-45}{$\scalebox{.75}{$>$}$}};
\node at (2.25,1.75) {\rotatebox[origin=tr]{-45}{$\scalebox{.75}{$>$}$}};
\node at (1.75,1.25) {\rotatebox[origin=tr]{-45}{$\scalebox{.75}{$>$}$}};
\node at (1.75,.25) {\rotatebox[origin=tr]{-45}{$\scalebox{.75}{$>$}$}};
\node at (1.25,.15) {\rotatebox[origin=tr]{-135}{$\scalebox{.75}{$>$}$}};

\node at (1.75,-1.75) {\rotatebox[origin=tr]{135}{$\scalebox{.75}{$>$}$}};
\node at (2.25,-1.25) {\rotatebox[origin=tr]{135}{$\scalebox{.75}{$>$}$}};
\node at (2.25,-.25) {\rotatebox[origin=tr]{135}{$\scalebox{.75}{$>$}$}};
\node at (2.75,-.25) {\rotatebox[origin=tr]{45}{$\scalebox{.75}{$>$}$}};
\node at (3.25,-.75) {\rotatebox[origin=tr]{45}{$\scalebox{.75}{$>$}$}};
\node at (2.75,-.75) {\rotatebox[origin=tr]{135}{$\scalebox{.75}{$>$}$}};
\node at (2.75,-1.25) {\rotatebox[origin=tr]{45}{$\scalebox{.75}{$>$}$}};
\node at (2.25,-1.75) {\rotatebox[origin=tr]{45}{$\scalebox{.75}{$>$}$}};

  \fill (0,0)  circle[radius=1.5pt];
  \fill (1,0)  circle[radius=1.5pt];  
  \fill (2,0)  circle[radius=1.5pt];
  \fill (3,0)  circle[radius=1.5pt];
  \fill (4,0)  circle[radius=1.5pt];

 \node[red] at (1.8,2.1) {$\scalebox{.75}{$1$}$};
 \node[red] at (1.8,-2.1) {$\scalebox{.75}{$1$}$};
 \node[red] at (1.3,1.6) {$\scalebox{.75}{$2$}$};
 \node[red] at (-.3,0) {$\scalebox{.75}{$2$}$};
 \node[red] at (4.3,0) {$\scalebox{.75}{$3$}$};
 \node[red] at (2.7,-1.7) {$\scalebox{.75}{$3$}$};
 \node[red] at (.8,1.2) {$\scalebox{.75}{$4$}$};
  \node[red] at (1.7,0) {$\scalebox{.75}{$8$}$};
  \node[red] at (3.3,0) {$\scalebox{.75}{$5$}$};
  \node[red] at (2.5,-.8) {$\scalebox{.75}{$5$}$};
 \node[red] at (3.2,-1.2) {$\scalebox{.75}{$7$}$};
 \node[red] at (1.5,.8) {$\scalebox{.75}{$6$}$};
 \node[red] at (.7,0) {$\scalebox{.75}{$6$}$};

\node at (0,-1.2) {$\;$};
\end{tikzpicture}
\qquad 
\begin{tikzpicture}[x=.75cm, y=.75cm,
    every edge/.style={
        draw,
      postaction={decorate,
                    decoration={markings}
                   }
        }
]

\node at (-1.5,-.25) {$\scalebox{1}{$w_1=$}$};

\draw[thick] (0,0)--(2,2)--(4,0);
\draw[thick] (0,0)--(2,-2)--(4,0);

\draw[thick] (1,0)--(1.5,0.5)--(2,0);
\draw[thick] (2,0)--(3,-1);
\draw[thick] (1,0)--(.5,-.5);
\draw[thick] (1.5,.5)--(2,1)--(3,0);
\draw[thick] (3,0)--(3.5,-.5);
\draw[thick] (2,1)--(1.5,1.5);

\draw[thick] (2,2)--(2,2.25);
\draw[thick] (2,-2)--(2,-2.25);

\node at (0,-1.2) {$\;$};

\node at (1.75,1.65) {\rotatebox[origin=tr]{-135}{$\scalebox{.75}{$>$}$}};
\node at (1.75,.65) {\rotatebox[origin=tr]{-135}{$\scalebox{.75}{$>$}$}};
\node at (1.25,1.15) {\rotatebox[origin=tr]{-135}{$\scalebox{.75}{$>$}$}};
\node at (1.75,-1.75) {\rotatebox[origin=tr]{135}{$\scalebox{.75}{$>$}$}};
\node at (2.75,-.75) {\rotatebox[origin=tr]{135}{$\scalebox{.75}{$>$}$}};
\node at (3.25,-.25) {\rotatebox[origin=tr]{135}{$\scalebox{.75}{$>$}$}};
\node at (.25,-.25) {\rotatebox[origin=tr]{135}{$\scalebox{.75}{$>$}$}};
\node at (3.25,-.75) {\rotatebox[origin=tr]{45}{$\scalebox{.75}{$>$}$}};
\node at (3.75,-.25) {\rotatebox[origin=tr]{45}{$\scalebox{.75}{$>$}$}};
\node at (2.75,-.75) {\rotatebox[origin=tr]{135}{$\scalebox{.75}{$>$}$}};
\node at (.75,-.25) {\rotatebox[origin=tr]{45}{$\scalebox{.75}{$>$}$}};
\node at (2.25,-1.75) {\rotatebox[origin=tr]{45}{$\scalebox{.75}{$>$}$}};
\node at (2.25,1.75) {\rotatebox[origin=tr]{-45}{$\scalebox{.75}{$>$}$}};
\node at (2.25,.75) {\rotatebox[origin=tr]{-45}{$\scalebox{.75}{$>$}$}};
\node at (1.75,.25) {\rotatebox[origin=tr]{-45}{$\scalebox{.75}{$>$}$}};
\node at (1.25,.15) {\rotatebox[origin=tr]{-135}{$\scalebox{.75}{$>$}$}};
\node at (1.75,1.25) {\rotatebox[origin=tr]{-45}{$\scalebox{.75}{$>$}$}};

  \fill (0,0)  circle[radius=1.5pt];
  \fill (1,0)  circle[radius=1.5pt];  
  \fill (2,0)  circle[radius=1.5pt];
  \fill (3,0)  circle[radius=1.5pt];
  \fill (4,0)  circle[radius=1.5pt];

 \node[red] at (1.8,2.1) {$\scalebox{.75}{$1$}$};
 \node[red] at (1.8,-2.1) {$\scalebox{.75}{$1$}$};
 \node[red] at (1.3,1.6) {$\scalebox{.75}{$2$}$};
 \node[red] at (-.3,0) {$\scalebox{.75}{$4$}$};
 \node[red] at (4.3,0) {$\scalebox{.75}{$3$}$};
 \node[red] at (.3,-.7) {$\scalebox{.75}{$2$}$};
 \node[red] at (2.2,1.2) {$\scalebox{.75}{$5$}$};
  \node[red] at (1.7,0) {$\scalebox{.75}{$6$}$};
  \node[red] at (3.3,0) {$\scalebox{.75}{$5$}$};
  \node[red] at (3.7,-.75) {$\scalebox{.75}{$7$}$};
  \node[red] at (3.2,-1.2) {$\scalebox{.75}{$3$}$};
 \node[red] at (1.5,.8) {$\scalebox{.75}{$8$}$};
 \node[red] at (.7,0) {$\scalebox{.75}{$5$}$};
  
\end{tikzpicture}
\]

\[
\begin{tikzpicture}[x=.75cm, y=.75cm,
    every edge/.style={
        draw,
      postaction={decorate,
                    decoration={markings}
                   }
        }
]

\node at (-2.5,-.25) {$\scalebox{1}{$w_2=$}$};
 
\draw[thick] (-1,0)--(1.5,2.5)--(2,2);
\draw[thick] (-1,0)--(1.5,-2.5)--(2,-2);

\draw[thick] (0,0)--(2,2)--(4,0);
\draw[thick] (0,0)--(2,-2)--(4,0);

\draw[thick] (2,0)--(2.5,.5)--(3,0);

\draw[thick] (1,0)--(.5,.5);
\draw[thick] (2.5,.5)--(1.5,1.5);

\draw[thick] (3,0)--(3.5,-.5);
\draw[thick] (2,0)--(1.5,-.5)--(1,0);
\draw[thick] (1.5,-.5)--(2.5,-1.5);

\draw[thick] (1.5,2.5)--(1.5,2.75);
\draw[thick] (1.5,-2.5)--(1.5,-2.75);

\node at (1.75,1.65) {\rotatebox[origin=tr]{-135}{$\scalebox{.75}{$>$}$}};
\node at (2.25,.15) {\rotatebox[origin=tr]{-135}{$\scalebox{.75}{$>$}$}};
\node at (2.25,1.75) {\rotatebox[origin=tr]{-45}{$\scalebox{.75}{$>$}$}};
\node at (1.25,2.15) {\rotatebox[origin=tr]{-135}{$\scalebox{.75}{$>$}$}};
\node at (1.75,2.25) {\rotatebox[origin=tr]{-45}{$\scalebox{.75}{$>$}$}};
\node at (.75,.25) {\rotatebox[origin=tr]{-45}{$\scalebox{.75}{$>$}$}};
\node at (2.75,.25) {\rotatebox[origin=tr]{-45}{$\scalebox{.75}{$>$}$}};
\node at (.25,.15) {\rotatebox[origin=tr]{-135}{$\scalebox{.75}{$>$}$}};
\node at (1.75,1.25) {\rotatebox[origin=tr]{-45}{$\scalebox{.75}{$>$}$}};
\node at (1.25,1.15) {\rotatebox[origin=tr]{-135}{$\scalebox{.75}{$>$}$}};

\node at (1.75,-1.75) {\rotatebox[origin=tr]{135}{$\scalebox{.75}{$>$}$}};
\node at (2.25,-1.25) {\rotatebox[origin=tr]{135}{$\scalebox{.75}{$>$}$}};
\node at (3.25,-.25) {\rotatebox[origin=tr]{135}{$\scalebox{.75}{$>$}$}};
\node at (2.75,-1.25) {\rotatebox[origin=tr]{45}{$\scalebox{.75}{$>$}$}};
\node at (3.75,-.25) {\rotatebox[origin=tr]{45}{$\scalebox{.75}{$>$}$}};
\node at (1.25,-.25) {\rotatebox[origin=tr]{135}{$\scalebox{.75}{$>$}$}};
\node at (1.25,-2.25) {\rotatebox[origin=tr]{135}{$\scalebox{.75}{$>$}$}};
\node at (1.75,-2.25) {\rotatebox[origin=tr]{45}{$\scalebox{.75}{$>$}$}};
\node at (1.75,-.25) {\rotatebox[origin=tr]{45}{$\scalebox{.75}{$>$}$}};
\node at (2.25,-1.75) {\rotatebox[origin=tr]{45}{$\scalebox{.75}{$>$}$}};

  \fill (0,0)  circle[radius=1.5pt];
  \fill (1,0)  circle[radius=1.5pt];  
  \fill (2,0)  circle[radius=1.5pt];
  \fill (3,0)  circle[radius=1.5pt];
  \fill (4,0)  circle[radius=1.5pt];
  \fill (-1,0)  circle[radius=1.5pt];

 \node[red] at (1.2,2.6) {$\scalebox{.75}{$1$}$};
 \node[red] at (1.2,-2.6) {$\scalebox{.75}{$1$}$};
  \node[red] at (2.2,2.2) {$\scalebox{.75}{$3$}$};
 \node[red] at (1.3,1.7) {$\scalebox{.75}{$6$}$};
 \node[red] at (.4,.8) {$\scalebox{.75}{$6$}$};
 \node[red] at (-.3,0) {$\scalebox{.75}{$6$}$};
   \node[red] at (1.7,0) {$\scalebox{.75}{$5$}$};
  \node[red] at (3.3,0) {$\scalebox{.75}{$6$}$};
 \node[red] at (.7,0) {$\scalebox{.75}{$8$}$};
 \node[red] at (-1.4,0) {$\scalebox{.75}{$2$}$};
 \node[red] at (4.3,0) {$\scalebox{.75}{$7$}$};
 \node[red] at (2.6,.8) {$\scalebox{.75}{$8$}$};
  \node[red] at (2.7,-1.7) {$\scalebox{.75}{$7$}$};
  \node[red] at (2.2,-2.2) {$\scalebox{.75}{$3$}$};
  \node[red] at (3.7,-.75) {$\scalebox{.75}{$3$}$};
  \node[red] at (1.3,-.75) {$\scalebox{.75}{$5$}$};

\node at (0,-1.2) {$\;$};
\end{tikzpicture}
\qquad 
\begin{tikzpicture}[x=.75cm, y=.75cm,
    every edge/.style={
        draw,
      postaction={decorate,
                    decoration={markings}
                   }
        }
]

\node at (-3.5,-.25) {$\scalebox{1}{$w_3=$}$};
 
\draw[thick] (-1,0)--(1.5,2.5)--(2,2);
\draw[thick] (-1,0)--(1.5,-2.5)--(2,-2);

\draw[thick] (-2,0)--(1,3)--(1.5,2.5);
\draw[thick] (-2,0)--(1,-3)--(1.5,-2.5);

\draw[thick] (0,0)--(2,2)--(4,0);
\draw[thick] (0,0)--(2,-2)--(4,0);
\draw[thick] (1,0)--(1.5,0.5)--(2,0);
\draw[thick] (1.5,0.5)--(1,1);
\draw[thick] (3,0)--(1.5,1.5);
\draw[thick] (3,0)--(2.5,-.5); 
\draw[thick] (4,0)--(2,2);
\draw[thick] (4,0)--(3.5,-.5);
\draw[thick] (1,0)--(2.5,-1.5);
\draw[thick] (2,0)--(3,-1);

\draw[thick] (1,3)--(1,3.25);
\draw[thick] (1,-3)--(1,-3.25);

\node at (1.75,1.65) {\rotatebox[origin=tr]{-135}{$\scalebox{.75}{$>$}$}};
\node at (2.25,1.75) {\rotatebox[origin=tr]{-45}{$\scalebox{.75}{$>$}$}};
\node at (1.25,2.15) {\rotatebox[origin=tr]{-135}{$\scalebox{.75}{$>$}$}};
\node at (1.75,2.25) {\rotatebox[origin=tr]{-45}{$\scalebox{.75}{$>$}$}};
\node at (.75,2.65) {\rotatebox[origin=tr]{-135}{$\scalebox{.75}{$>$}$}};
\node at (1.25,2.75) {\rotatebox[origin=tr]{-45}{$\scalebox{.75}{$>$}$}};
\node at (.75,.65) {\rotatebox[origin=tr]{-135}{$\scalebox{.75}{$>$}$}};
\node at (1.25,.15) {\rotatebox[origin=tr]{-135}{$\scalebox{.75}{$>$}$}};
\node at (1.75,.25) {\rotatebox[origin=tr]{-45}{$\scalebox{.75}{$>$}$}};
\node at (1.25,.75) {\rotatebox[origin=tr]{-45}{$\scalebox{.75}{$>$}$}};
\node at (1.75,1.25) {\rotatebox[origin=tr]{-45}{$\scalebox{.75}{$>$}$}};
\node at (1.25,1.15) {\rotatebox[origin=tr]{-135}{$\scalebox{.75}{$>$}$}};

\node at (2.25,-.25) {\rotatebox[origin=tr]{135}{$\scalebox{.75}{$>$}$}};
\node at (1.75,-1.75) {\rotatebox[origin=tr]{135}{$\scalebox{.75}{$>$}$}};
\node at (2.25,-1.25) {\rotatebox[origin=tr]{135}{$\scalebox{.75}{$>$}$}};
\node at (2.75,-.75) {\rotatebox[origin=tr]{135}{$\scalebox{.75}{$>$}$}};
\node at (.75,-2.75) {\rotatebox[origin=tr]{135}{$\scalebox{.75}{$>$}$}};
\node at (2.75,-1.25) {\rotatebox[origin=tr]{45}{$\scalebox{.75}{$>$}$}};
\node at (3.25,-.75) {\rotatebox[origin=tr]{45}{$\scalebox{.75}{$>$}$}};
\node at (1.25,-2.25) {\rotatebox[origin=tr]{135}{$\scalebox{.75}{$>$}$}};
\node at (1.25,-2.75) {\rotatebox[origin=tr]{45}{$\scalebox{.75}{$>$}$}};
\node at (1.75,-2.25) {\rotatebox[origin=tr]{45}{$\scalebox{.75}{$>$}$}};
\node at (2.75,-.25) {\rotatebox[origin=tr]{45}{$\scalebox{.75}{$>$}$}};
\node at (2.25,-1.75) {\rotatebox[origin=tr]{45}{$\scalebox{.75}{$>$}$}};

  \fill (0,0)  circle[radius=1.5pt];
  \fill (1,0)  circle[radius=1.5pt];  
  \fill (2,0)  circle[radius=1.5pt];
  \fill (3,0)  circle[radius=1.5pt];
  \fill (4,0)  circle[radius=1.5pt];
  \fill (-1,0)  circle[radius=1.5pt];
  \fill (-2,0)  circle[radius=1.5pt];

 \node[red] at (2.1,2.2) {$\scalebox{.75}{$7$}$};
 \node[red] at (2.1,-2.3) {$\scalebox{.75}{$7$}$};
 \node[red] at (1.3,1.6) {$\scalebox{.75}{$5$}$};
 \node[red] at (-.3,0) {$\scalebox{.75}{$5$}$};
 \node[red] at (4.3,0) {$\scalebox{.75}{$3$}$};
 \node[red] at (2.7,-1.7) {$\scalebox{.75}{$3$}$};
 \node[red] at (.8,1.2) {$\scalebox{.75}{$8$}$};
  \node[red] at (1.7,0) {$\scalebox{.75}{$8$}$};
  \node[red] at (3.3,0) {$\scalebox{.75}{$5$}$};
  \node[red] at (2.5,-.8) {$\scalebox{.75}{$5$}$};
 \node[red] at (3.2,-1.2) {$\scalebox{.75}{$7$}$};
 \node[red] at (1.5,.8) {$\scalebox{.75}{$6$}$};
 \node[red] at (.7,0) {$\scalebox{.75}{$6$}$};
 \node[red] at (-1.3,0) {$\scalebox{.75}{$6$}$};
 \node[red] at (-2.3,0) {$\scalebox{.75}{$2$}$};
 \node[red] at (.7,-3.1) {$\scalebox{.75}{$1$}$};
 \node[red] at (.7,3.1) {$\scalebox{.75}{$1$}$};
 \node[red] at (1.7,2.7) {$\scalebox{.75}{$3$}$};
 \node[red] at (1.7,-2.7) {$\scalebox{.75}{$3$}$};

\node at (0,-1.2) {$\;$};
\end{tikzpicture}
\]
\[
\begin{tikzpicture}[x=.75cm, y=.75cm,
    every edge/.style={
        draw,
      postaction={decorate,
                    decoration={markings}
                   }
        }
]

\node at (-1.5,-.25) {$\scalebox{1}{$x_0^2=$}$};

\draw[thick] (0,0)--(1.5,1.5)--(3,0)--(1.5,-1.5)--(0,0);
\draw[thick] (.5,.5)--(2,-1);
\draw[thick] (1,1)--(2.5,-.5);
\draw[thick] (1.5,1.5)--(1.5,2);
\draw[thick] (1.5,-1.5)--(1.5,-2);

\node at (.25,.15) {\rotatebox[origin=tr]{-135}{$\scalebox{.75}{$>$}$}};
\node at (1.25,1.15) {\rotatebox[origin=tr]{-135}{$\scalebox{.75}{$>$}$}};
\node at (.75,.65) {\rotatebox[origin=tr]{-135}{$\scalebox{.75}{$>$}$}};
 \node at (1.75,1.25) {\rotatebox[origin=tr]{-45}{$\scalebox{.75}{$>$}$}};
\node at (1.25,.75) {\rotatebox[origin=tr]{-45}{$\scalebox{.75}{$>$}$}};
\node at (.75,.25) {\rotatebox[origin=tr]{-45}{$\scalebox{.75}{$>$}$}};

\node at (1.25,-1.25) {\rotatebox[origin=tr]{135}{$\scalebox{.75}{$>$}$}};
\node at (1.75,-.75) {\rotatebox[origin=tr]{135}{$\scalebox{.75}{$>$}$}};
\node at (2.25,-.25) {\rotatebox[origin=tr]{135}{$\scalebox{.75}{$>$}$}};
\node at (2.75,-.25) {\rotatebox[origin=tr]{45}{$\scalebox{.75}{$>$}$}};
\node at (2.25,-.75) {\rotatebox[origin=tr]{45}{$\scalebox{.75}{$>$}$}};
\node at (1.75,-1.25) {\rotatebox[origin=tr]{45}{$\scalebox{.75}{$>$}$}};

  \fill (0,0)  circle[radius=1.5pt];
  \fill (1,0)  circle[radius=1.5pt];  
  \fill (2,0)  circle[radius=1.5pt];
  \fill (3,0)  circle[radius=1.5pt];

 \node[red] at (1.7,-1.6) {$\scalebox{.75}{$1$}$};
 \node[red] at (1.3,1.6) {$\scalebox{.75}{$1$}$};
 \node[red] at (-.3,0) {$\scalebox{.75}{$2$}$};
 \node[red] at (2.2,-1.2) {$\scalebox{.75}{$3$}$};
 \node[red] at (.8,1.2) {$\scalebox{.75}{$2$}$};
  \node[red] at (1.7,0) {$\scalebox{.75}{$5$}$};
  \node[red] at (3.3,0) {$\scalebox{.75}{$3$}$};
  \node[red] at (2.6,-.8) {$\scalebox{.75}{$7$}$};
 \node[red] at (.4,.8) {$\scalebox{.75}{$4$}$};
 \node[red] at (.7,0) {$\scalebox{.75}{$6$}$};

\node at (0,-1.2) {$\;$};
\end{tikzpicture}
\]
 \caption{The labellings of the generators of $M_0$.}\label{labelsGENH}
\end{figure}

As a preliminary step to proving the maximality of $M_0$, we show that adding a positive element of length $2$ to it either doesn't change it (Lemma \ref{M0even})
or gives the whole $K_{(2,2)}$ (Lemma \ref{lemmaH1}).

Let us start by understanding better the subgroup $M_0$.
We first show that $M_0$ is a proper subgroup of $K_{(2,2)}$.

\begin{lemma}	\label{non-max-H}
The subgroup $M_0$ is a proper subgroup of $K_{(2,2)}$. 
\end{lemma}
\begin{proof}
It suffices to show that $x_1^2\not\in M_0$.
To this end, we will actually prove an a priori stronger result: $x_1^2$ does not belong to the closure of $M_0$.
Our main tool is the core of $M_0$.
In order to do this, it suffices to show that $x_1^2$ does not admit a labelling compatible with that of the generators of $M_0$, that is $x_1^2$ does not belong to $Cl(M_0)$ (see Figure \ref{labelsGENH}).
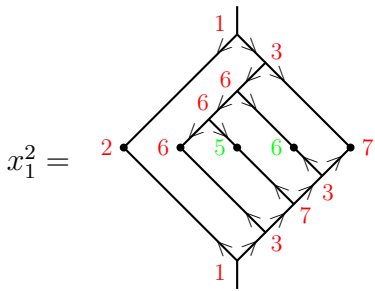
\[
\begin{tikzpicture}[x=.75cm, y=.75cm,
    every edge/.style={
        draw,
      postaction={decorate,
                    decoration={markings}
                   }
        }
]

\node at (-2.5,-.25) {$\scalebox{1}{$x_1^2=$}$};
 
\draw[thick] (1,1)--(2.5,-.5);
\draw[thick] (0,0)--(1.5,1.5)--(3,0)--(1.5,-1.5)--(0,0);
\draw[thick] (.5,.5)--(2,-1);
\draw[thick] (1,-2)--(1,-2.5);
\draw[thick] (1,2)--(1,2.5);
\draw[thick] (-1,0)--(1,2)--(1.5,1.5);
\draw[thick] (-1,0)--(1,-2)--(1.5,-1.5);
 
\node at (.75,1.65) {\rotatebox[origin=tr]{-135}{$\scalebox{.75}{$>$}$}};
\node at (1.25,1.75) {\rotatebox[origin=tr]{-45}{$\scalebox{.75}{$>$}$}};
\node at (.75,-1.75) {\rotatebox[origin=tr]{135}{$\scalebox{.75}{$>$}$}};
\node at (1.25,-1.75) {\rotatebox[origin=tr]{45}{$\scalebox{.75}{$>$}$}};
\node at (.25,.15) {\rotatebox[origin=tr]{-135}{$\scalebox{.75}{$>$}$}};
\node at (1.25,1.15) {\rotatebox[origin=tr]{-135}{$\scalebox{.75}{$>$}$}};
\node at (.75,.65) {\rotatebox[origin=tr]{-135}{$\scalebox{.75}{$>$}$}};
 \node at (1.75,1.25) {\rotatebox[origin=tr]{-45}{$\scalebox{.75}{$>$}$}};
\node at (1.25,.75) {\rotatebox[origin=tr]{-45}{$\scalebox{.75}{$>$}$}};
\node at (.75,.25) {\rotatebox[origin=tr]{-45}{$\scalebox{.75}{$>$}$}};

\node at (1.25,-1.25) {\rotatebox[origin=tr]{135}{$\scalebox{.75}{$>$}$}};
\node at (1.75,-.75) {\rotatebox[origin=tr]{135}{$\scalebox{.75}{$>$}$}};
\node at (2.25,-.25) {\rotatebox[origin=tr]{135}{$\scalebox{.75}{$>$}$}};
\node at (2.75,-.25) {\rotatebox[origin=tr]{45}{$\scalebox{.75}{$>$}$}};
\node at (2.25,-.75) {\rotatebox[origin=tr]{45}{$\scalebox{.75}{$>$}$}};
\node at (1.75,-1.25) {\rotatebox[origin=tr]{45}{$\scalebox{.75}{$>$}$}};

  \fill (-1,0)  circle[radius=1.5pt];
  \fill (0,0)  circle[radius=1.5pt];
  \fill (1,0)  circle[radius=1.5pt];  
  \fill (2,0)  circle[radius=1.5pt];
  \fill (3,0)  circle[radius=1.5pt];
 
   \node[red] at (.7,-2.2) {$\scalebox{.75}{$1$}$};
   \node[red] at (.7,2.2) {$\scalebox{.75}{$1$}$};
   \node[red] at (-1.3,0) {$\scalebox{.75}{$2$}$};

 \node[red] at (1.7,-1.7) {$\scalebox{.75}{$3$}$};
 \node[red] at (1.7,1.7) {$\scalebox{.75}{$3$}$};
 \node[red] at (-.3,0) {$\scalebox{.75}{$6$}$};
 \node[red] at (2.2,-1.2) {$\scalebox{.75}{$7$}$};
 \node[red] at (.8,1.2) {$\scalebox{.75}{$6$}$};
  \node[green] at (1.7,0) {$\scalebox{.75}{$6$}$};
  \node[red] at (3.3,0) {$\scalebox{.75}{$7$}$};
  \node[red] at (2.6,-.8) {$\scalebox{.75}{$3$}$};
 \node[red] at (.4,.8) {$\scalebox{.75}{$6$}$};
 \node[green] at (.7,0) {$\scalebox{.75}{$5$}$};

\node at (0,-1.2) {$\;$};
\end{tikzpicture}
\]
\end{proof}
\begin{remark}
The previous lemma shows that $\CF$ is not a maximal subgroup of $K_{(2,2)}$.
\end{remark}

\begin{proposition}\label{lemmaH}
The subgroup $M_0=\langle \CF, x_0^2\rangle$  is generated by $x_0^2$, $x_2^2$, $x_1x_2$.
\end{proposition}
\begin{proof}
First we show that the following elements generate $M_0$
$$
\CS_{M_0}:=\{x_{2k}^2, x_{2k+1}x_{2k+2} \; | \; k=0, 1, \ldots\}\; .
$$
Indeed, these elements are in $M_0$ 
\begin{align*}
&x_0^{-2}w_1x_0^2=x_0^{-1}x_1^2x_0=x_2^2\in M_0\\
&x_0^{-2k}x_2^2x_0^{2k}=x_{2k+2}^2\in M_0\\
&x_0^{-2}w_0x_2^2=x_1x_2\in M_0\\
&x_0^{-2k}x_1x_2x_0^{2k}=x_{1+2k}x_{2+2k}\in M_0\; .
\end{align*}
It is also easy to see that the generators of $M_0$ can be obtained from elements in $\CS_{M_0}$
\begin{align*}
&w_0=x_0^2(x_1x_2)x_2^{-2}\\
&w_1=x_0^2x_2^2x_0^{-2}\\
&w_2=(x_1x_2)(x_1x_2)x_2^{-2}\\
&w_3=x_2^2(x_3x_4)x_4^{-2}
\end{align*} 
\end{proof}

\begin{lemma} \label{M0even}
For all $i\geq 0$, the subgroups $\langle x_{2i}^2,\CF\rangle$ are all equal.
\end{lemma}
\begin{proof}
Denote by $R_{2i}:=\langle x_{2i}^2,\CF\rangle$. Since $\varphi(\CF)\subset \CF$, it holds $\varphi^{2i}(R_0)=\varphi^{2i}(\langle x_{0}^2,\CF\rangle)=\langle x_{2i}^2,\varphi^{2i}(\CF)\rangle\subset R_{2i}$. As $x_1x_2$, $x_2^2\in R_0$, we have $x_{2i+1}x_{2i+2}$, $x_{2i+2}^2\in R_{2i}$. In particular, $R_{2i+2}\leq R_{2i}$. 
We want to prove that the converse inclusion holds. 
First, notice that $w_1x_{2i+1}x_{2i+2}w_1^{-1}=x_{2i-1}x_{2i}\in R_{2i}$ for all $i\geq 1$,
where $w_1=x_0x_1^2x_0^{-1}$.
Therefore, we have $\varphi^{2i-2}(w_0)x_{2i}^2(x_{2i-1}x_{2i})^{-1}=x_{2i-2}^2\in R_{2i}$.
This means that $R_{2i-2}\leq R_{2i}$ and we are done.
\end{proof}

\begin{lemma}\label{lemmaH1} 
For any $g=x_ix_j\in K_{(2,2)}\setminus M_0$, the subgroup $\langle g,M_0\rangle$ is equal to $K_{(2,2)}$.
\end{lemma}
\begin{proof}
We note that since $g=x_ix_j\in K_{(2,2)}\setminus M_0$, then both $i$ and $j$ are different from $0$.\\
For the sake of clarity we divide the proof in seven cases depending on the form of $g$.

\noindent
\textbf{Case 1}: $g=x_1^2$.\\
Since $M_0\leq H_1:=\langle x_0^2, x_1^2,\CF\rangle$, 
by the proof of Proposition \ref{lemmaH}, 
we already know that $x_{2k}^2, x_{2k+1}x_{2k+2}\in M_0\leq H_1$.
We also have 
\begin{align*}
&x_0^{-2k}x_1^2x_0^{2k}=x_{1+2k}^2\in H_1\qquad k=0, 1, \ldots\\
&x_1^{-2}w_2x_2^2=x_3x_2\in H_1\\
&x_2x_3=x_2^2(x_3x_2)^{-1} x_3^{2}\in H_1\\
&x_0^{-2k}x_2x_3x_0^{2k}=x_{2+2k}x_{3+2k}\in H_1\qquad k=0, 1, \ldots \\
&(x_2x_3)x_3^{-2}(x_3x_4)=x_2x_4\in H_1\\
&x_0^{-2k}x_2x_4x_0^{2k}=x_{2+2k}x_{4+2k}\in H_1\\
&x_1x_3=x_2x_1=w_0^{-1}x_0^2x_1^2\in H_1\\
&x_0^{-2k}x_1x_3x_0^{2k}=x_{1+2k}x_{3+2k}\in H_1
\end{align*}
Since the elements $\{x_{1+2k}x_{3+2k}, x_{2+2k}x_{4+2k}, x_{2k}^2, x_{2k+1}^2, x_{2k+1}x_{2k+2}, x_{2k+2}x_{2k+3}\}_{k\geq 0}$ generate $K_{(2,2)}$ we are done.


\noindent
\textbf{Case 2}:  $g= x_{2i+1}^2$ for any $i\geq1$.\\
For any $i\geq 1$, it holds $x_0^{2}x_{2i+1}^2x_0^{-2}=x_{2i-1}^2$.
Therefore, the claim follows by 
Case 1 and an inductive argument.


\noindent
\textbf{Case 3}:  $g=x_{2i}x_{2i+1}$ for any $i\geq1$.\\
Assume first that $i=1$. We have 
\begin{align*}
&x_0^{-2}w_0x_2x_3=x_1x_3=x_2x_1\in \langle x_{2}x_{3},M_0\rangle\\
&x_0^{-2}w_0x_2x_1=x_1^2\in \langle x_{2}x_{3},M_0\rangle
\end{align*}
where 
$w_0=x_0^2x_1x_2^{-1}$.
In particular, by Case 1 we have $K_{(2,2)}=\langle  x_{1}^2,M_0\rangle \leq \langle x_{2}x_{3},M_0\rangle\leq K_{(2,2)}$ and we are done. 

If $i\geq 2$, the claim follows by induction from the equality 
$x_0^{2}x_{2i}x_{2i+1}x_0^{-2}=x_{2(i-1)}x_{2(i-1)+1}$.

\noindent
\textbf{Case 4}:  $g= x_{2i}x_{2j}$ for any $j\geq i+1\geq 1$, $i\geq 1$.\\
Suppose first that $j=i+1$. 
As $x_0^{2}(x_{2i}x_{2i+2})x_0^{-2}=x_{2(i-1)}x_{2i}$ for $i\geq 2$, we may assume also that $i=1$.
We know from Proposition \ref{lemmaH} that $x_1x_2\in M_0$ and, thus, $x_1^2=(x_1x_2)^2(x_2x_4)^{-1}\in \langle x_{2}x_{4},M_0\rangle$.
By Case 1 we get $\langle x_{2i}x_{2j},M_0\rangle =K_{(2,2)}$.

If $j\geq i+2$, the claim follows by induction from the equality $x_{2i}^{2}(x_{2i}x_{2j})x_{2i}^{-2}=x_{2i}x_{2(j-1)}\in \langle x_{2i}x_{2j},M_0\rangle$ (we recall that $x_{2i}^2\in M_0$ for all $i\geq 0$ by Proposition \ref{lemmaH}).


\noindent
\textbf{Case 5}:  $g=x_{2i}x_{2j+1}$ for any $j\geq i+1\geq 1$, $i\geq 1$.\\
Since $x_{2j}^{-2}(x_{2i}x_{2j+1})x_{2i}^{-2}=x_{2j}^{-2}(x_{2j}x_{2i})x_{2i}^{-2}=x_{2j}^{-1}x_{2i}^{-1}\in \langle x_{2i}x_{2j+1},M_0\rangle$, we have $x_{2i}x_{2j}\in \langle x_{2i}x_{2j+1},M_0\rangle$ and the claim follows from Case 4.


\noindent
\textbf{Case 6}:  $g=x_{2i+1}x_{2j}$ for any $j> i+1$, $i\geq 0$.\\
Since $x_{2i+2}^{2}(x_{2i+1}x_{2i+2})^{-1}(x_{2i+1}x_{2j})=x_{2(i+1)}x_{2j}$ the claim follows from Case 4.

\noindent
\textbf{Case 7}: $g=x_{2i+1}x_{2j+1}$, for any $j\geq i$, $i\geq 0$. \\
When $j=i$ the claim is precisely the content of Case 2. \\
Suppose that $j\geq i+1$.
We note that $x_{2i+2}^{2}(x_{2i+1}x_{2i+2})^{-1}(x_{2i+1}x_{2j+1})=x_{2i+2}x_{2j+1}$.
Now when $j=i+1$ the claim follows from 
Case 3. When $j\geq i+2$ it follows from
Case 5.

\end{proof}
The next result is not needed for the proof of the maximality of $M_0$ in $K_{(2,2)}$, but it will come in handy in 
 the proof of the maximality of the subgroup $M_1$.
\begin{lemma}\label{lemma-zero-j}
For any $j\geq 1$, it holds $\langle x_{0}x_{j},M_0\rangle = K_{(1,2)}$.
\end{lemma}
\begin{proof}
As $x_0^{-2}x_0x_jx_0^{-2}=x_0x_{j-2}$ for any $j\geq 2$, it is enough to consider the cases $j=1$ and $j=2$.

Assume $j=1$. Then, we have
\begin{align*}
&x_0^{-2}(x_0x_1)^2=x_1x_3=x_2x_1\in \langle x_0x_1, M_0\rangle\\
&x_1^2=(x_1x_2)x_2^{-2}(x_2x_1)\in \langle x_0x_1, M_0\rangle
\end{align*}
so by Lemma \ref{lemmaH1} we see that $K_{(2,2)}=\langle x_0^2,x_1^2, \CF\rangle <  \langle x_0x_1, M_0\rangle\leq K_{(1,2)}$. Now $|K_{(1,2)}:K_{(2,2)}|=2$ and, thus, $\langle x_0x_1, M_0\rangle =K_{(1,2)}$.

Similarly, for $j=2$, first we observe that $x_2x_4=x_3x_2=x_0^{-2}(x_0x_2)^2\in \langle x_0x_2, M_0\rangle$. 
Recall that $w_2=x_1^2x_3x_2^{-1}$.
Then, $x_1^2=w_2(x_2x_4)(x_3x_4)^{-1}\in \langle x_0x_1, M_0\rangle$ and, therefore, by Lemma \ref{lemmaH1} we have $K_{(2,2)}=\langle x_0^2,x_1^2, M_0\rangle <  \langle x_0x_2, M_0\rangle \leq K_{(1,2)}$. Since $\langle x_0x_1, K_{(2,2)}\rangle=K_{(1,2)}$, we are done.
\end{proof}
 
In the next definition and lemma we collect some notions   and  results from \cite[Section 3]{GS2}.
\begin{definition}
Let  $w=x_{i_1}\cdots x_{i_n}$ be a positive normal form of $F$, then
\begin{itemize}
\item a letter $x_i$ is said to \textbf{skip over} $w$ if $x_iw=wx_{i+n}$.
\item $w$ is a \textbf{block}  if contains at least two distinct letters and $x_{i_1+1}$ skips over $B$. 
\item $w$ is a \textbf{minimal block} if $w'=x_{i_2}\cdots x_{i_n}$ is not a block.
\item for any $k\in\IN$, the element $w''=x_{i_1+k}\cdots x_{i_n+k}$ is said to be a \textbf{translation} of $w$. 
\end{itemize}  
\end{definition}
\begin{lemma}\cite[Lemma 3.6, Lemma 3.9, remark in proof of Lemma 3.11]{GS2} \label{lemma-block-prop}
Let  $B=x_{i_1}\cdots x_{i_n}$ be a positive normal form. 
Then
\begin{enumerate}
\item a letter $x_i$ skips over $B$ if and only if for all  $j=1, \ldots , n$ we have $i_j<i+j-1$;
\item if $x_i$ skips over $B$, then $x_k$ skips over $w$ for all $k>i$.
\end{enumerate}
%
Moreover, if  $B=x_{i_1}\cdots x_{i_n}$ is a block, then
\begin{enumerate}
\item[(3)]  every translation $B'$ of $B$  is a block; 
\item[(4)] for every $j\neq i_1$ we have $x_j^{-1}B=B'x_r^{-1}$, where $B'$ is a translation of $B$ and $r\in\IN$; 
\item[(5)] let $w=w_1Bw_2$ be a positive normal form, where $B$ is a block and $w_1$, $w_2$ are some possibly empty words; then either $x_j^{-1}w$ is shorter than $w$ or $x_j^{-1}w=w_1'B'w_2'$ for some $w_1'$, $w_2'$, $B'$ such that $|w_1'|=|w_1|$, $|w_2'|=|w_2|+1$, $B'$ is a translation of $B$; 
\item[(6)] If $B$ is a minimal block, then there exists a $j\in\{2,\ldots , n\}$ such that $i_j=i_1+j-1$.
\end{enumerate}
\end{lemma}
Observe that for a normal form $B=x_{i_1}\cdots x_{i_n}$ being a block amounts to having $x_{i_1}\neq x_{i_n}$ and
that $x_{i_1+1}$ skips over $B$, \cite[Remark 3.8]{GS2}.\\
We will need the following   lemma in the proof of Theorem \ref{thmM1}. 
\begin{lemma}
 \label{positivity}
For any $g\in 
 K_{(2,2)}\setminus M_0$, the double coset $M_0gM_0$ contains a positive element $w$ such that for every $w'\in M_0gM_0$, it holds $|w|\leq |w'|$. Moreover, $w$ may be chosen such that it also does not contain any block. 
\end{lemma}
\begin{proof}
First we show that there is a positive element of minimal length in the double coset $M_0gM_0$.
Take $w\in M_0gM_0$ of minimal length.
If it is positive, we are done. Otherwise, 
find its normal form: $w=x_{0}^{a_0}\cdots x_{n}^{a_n}x_{n}^{-b_n}\cdots x_{0}^{-b_0}$. The last non-zero factor is $x_{i_0}^{-1}$. If $i_0\in 2\IN_0$, take the element $wx_{i_0}^2\in M_0gM_0$, otherwise  take $wx_{i_0}x_{i_0+1}\in M_0gM_0$. In both cases, the new element admits a normal form of the same length, but with less negative factors. The claim follows by iteration.

Now we prove that $w$ does not contain blocks.
Let $w$ an element in the double coset $M_0gM_0$ of minimal length. By the previous discussion we may assume that $w$ is in $F_+$. 
Suppose that the normal form of $w$ contains a block, that is $w=z_1Bz_2$, with $B$ being a minimal block.

We now show that we can replace $w$ with another element in $M_0gM_0$ of the same length and such that $w'=B'z_2'$, with $B'$ being a translation of $B$. 
If $z_1=\emptyset$, there is nothing to do.
If $z_1$ is non-empty, then $z_1=x_jz''_1$. 
First, we observe that if $j=0$, then actually $z_1=x_0^{2k}\tilde{w}$. However, this is in contradiction with the minimality of the length of $w$ because $x_0^{-2}w$ is shorter than $w$ and still in $M_0gM_0$. Therefore, $j$ is necessarily different from $0$.
Now if $j$ is odd, consider $z'':=(x_jx_{j+1})^{-1}w=x_{j+1}^{-1}z''_1Bz_2$ (recall that $x_jx_{j+1}\in M_0$).
By Lemma \ref{lemma-block-prop} (5) and the minimality of $w$ we get $z''=z_1'B'z_2'$, where $|z_1'|=|z_1''|=|z_1|-1$, $|z_2'|=|z_2|+1$, $B'$ is a translation of $B$. By iteration we may assume that $z_1$ is empty. If $j$ is even, consider 
$x_j^{-2}w=x_{j}^{-1}z''_1Bz_2$ (note that $x_j^{2}\in M_0$) and argue as before.

By the previous discussion we may assume that $w=Bz_2$, where $B=x_{i_1}\cdots x_{i_n}$. 
There are two cases depending on whether $i_1$ is odd or even. In the first case consider $t_1=(x_{i_1}x_{i_1+1})^{-1}w=x_{i_1+1}^{-1}x_{i_2}\cdots x_{i_n}$.
By Lemma \ref{lemma-block-prop} (6) 
 there exists a $j$ such that $i_j=i_1+j-1$. This means that $x_{i_1+1}^{-1}$
cancel the first occurrence of $x_{i_j}$ in $t$. This is in contradiction with our hypothesis of $w$ being of minimal length and we are done.
We want to show that the second case cannot occur. Since $w$ is of minimal length, $i_2\geq i_1+1$ (if $i_2=i_1$, then the element $x_{i_1}^{-2}w\in M_0gM_0$ is shorter than $w$).
As $B$ is a block $i_2<i_1+2$ and thus $i_2=i_1+1$.
Similarly $i_3\geq i_2=i_1+1$ and $i_3<i_1+3$. 
Therefore, we have two sub-cases: $i_3=i_2=i_1+1$ and $i_3=i_1+2$.
In the first sub-case 
the element $x_{i_1}^{-2}(x_{i_1-1}x_{i_1}^{-1})w=x_{i_1}^{-2}(x_{i_1-1}x_{i_1}^{-1})Bz_2=x_{i_1}^{-2}x_{i_1-1}x_{i_1+1}^2x_{i_4}\cdots x_{i_n}z_2=x_{i_1}^{-2}x_{i_1}^2 x_{i_1-1}x_{i_4}\cdots x_{i_n}z_2=x_{i_1-1}x_{i_4}\cdots x_{i_n}z_2\in M_0gM_0$ 
 is shorter than $w$. This is absurd.
 The second sub-case is impossible as well since we assumed that $B$ is a minimal block, while $\tilde{B}:=x_{i_2}\cdots x_{i_n}$ is a block. Indeed, set $i'_j:=i_{j+1}$ for $j=1, \ldots , n-1$ and consider $\tilde{B}=x_{i'_1}\cdots x_{i'_{n-1}}$. We only have to check that $i'_j<i'_1+j$ for all $j=1, \ldots , n-1$.
By definition we have $i'_j:=i_{j+1}<i_1+j+1=i_2+j=i'_1+j$. 
\end{proof}
We are finally in a position to prove 
Theorem \ref{thmM1} for $M_0$.  
\begin{proof}[Proof of Theorem \ref{thmM1} for $M_0$]
Let $g\in K_{(2,2)}\setminus M_0$. We need to show that $\langle g, M_0\rangle =K_{(2,2)}$.
Note that $\langle g, M_0\rangle =\langle g', M_0\rangle$ for any $g'\in M_0 gM_0$.
By Lemma \ref{positivity} 
  we may hence 
  assume that $g$ is positive, 
   does not contain any block and is an element of minimal length in $M_0gM_0$.

We  give a proof by induction on the length of the normal form of $g$ (which is even because $g$ is in $K_{(2,2)}$).
If the length is $2$, then the claim is precisely the content of Lemma \ref{lemmaH1}.
 
Suppose that the length is bigger than $2$. 
 We have $g=w'x_j^k$ with $j\in \IN$ and $w'=x_{i_1}\cdots x_{i_m}$ is either empty or its last letter is not $x_j$. If $w'$ is empty (in this case this implies that $k\in 2\IN$, $j\in 2\IN_0+1$), then $x_j^{-k}(x_j x_{j+1}) x_j^k=x_jx_{j+1+k}\in \langle M_0, x_j^k\rangle$ and by Lemma \ref{lemmaH1}
 we are done.

Suppose that $w'$ is non-empty and let $m=|w'|$. 
Now we have two cases: (1) $j$ is odd; (2) $j$ is even and $k=1$. Without loss of generality, we may suppose that $j\geq m=|w'|$ (it suffices to replace $g$ by $x_0^{-2l}gx_0^{2l}$ with $l$ big enough). In this case it holds $x_{j-m}w'=w'x_j$ if $j\geq m$ (\cite[Formula ($\star$) in proof of  Theorem 3.12]{GS2}).
For case (1), there are two sub-cases: $k$ and $m$ are even, $k$ and $m$ are odd.
If $k$ and $m$ are even take the element
\begin{align*}
g^{-1}(x_{j-m}x_{j-m+1})g& = x_j^{-k}w'^{-1}	x_{j-m}x_{j-m+1}w'x_j^k	\\
&= x_j^{-k}w'^{-1}w'x_{j}x_{j+1}x_j^k\\
&= x_j^{-k}x_{j}x_{j+1}x_j^k\\
&= x_{j}x_{j+1+k} \in \langle g, M_0\rangle
\end{align*}
which contains $K_{(2,2)}$ by 
Lemma \ref{lemmaH1}.\\
If $k$ and $m$ are odd take the element 
\begin{align*}
g^{-1}x_{j-m}^2g& = x_j^{-k}w'^{-1}	x_{j-m}^2w'x_j^k	\\
&= x_j^{-k}w'^{-1}w'x_j^2x_j^k\\
&=x_{j}^2 \in \langle g, M_0\rangle
\end{align*}
For some $l\in \IN$, we have $x_0^{2l}x_j^2x_0^{-2l}=x_1^2$. The claim now follows from Lemma \ref{lemmaH1}.

For case (2), we observe that $m$ is odd and consider the element 
\begin{align*}
g^{-1}(x_{j-m}x_{j-m+1})g&=x_j^{-1}w'^{-1}(x_{j-m}x_{j-m+1})w'x_j\\
&=x_j^{-1}w'^{-1}w'x_jx_{j+1}x_j\\
&=x_{j+1}x_j=x_jx_{j+2}\in M_0gM_0
\end{align*}
Now the claim follows from Lemma \ref{lemmaH1}.

\end{proof}

\begin{corollary}\label{cor41}
The subgroup $\theta^{-1}(M_0)$ is a maximal infinite index subgroup of $F$.
\end{corollary}

\begin{remark}
In the proof of Lemma \ref{non-max-H} it was shown that $Cl(M_0)$ is a proper subgroup of $K_{(2,2)}$ (since $x_1^2\not\in Cl(M_0)$). 
As $M_0\leq Cl(M_0)<K_{(2,2)}$ (the second inclusion follows from \cite[Lemma 20]{NR}), it follows that the subgroup $M_0$ is closed, that is $Cl(M_0)=M_0$.
\end{remark}

We now consider two other infinite index subgroups of $K_{(2,2)}$ containing $\CF$: $M_1:=\langle x_1^2, \CF\rangle$ and $M_2:=\langle x_0x_1x_2x_0^{-3}, \CF\rangle$. Note that $\sigma(x_1^2)=x_0x_1x_2x_0^{-3}$. Clearly, $M_1$ is isomorphic with $M_2$ thanks to $\sigma\upharpoonright_{K_{(2,2)}}: K_{(2,2)}\to K_{(2,2)}$, and $M_1$ is a maximal subgroup if and only if $M_2$ is a maximal subgroup because $\sigma(M_1)=M_2$.

We begin with a couple of lemmas that allow us to understand better the subgroups $M_1$ and $M_2$. 
 \begin{lemma}\label{lemma-M}
The subgroup $M_1$ contains the subset
$$
\CS_{M_1}:=\{x_{2k+1}^2, x_{2k+2}x_{2k+3}\; | \; k=0,1,\ldots\}\; .
$$
In particular, $\varphi(M_0)=\langle \CS_{M_1}\rangle$ is contained in $M_1$.
\end{lemma}
\begin{proof}
The claim follows by easy computations
\begin{align*}
&x_1^{-2}\varphi(w_1)x_1^2=x_3^2\in M_1\\
&x_1^{-2k}x_3^2x_1^{2k}=x_{3+2k}^2\in M_1\\
&x_1^{-2}\varphi(w_0)x_3^2=x_2x_3\in M_1\\
&x_1^{-2k}x_2x_3x_1^{2k}=x_{2+2k}x_{3+2k}\in M_1
\end{align*}
where 
$w_0=x_0^2x_1x_2^{-1}$,
$w_1=x_0x_1^2x_0^{-1}$.
\end{proof}

\begin{lemma}
The subgroups $M_1$ and $M_2$ are proper distinct subgroups of $K_{(2,2)}$. 
\end{lemma}
\begin{proof}
First we show that $x_0^2\not\in M_1$.
In order to do this, it suffices to show that $x_0^2$ does not admit a labelling compatible with that of the generators of $M_1$, that is $x_0^2$ does not belong to $Cl(M_1)$ (see Figure \ref{labelsGENM}).
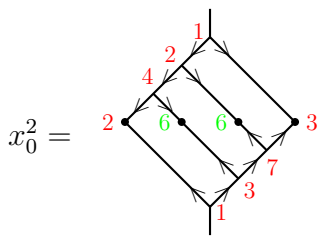
\[
\begin{tikzpicture}[x=.75cm, y=.75cm,
    every edge/.style={
        draw,
      postaction={decorate,
                    decoration={markings}
                   }
        }
]

\node at (-1.5,-.25) {$\scalebox{1}{$x_0^2=$}$};

\draw[thick] (0,0)--(1.5,1.5)--(3,0)--(1.5,-1.5)--(0,0);
\draw[thick] (.5,.5)--(2,-1);
\draw[thick] (1,1)--(2.5,-.5);
\draw[thick] (1.5,1.5)--(1.5,2);
\draw[thick] (1.5,-1.5)--(1.5,-2);

\node at (.25,.15) {\rotatebox[origin=tr]{-135}{$\scalebox{.75}{$>$}$}};
\node at (1.25,1.15) {\rotatebox[origin=tr]{-135}{$\scalebox{.75}{$>$}$}};
\node at (.75,.65) {\rotatebox[origin=tr]{-135}{$\scalebox{.75}{$>$}$}};
 \node at (1.75,1.25) {\rotatebox[origin=tr]{-45}{$\scalebox{.75}{$>$}$}};
\node at (1.25,.75) {\rotatebox[origin=tr]{-45}{$\scalebox{.75}{$>$}$}};
\node at (.75,.25) {\rotatebox[origin=tr]{-45}{$\scalebox{.75}{$>$}$}};

\node at (1.25,-1.25) {\rotatebox[origin=tr]{135}{$\scalebox{.75}{$>$}$}};
\node at (1.75,-.75) {\rotatebox[origin=tr]{135}{$\scalebox{.75}{$>$}$}};
\node at (2.25,-.25) {\rotatebox[origin=tr]{135}{$\scalebox{.75}{$>$}$}};
\node at (2.75,-.25) {\rotatebox[origin=tr]{45}{$\scalebox{.75}{$>$}$}};
\node at (2.25,-.75) {\rotatebox[origin=tr]{45}{$\scalebox{.75}{$>$}$}};
\node at (1.75,-1.25) {\rotatebox[origin=tr]{45}{$\scalebox{.75}{$>$}$}};

  \fill (0,0)  circle[radius=1.5pt];
  \fill (1,0)  circle[radius=1.5pt];  
  \fill (2,0)  circle[radius=1.5pt];
  \fill (3,0)  circle[radius=1.5pt];

 \node[red] at (1.7,-1.6) {$\scalebox{.75}{$1$}$};
 \node[red] at (1.3,1.6) {$\scalebox{.75}{$1$}$};
 \node[red] at (-.3,0) {$\scalebox{.75}{$2$}$};
 \node[red] at (2.2,-1.2) {$\scalebox{.75}{$3$}$};
 \node[red] at (.8,1.2) {$\scalebox{.75}{$2$}$};
  \node[green] at (1.7,0) {$\scalebox{.75}{$6$}$};
  \node[red] at (3.3,0) {$\scalebox{.75}{$3$}$};
  \node[red] at (2.6,-.8) {$\scalebox{.75}{$7$}$};
 \node[red] at (.4,.8) {$\scalebox{.75}{$4$}$};
 \node[green] at (.7,0) {$\scalebox{.75}{$6$}$};

\node at (0,-1.2) {$\;$};
\end{tikzpicture}
\]
Now we show that $x_0^2$ and $x_1^2$ do not belong to $M_2$ (see Figure \ref{labelsGENM2}) and, thus, $M_2$ is different from both $M_1$ and $M_0$. 
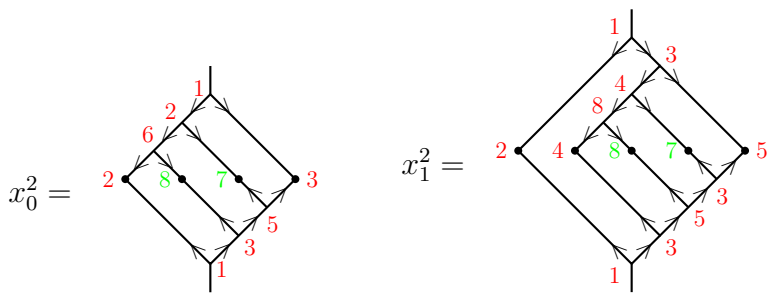
\[
\begin{tikzpicture}[x=.75cm, y=.75cm,
    every edge/.style={
        draw,
      postaction={decorate,
                    decoration={markings}
                   }
        }
]

\node at (-1.5,-.25) {$\scalebox{1}{$x_0^2=$}$};

\draw[thick] (0,0)--(1.5,1.5)--(3,0)--(1.5,-1.5)--(0,0);
\draw[thick] (.5,.5)--(2,-1);
\draw[thick] (1,1)--(2.5,-.5);
\draw[thick] (1.5,1.5)--(1.5,2);
\draw[thick] (1.5,-1.5)--(1.5,-2);

\node at (.25,.15) {\rotatebox[origin=tr]{-135}{$\scalebox{.75}{$>$}$}};
\node at (1.25,1.15) {\rotatebox[origin=tr]{-135}{$\scalebox{.75}{$>$}$}};
\node at (.75,.65) {\rotatebox[origin=tr]{-135}{$\scalebox{.75}{$>$}$}};
 \node at (1.75,1.25) {\rotatebox[origin=tr]{-45}{$\scalebox{.75}{$>$}$}};
\node at (1.25,.75) {\rotatebox[origin=tr]{-45}{$\scalebox{.75}{$>$}$}};
\node at (.75,.25) {\rotatebox[origin=tr]{-45}{$\scalebox{.75}{$>$}$}};

\node at (1.25,-1.25) {\rotatebox[origin=tr]{135}{$\scalebox{.75}{$>$}$}};
\node at (1.75,-.75) {\rotatebox[origin=tr]{135}{$\scalebox{.75}{$>$}$}};
\node at (2.25,-.25) {\rotatebox[origin=tr]{135}{$\scalebox{.75}{$>$}$}};
\node at (2.75,-.25) {\rotatebox[origin=tr]{45}{$\scalebox{.75}{$>$}$}};
\node at (2.25,-.75) {\rotatebox[origin=tr]{45}{$\scalebox{.75}{$>$}$}};
\node at (1.75,-1.25) {\rotatebox[origin=tr]{45}{$\scalebox{.75}{$>$}$}};

  \fill (0,0)  circle[radius=1.5pt];
  \fill (1,0)  circle[radius=1.5pt];  
  \fill (2,0)  circle[radius=1.5pt];
  \fill (3,0)  circle[radius=1.5pt];

 \node[red] at (1.7,-1.6) {$\scalebox{.75}{$1$}$};
 \node[red] at (1.3,1.6) {$\scalebox{.75}{$1$}$};
 \node[red] at (-.3,0) {$\scalebox{.75}{$2$}$};
 \node[red] at (2.2,-1.2) {$\scalebox{.75}{$3$}$};
 \node[red] at (.8,1.2) {$\scalebox{.75}{$2$}$};
  \node[green] at (1.7,0) {$\scalebox{.75}{$7$}$};
  \node[red] at (3.3,0) {$\scalebox{.75}{$3$}$};
  \node[red] at (2.6,-.8) {$\scalebox{.75}{$5$}$};
 \node[red] at (.4,.8) {$\scalebox{.75}{$6$}$};
 \node[green] at (.7,0) {$\scalebox{.75}{$8$}$};

\node at (0,-1.2) {$\;$};
\end{tikzpicture}
\qquad 
\begin{tikzpicture}[x=.75cm, y=.75cm,
    every edge/.style={
        draw,
      postaction={decorate,
                    decoration={markings}
                   }
        }
]

\node at (-2.5,-.25) {$\scalebox{1}{$x_1^2=$}$};
 
\draw[thick] (1,1)--(2.5,-.5);
\draw[thick] (0,0)--(1.5,1.5)--(3,0)--(1.5,-1.5)--(0,0);
\draw[thick] (.5,.5)--(2,-1);
\draw[thick] (1,-2)--(1,-2.5);
\draw[thick] (1,2)--(1,2.5);
\draw[thick] (-1,0)--(1,2)--(1.5,1.5);
\draw[thick] (-1,0)--(1,-2)--(1.5,-1.5);
 
\node at (.75,1.65) {\rotatebox[origin=tr]{-135}{$\scalebox{.75}{$>$}$}};
\node at (1.25,1.75) {\rotatebox[origin=tr]{-45}{$\scalebox{.75}{$>$}$}};
\node at (.75,-1.75) {\rotatebox[origin=tr]{135}{$\scalebox{.75}{$>$}$}};
\node at (1.25,-1.75) {\rotatebox[origin=tr]{45}{$\scalebox{.75}{$>$}$}};
\node at (.25,.15) {\rotatebox[origin=tr]{-135}{$\scalebox{.75}{$>$}$}};
\node at (1.25,1.15) {\rotatebox[origin=tr]{-135}{$\scalebox{.75}{$>$}$}};
\node at (.75,.65) {\rotatebox[origin=tr]{-135}{$\scalebox{.75}{$>$}$}};
 \node at (1.75,1.25) {\rotatebox[origin=tr]{-45}{$\scalebox{.75}{$>$}$}};
\node at (1.25,.75) {\rotatebox[origin=tr]{-45}{$\scalebox{.75}{$>$}$}};
\node at (.75,.25) {\rotatebox[origin=tr]{-45}{$\scalebox{.75}{$>$}$}};

\node at (1.25,-1.25) {\rotatebox[origin=tr]{135}{$\scalebox{.75}{$>$}$}};
\node at (1.75,-.75) {\rotatebox[origin=tr]{135}{$\scalebox{.75}{$>$}$}};
\node at (2.25,-.25) {\rotatebox[origin=tr]{135}{$\scalebox{.75}{$>$}$}};
\node at (2.75,-.25) {\rotatebox[origin=tr]{45}{$\scalebox{.75}{$>$}$}};
\node at (2.25,-.75) {\rotatebox[origin=tr]{45}{$\scalebox{.75}{$>$}$}};
\node at (1.75,-1.25) {\rotatebox[origin=tr]{45}{$\scalebox{.75}{$>$}$}};

  \fill (-1,0)  circle[radius=1.5pt];
  \fill (0,0)  circle[radius=1.5pt];
  \fill (1,0)  circle[radius=1.5pt];  
  \fill (2,0)  circle[radius=1.5pt];
  \fill (3,0)  circle[radius=1.5pt];
 
   \node[red] at (.7,-2.2) {$\scalebox{.75}{$1$}$};
   \node[red] at (.7,2.2) {$\scalebox{.75}{$1$}$};
   \node[red] at (-1.3,0) {$\scalebox{.75}{$2$}$};

 \node[red] at (1.7,-1.7) {$\scalebox{.75}{$3$}$};
 \node[red] at (1.7,1.7) {$\scalebox{.75}{$3$}$};
 \node[red] at (-.3,0) {$\scalebox{.75}{$4$}$};
 \node[red] at (2.2,-1.2) {$\scalebox{.75}{$5$}$};
 \node[red] at (.8,1.2) {$\scalebox{.75}{$4$}$};
  \node[green] at (1.7,0) {$\scalebox{.75}{$7$}$};
  \node[red] at (3.3,0) {$\scalebox{.75}{$5$}$};
  \node[red] at (2.6,-.8) {$\scalebox{.75}{$3$}$};
 \node[red] at (.4,.8) {$\scalebox{.75}{$8$}$};
 \node[green] at (.7,0) {$\scalebox{.75}{$8$}$};

\node at (0,-1.2) {$\;$};
\end{tikzpicture}
\]
\end{proof}

 \begin{figure}
\phantom{This text will be invisible} 
\[
\begin{tikzpicture}[x=.75cm, y=.75cm,
    every edge/.style={
        draw,
      postaction={decorate,
                    decoration={markings}
                   }
        }
]

\node at (-1.5,-.25) {$\scalebox{1}{$w_0=$}$};

\draw[thick] (0,0)--(2,2)--(4,0);
\draw[thick] (0,0)--(2,-2)--(4,0);
\draw[thick] (1,0)--(1.5,0.5)--(2,0);
\draw[thick] (1.5,0.5)--(1,1);
\draw[thick] (3,0)--(1.5,1.5);
\draw[thick] (3,0)--(2.5,-.5); 
\draw[thick] (4,0)--(2,2);
\draw[thick] (4,0)--(3.5,-.5);
\draw[thick] (1,0)--(2.5,-1.5);
\draw[thick] (2,0)--(3,-1);

\draw[thick] (2,2)--(2,2.25);
\draw[thick] (2,-2)--(2,-2.25);

\node at (1.75,1.65) {\rotatebox[origin=tr]{-135}{$\scalebox{.75}{$>$}$}};
\node at (1.25,1.15) {\rotatebox[origin=tr]{-135}{$\scalebox{.75}{$>$}$}};
\node at (.75,.65) {\rotatebox[origin=tr]{-135}{$\scalebox{.75}{$>$}$}};
\node at (1.25,.75) {\rotatebox[origin=tr]{-45}{$\scalebox{.75}{$>$}$}};
\node at (2.25,1.75) {\rotatebox[origin=tr]{-45}{$\scalebox{.75}{$>$}$}};
\node at (1.75,1.25) {\rotatebox[origin=tr]{-45}{$\scalebox{.75}{$>$}$}};
\node at (1.75,.25) {\rotatebox[origin=tr]{-45}{$\scalebox{.75}{$>$}$}};
\node at (1.25,.15) {\rotatebox[origin=tr]{-135}{$\scalebox{.75}{$>$}$}};

\node at (1.75,-1.75) {\rotatebox[origin=tr]{135}{$\scalebox{.75}{$>$}$}};
\node at (2.25,-1.25) {\rotatebox[origin=tr]{135}{$\scalebox{.75}{$>$}$}};
\node at (2.25,-.25) {\rotatebox[origin=tr]{135}{$\scalebox{.75}{$>$}$}};
\node at (2.75,-.25) {\rotatebox[origin=tr]{45}{$\scalebox{.75}{$>$}$}};
\node at (3.25,-.75) {\rotatebox[origin=tr]{45}{$\scalebox{.75}{$>$}$}};
\node at (2.75,-.75) {\rotatebox[origin=tr]{135}{$\scalebox{.75}{$>$}$}};
\node at (2.75,-1.25) {\rotatebox[origin=tr]{45}{$\scalebox{.75}{$>$}$}};
\node at (2.25,-1.75) {\rotatebox[origin=tr]{45}{$\scalebox{.75}{$>$}$}};

  \fill (0,0)  circle[radius=1.5pt];
  \fill (1,0)  circle[radius=1.5pt];  
  \fill (2,0)  circle[radius=1.5pt];
  \fill (3,0)  circle[radius=1.5pt];
  \fill (4,0)  circle[radius=1.5pt];

 \node[red] at (1.8,2.1) {$\scalebox{.75}{$1$}$};
 \node[red] at (1.8,-2.1) {$\scalebox{.75}{$1$}$};
 \node[red] at (1.3,1.6) {$\scalebox{.75}{$2$}$};
 \node[red] at (-.3,0) {$\scalebox{.75}{$2$}$};
 \node[red] at (4.3,0) {$\scalebox{.75}{$3$}$};
 \node[red] at (2.7,-1.7) {$\scalebox{.75}{$3$}$};
 \node[red] at (.8,1.2) {$\scalebox{.75}{$4$}$};
  \node[red] at (1.7,0) {$\scalebox{.75}{$8$}$};
  \node[red] at (3.3,0) {$\scalebox{.75}{$6$}$};
  \node[red] at (2.5,-.8) {$\scalebox{.75}{$8$}$};
 \node[red] at (3.2,-1.2) {$\scalebox{.75}{$7$}$};
 \node[red] at (1.5,.8) {$\scalebox{.75}{$6$}$};
 \node[red] at (.7,0) {$\scalebox{.75}{$5$}$};

\node at (0,-1.2) {$\;$};
\end{tikzpicture}
\qquad 
\begin{tikzpicture}[x=.75cm, y=.75cm,
    every edge/.style={
        draw,
      postaction={decorate,
                    decoration={markings}
                   }
        }
]

\node at (-1.5,-.25) {$\scalebox{1}{$w_1=$}$};

\draw[thick] (0,0)--(2,2)--(4,0);
\draw[thick] (0,0)--(2,-2)--(4,0);

\draw[thick] (1,0)--(1.5,0.5)--(2,0);
\draw[thick] (2,0)--(3,-1);
\draw[thick] (1,0)--(.5,-.5);
\draw[thick] (1.5,.5)--(2,1)--(3,0);
\draw[thick] (3,0)--(3.5,-.5);
\draw[thick] (2,1)--(1.5,1.5);

\draw[thick] (2,2)--(2,2.25);
\draw[thick] (2,-2)--(2,-2.25);

\node at (0,-1.2) {$\;$};

\node at (1.75,1.65) {\rotatebox[origin=tr]{-135}{$\scalebox{.75}{$>$}$}};
\node at (1.75,.65) {\rotatebox[origin=tr]{-135}{$\scalebox{.75}{$>$}$}};
\node at (1.25,1.15) {\rotatebox[origin=tr]{-135}{$\scalebox{.75}{$>$}$}};
\node at (1.75,-1.75) {\rotatebox[origin=tr]{135}{$\scalebox{.75}{$>$}$}};
\node at (2.75,-.75) {\rotatebox[origin=tr]{135}{$\scalebox{.75}{$>$}$}};
\node at (3.25,-.25) {\rotatebox[origin=tr]{135}{$\scalebox{.75}{$>$}$}};
\node at (.25,-.25) {\rotatebox[origin=tr]{135}{$\scalebox{.75}{$>$}$}};
\node at (3.25,-.75) {\rotatebox[origin=tr]{45}{$\scalebox{.75}{$>$}$}};
\node at (3.75,-.25) {\rotatebox[origin=tr]{45}{$\scalebox{.75}{$>$}$}};
\node at (2.75,-.75) {\rotatebox[origin=tr]{135}{$\scalebox{.75}{$>$}$}};
\node at (.75,-.25) {\rotatebox[origin=tr]{45}{$\scalebox{.75}{$>$}$}};
\node at (2.25,-1.75) {\rotatebox[origin=tr]{45}{$\scalebox{.75}{$>$}$}};
\node at (2.25,1.75) {\rotatebox[origin=tr]{-45}{$\scalebox{.75}{$>$}$}};
\node at (2.25,.75) {\rotatebox[origin=tr]{-45}{$\scalebox{.75}{$>$}$}};
\node at (1.75,.25) {\rotatebox[origin=tr]{-45}{$\scalebox{.75}{$>$}$}};
\node at (1.25,.15) {\rotatebox[origin=tr]{-135}{$\scalebox{.75}{$>$}$}};
\node at (1.75,1.25) {\rotatebox[origin=tr]{-45}{$\scalebox{.75}{$>$}$}};

  \fill (0,0)  circle[radius=1.5pt];
  \fill (1,0)  circle[radius=1.5pt];  
  \fill (2,0)  circle[radius=1.5pt];
  \fill (3,0)  circle[radius=1.5pt];
  \fill (4,0)  circle[radius=1.5pt];

 \node[red] at (1.8,2.1) {$\scalebox{.75}{$1$}$};
 \node[red] at (1.8,-2.1) {$\scalebox{.75}{$1$}$};
 \node[red] at (1.3,1.6) {$\scalebox{.75}{$2$}$};
 \node[red] at (-.3,0) {$\scalebox{.75}{$4$}$};
 \node[red] at (4.3,0) {$\scalebox{.75}{$3$}$};
 \node[red] at (.3,-.7) {$\scalebox{.75}{$2$}$};
 \node[red] at (2.2,1.2) {$\scalebox{.75}{$6$}$};
  \node[red] at (1.7,0) {$\scalebox{.75}{$5$}$};
  \node[red] at (3.3,0) {$\scalebox{.75}{$8$}$};
  \node[red] at (3.7,-.75) {$\scalebox{.75}{$7$}$};
  \node[red] at (3.2,-1.2) {$\scalebox{.75}{$3$}$};
 \node[red] at (1.5,.8) {$\scalebox{.75}{$5$}$};
 \node[red] at (.7,0) {$\scalebox{.75}{$6$}$};
  
\end{tikzpicture}
\]

\[
\begin{tikzpicture}[x=.75cm, y=.75cm,
    every edge/.style={
        draw,
      postaction={decorate,
                    decoration={markings}
                   }
        }
]

\node at (-2.5,-.25) {$\scalebox{1}{$w_2=$}$};
 
\draw[thick] (-1,0)--(1.5,2.5)--(2,2);
\draw[thick] (-1,0)--(1.5,-2.5)--(2,-2);

\draw[thick] (0,0)--(2,2)--(4,0);
\draw[thick] (0,0)--(2,-2)--(4,0);

\draw[thick] (2,0)--(2.5,.5)--(3,0);

\draw[thick] (1,0)--(.5,.5);
\draw[thick] (2.5,.5)--(1.5,1.5);

\draw[thick] (3,0)--(3.5,-.5);
\draw[thick] (2,0)--(1.5,-.5)--(1,0);
\draw[thick] (1.5,-.5)--(2.5,-1.5);

\draw[thick] (1.5,2.5)--(1.5,2.75);
\draw[thick] (1.5,-2.5)--(1.5,-2.75);

\node at (1.75,1.65) {\rotatebox[origin=tr]{-135}{$\scalebox{.75}{$>$}$}};
\node at (2.25,.15) {\rotatebox[origin=tr]{-135}{$\scalebox{.75}{$>$}$}};
\node at (2.25,1.75) {\rotatebox[origin=tr]{-45}{$\scalebox{.75}{$>$}$}};
\node at (1.25,2.15) {\rotatebox[origin=tr]{-135}{$\scalebox{.75}{$>$}$}};
\node at (1.75,2.25) {\rotatebox[origin=tr]{-45}{$\scalebox{.75}{$>$}$}};
\node at (.75,.25) {\rotatebox[origin=tr]{-45}{$\scalebox{.75}{$>$}$}};
\node at (2.75,.25) {\rotatebox[origin=tr]{-45}{$\scalebox{.75}{$>$}$}};
\node at (.25,.15) {\rotatebox[origin=tr]{-135}{$\scalebox{.75}{$>$}$}};
\node at (1.75,1.25) {\rotatebox[origin=tr]{-45}{$\scalebox{.75}{$>$}$}};
\node at (1.25,1.15) {\rotatebox[origin=tr]{-135}{$\scalebox{.75}{$>$}$}};

\node at (1.75,-1.75) {\rotatebox[origin=tr]{135}{$\scalebox{.75}{$>$}$}};
\node at (2.25,-1.25) {\rotatebox[origin=tr]{135}{$\scalebox{.75}{$>$}$}};
\node at (3.25,-.25) {\rotatebox[origin=tr]{135}{$\scalebox{.75}{$>$}$}};
\node at (2.75,-1.25) {\rotatebox[origin=tr]{45}{$\scalebox{.75}{$>$}$}};
\node at (3.75,-.25) {\rotatebox[origin=tr]{45}{$\scalebox{.75}{$>$}$}};
\node at (1.25,-.25) {\rotatebox[origin=tr]{135}{$\scalebox{.75}{$>$}$}};
\node at (1.25,-2.25) {\rotatebox[origin=tr]{135}{$\scalebox{.75}{$>$}$}};
\node at (1.75,-2.25) {\rotatebox[origin=tr]{45}{$\scalebox{.75}{$>$}$}};
\node at (1.75,-.25) {\rotatebox[origin=tr]{45}{$\scalebox{.75}{$>$}$}};
\node at (2.25,-1.75) {\rotatebox[origin=tr]{45}{$\scalebox{.75}{$>$}$}};

  \fill (0,0)  circle[radius=1.5pt];
  \fill (1,0)  circle[radius=1.5pt];  
  \fill (2,0)  circle[radius=1.5pt];
  \fill (3,0)  circle[radius=1.5pt];
  \fill (4,0)  circle[radius=1.5pt];
  \fill (-1,0)  circle[radius=1.5pt];

 \node[red] at (1.2,2.6) {$\scalebox{.75}{$1$}$};
 \node[red] at (1.2,-2.6) {$\scalebox{.75}{$1$}$};
  \node[red] at (2.2,2.2) {$\scalebox{.75}{$3$}$};
 \node[red] at (1.3,1.7) {$\scalebox{.75}{$5$}$};
 \node[red] at (.4,.8) {$\scalebox{.75}{$6$}$};
 \node[red] at (-.3,0) {$\scalebox{.75}{$5$}$};
   \node[red] at (1.7,0) {$\scalebox{.75}{$6$}$};
  \node[red] at (3.3,0) {$\scalebox{.75}{$5$}$};
 \node[red] at (.7,0) {$\scalebox{.75}{$8$}$};
 \node[red] at (-1.4,0) {$\scalebox{.75}{$2$}$};
 \node[red] at (4.3,0) {$\scalebox{.75}{$7$}$};
 \node[red] at (2.6,.8) {$\scalebox{.75}{$5$}$};
  \node[red] at (2.7,-1.7) {$\scalebox{.75}{$7$}$};
  \node[red] at (2.2,-2.2) {$\scalebox{.75}{$3$}$};
  \node[red] at (3.7,-.75) {$\scalebox{.75}{$3$}$};
  \node[red] at (1.3,-.75) {$\scalebox{.75}{$8$}$};

\node at (0,-1.2) {$\;$};
\end{tikzpicture}
\qquad 
\begin{tikzpicture}[x=.75cm, y=.75cm,
    every edge/.style={
        draw,
      postaction={decorate,
                    decoration={markings}
                   }
        }
]

\node at (-3.5,-.25) {$\scalebox{1}{$w_3=$}$};
 
\draw[thick] (-1,0)--(1.5,2.5)--(2,2);
\draw[thick] (-1,0)--(1.5,-2.5)--(2,-2);

\draw[thick] (-2,0)--(1,3)--(1.5,2.5);
\draw[thick] (-2,0)--(1,-3)--(1.5,-2.5);

\draw[thick] (0,0)--(2,2)--(4,0);
\draw[thick] (0,0)--(2,-2)--(4,0);
\draw[thick] (1,0)--(1.5,0.5)--(2,0);
\draw[thick] (1.5,0.5)--(1,1);
\draw[thick] (3,0)--(1.5,1.5);
\draw[thick] (3,0)--(2.5,-.5); 
\draw[thick] (4,0)--(2,2);
\draw[thick] (4,0)--(3.5,-.5);
\draw[thick] (1,0)--(2.5,-1.5);
\draw[thick] (2,0)--(3,-1);

\draw[thick] (1,3)--(1,3.25);
\draw[thick] (1,-3)--(1,-3.25);

\node at (1.75,1.65) {\rotatebox[origin=tr]{-135}{$\scalebox{.75}{$>$}$}};
\node at (2.25,1.75) {\rotatebox[origin=tr]{-45}{$\scalebox{.75}{$>$}$}};
\node at (1.25,2.15) {\rotatebox[origin=tr]{-135}{$\scalebox{.75}{$>$}$}};
\node at (1.75,2.25) {\rotatebox[origin=tr]{-45}{$\scalebox{.75}{$>$}$}};
\node at (.75,2.65) {\rotatebox[origin=tr]{-135}{$\scalebox{.75}{$>$}$}};
\node at (1.25,2.75) {\rotatebox[origin=tr]{-45}{$\scalebox{.75}{$>$}$}};
\node at (.75,.65) {\rotatebox[origin=tr]{-135}{$\scalebox{.75}{$>$}$}};
\node at (1.25,.15) {\rotatebox[origin=tr]{-135}{$\scalebox{.75}{$>$}$}};
\node at (1.75,.25) {\rotatebox[origin=tr]{-45}{$\scalebox{.75}{$>$}$}};
\node at (1.25,.75) {\rotatebox[origin=tr]{-45}{$\scalebox{.75}{$>$}$}};
\node at (1.75,1.25) {\rotatebox[origin=tr]{-45}{$\scalebox{.75}{$>$}$}};
\node at (1.25,1.15) {\rotatebox[origin=tr]{-135}{$\scalebox{.75}{$>$}$}};

\node at (2.25,-.25) {\rotatebox[origin=tr]{135}{$\scalebox{.75}{$>$}$}};
\node at (1.75,-1.75) {\rotatebox[origin=tr]{135}{$\scalebox{.75}{$>$}$}};
\node at (2.25,-1.25) {\rotatebox[origin=tr]{135}{$\scalebox{.75}{$>$}$}};
\node at (2.75,-.75) {\rotatebox[origin=tr]{135}{$\scalebox{.75}{$>$}$}};
\node at (.75,-2.75) {\rotatebox[origin=tr]{135}{$\scalebox{.75}{$>$}$}};
\node at (2.75,-1.25) {\rotatebox[origin=tr]{45}{$\scalebox{.75}{$>$}$}};
\node at (3.25,-.75) {\rotatebox[origin=tr]{45}{$\scalebox{.75}{$>$}$}};
\node at (1.25,-2.25) {\rotatebox[origin=tr]{135}{$\scalebox{.75}{$>$}$}};
\node at (1.25,-2.75) {\rotatebox[origin=tr]{45}{$\scalebox{.75}{$>$}$}};
\node at (1.75,-2.25) {\rotatebox[origin=tr]{45}{$\scalebox{.75}{$>$}$}};
\node at (2.75,-.25) {\rotatebox[origin=tr]{45}{$\scalebox{.75}{$>$}$}};
\node at (2.25,-1.75) {\rotatebox[origin=tr]{45}{$\scalebox{.75}{$>$}$}};

  \fill (0,0)  circle[radius=1.5pt];
  \fill (1,0)  circle[radius=1.5pt];  
  \fill (2,0)  circle[radius=1.5pt];
  \fill (3,0)  circle[radius=1.5pt];
  \fill (4,0)  circle[radius=1.5pt];
  \fill (-1,0)  circle[radius=1.5pt];
  \fill (-2,0)  circle[radius=1.5pt];

 \node[red] at (2.1,2.2) {$\scalebox{.75}{$7$}$};
 \node[red] at (2.1,-2.3) {$\scalebox{.75}{$7$}$};
 \node[red] at (1.3,1.6) {$\scalebox{.75}{$8$}$};
 \node[red] at (-.3,0) {$\scalebox{.75}{$8$}$};
 \node[red] at (4.3,0) {$\scalebox{.75}{$3$}$};
 \node[red] at (2.7,-1.7) {$\scalebox{.75}{$3$}$};
 \node[red] at (.8,1.2) {$\scalebox{.75}{$8$}$};
  \node[red] at (1.7,0) {$\scalebox{.75}{$8$}$};
  \node[red] at (3.3,0) {$\scalebox{.75}{$6$}$};
  \node[red] at (2.5,-.8) {$\scalebox{.75}{$8$}$};
 \node[red] at (3.2,-1.2) {$\scalebox{.75}{$7$}$};
 \node[red] at (1.5,.8) {$\scalebox{.75}{$6$}$};
 \node[red] at (.7,0) {$\scalebox{.75}{$5$}$};
 \node[red] at (-1.3,0) {$\scalebox{.75}{$5$}$};
 \node[red] at (-2.3,0) {$\scalebox{.75}{$2$}$};
 \node[red] at (.7,-3.1) {$\scalebox{.75}{$1$}$};
 \node[red] at (.7,3.1) {$\scalebox{.75}{$1$}$};
 \node[red] at (1.7,2.7) {$\scalebox{.75}{$3$}$};
 \node[red] at (1.7,-2.7) {$\scalebox{.75}{$3$}$};

\node at (0,-1.2) {$\;$};

\end{tikzpicture}
\]
\[
\begin{tikzpicture}[x=.75cm, y=.75cm,
    every edge/.style={
        draw,
      postaction={decorate,
                    decoration={markings}
                   }
        }
]

\node at (-2.5,-.25) {$\scalebox{1}{$x_1^2=$}$};
 
\draw[thick] (1,1)--(2.5,-.5);
\draw[thick] (0,0)--(1.5,1.5)--(3,0)--(1.5,-1.5)--(0,0);
\draw[thick] (.5,.5)--(2,-1);
\draw[thick] (1,-2)--(1,-2.5);
\draw[thick] (1,2)--(1,2.5);
\draw[thick] (-1,0)--(1,2)--(1.5,1.5);
\draw[thick] (-1,0)--(1,-2)--(1.5,-1.5);
 
\node at (.75,1.65) {\rotatebox[origin=tr]{-135}{$\scalebox{.75}{$>$}$}};
\node at (1.25,1.75) {\rotatebox[origin=tr]{-45}{$\scalebox{.75}{$>$}$}};
\node at (.75,-1.75) {\rotatebox[origin=tr]{135}{$\scalebox{.75}{$>$}$}};
\node at (1.25,-1.75) {\rotatebox[origin=tr]{45}{$\scalebox{.75}{$>$}$}};
\node at (.25,.15) {\rotatebox[origin=tr]{-135}{$\scalebox{.75}{$>$}$}};
\node at (1.25,1.15) {\rotatebox[origin=tr]{-135}{$\scalebox{.75}{$>$}$}};
\node at (.75,.65) {\rotatebox[origin=tr]{-135}{$\scalebox{.75}{$>$}$}};
 \node at (1.75,1.25) {\rotatebox[origin=tr]{-45}{$\scalebox{.75}{$>$}$}};
\node at (1.25,.75) {\rotatebox[origin=tr]{-45}{$\scalebox{.75}{$>$}$}};
\node at (.75,.25) {\rotatebox[origin=tr]{-45}{$\scalebox{.75}{$>$}$}};

\node at (1.25,-1.25) {\rotatebox[origin=tr]{135}{$\scalebox{.75}{$>$}$}};
\node at (1.75,-.75) {\rotatebox[origin=tr]{135}{$\scalebox{.75}{$>$}$}};
\node at (2.25,-.25) {\rotatebox[origin=tr]{135}{$\scalebox{.75}{$>$}$}};
\node at (2.75,-.25) {\rotatebox[origin=tr]{45}{$\scalebox{.75}{$>$}$}};
\node at (2.25,-.75) {\rotatebox[origin=tr]{45}{$\scalebox{.75}{$>$}$}};
\node at (1.75,-1.25) {\rotatebox[origin=tr]{45}{$\scalebox{.75}{$>$}$}};

  \fill (-1,0)  circle[radius=1.5pt];
  \fill (0,0)  circle[radius=1.5pt];
  \fill (1,0)  circle[radius=1.5pt];  
  \fill (2,0)  circle[radius=1.5pt];
  \fill (3,0)  circle[radius=1.5pt];
 
   \node[red] at (.7,-2.2) {$\scalebox{.75}{$1$}$};
   \node[red] at (.7,2.2) {$\scalebox{.75}{$1$}$};
   \node[red] at (-1.3,0) {$\scalebox{.75}{$2$}$};

 \node[red] at (1.7,-1.7) {$\scalebox{.75}{$3$}$};
 \node[red] at (1.7,1.7) {$\scalebox{.75}{$3$}$};
 \node[red] at (-.3,0) {$\scalebox{.75}{$5$}$};
 \node[red] at (2.2,-1.2) {$\scalebox{.75}{$7$}$};
 \node[red] at (.8,1.2) {$\scalebox{.75}{$5$}$};
  \node[red] at (1.7,0) {$\scalebox{.75}{$5$}$};
  \node[red] at (3.3,0) {$\scalebox{.75}{$7$}$};
  \node[red] at (2.6,-.8) {$\scalebox{.75}{$3$}$};
 \node[red] at (.4,.8) {$\scalebox{.75}{$6$}$};
 \node[red] at (.7,0) {$\scalebox{.75}{$8$}$};

\node at (0,-1.2) {$\;$};
\end{tikzpicture}
\]
 \caption{The labellings of the generators of $M_1$.}\label{labelsGENM}
\end{figure}

 \begin{figure}
\phantom{This text will be invisible} 
\[
\begin{tikzpicture}[x=.75cm, y=.75cm,
    every edge/.style={
        draw,
      postaction={decorate,
                    decoration={markings}
                   }
        }
]

\node at (-1.5,-.25) {$\scalebox{1}{$w_0=$}$};

\draw[thick] (0,0)--(2,2)--(4,0);
\draw[thick] (0,0)--(2,-2)--(4,0);
\draw[thick] (1,0)--(1.5,0.5)--(2,0);
\draw[thick] (1.5,0.5)--(1,1);
\draw[thick] (3,0)--(1.5,1.5);
\draw[thick] (3,0)--(2.5,-.5); 
\draw[thick] (4,0)--(2,2);
\draw[thick] (4,0)--(3.5,-.5);
\draw[thick] (1,0)--(2.5,-1.5);
\draw[thick] (2,0)--(3,-1);

\draw[thick] (2,2)--(2,2.25);
\draw[thick] (2,-2)--(2,-2.25);

\node at (1.75,1.65) {\rotatebox[origin=tr]{-135}{$\scalebox{.75}{$>$}$}};
\node at (1.25,1.15) {\rotatebox[origin=tr]{-135}{$\scalebox{.75}{$>$}$}};
\node at (.75,.65) {\rotatebox[origin=tr]{-135}{$\scalebox{.75}{$>$}$}};
\node at (1.25,.75) {\rotatebox[origin=tr]{-45}{$\scalebox{.75}{$>$}$}};
\node at (2.25,1.75) {\rotatebox[origin=tr]{-45}{$\scalebox{.75}{$>$}$}};
\node at (1.75,1.25) {\rotatebox[origin=tr]{-45}{$\scalebox{.75}{$>$}$}};
\node at (1.75,.25) {\rotatebox[origin=tr]{-45}{$\scalebox{.75}{$>$}$}};
\node at (1.25,.15) {\rotatebox[origin=tr]{-135}{$\scalebox{.75}{$>$}$}};

\node at (1.75,-1.75) {\rotatebox[origin=tr]{135}{$\scalebox{.75}{$>$}$}};
\node at (2.25,-1.25) {\rotatebox[origin=tr]{135}{$\scalebox{.75}{$>$}$}};
\node at (2.25,-.25) {\rotatebox[origin=tr]{135}{$\scalebox{.75}{$>$}$}};
\node at (2.75,-.25) {\rotatebox[origin=tr]{45}{$\scalebox{.75}{$>$}$}};
\node at (3.25,-.75) {\rotatebox[origin=tr]{45}{$\scalebox{.75}{$>$}$}};
\node at (2.75,-.75) {\rotatebox[origin=tr]{135}{$\scalebox{.75}{$>$}$}};
\node at (2.75,-1.25) {\rotatebox[origin=tr]{45}{$\scalebox{.75}{$>$}$}};
\node at (2.25,-1.75) {\rotatebox[origin=tr]{45}{$\scalebox{.75}{$>$}$}};

  \fill (0,0)  circle[radius=1.5pt];
  \fill (1,0)  circle[radius=1.5pt];  
  \fill (2,0)  circle[radius=1.5pt];
  \fill (3,0)  circle[radius=1.5pt];
  \fill (4,0)  circle[radius=1.5pt];

 \node[red] at (1.8,2.1) {$\scalebox{.75}{$1$}$};
 \node[red] at (1.8,-2.1) {$\scalebox{.75}{$1$}$};
 \node[red] at (1.3,1.6) {$\scalebox{.75}{$2$}$};
 \node[red] at (-.3,0) {$\scalebox{.75}{$2$}$};
 \node[red] at (4.3,0) {$\scalebox{.75}{$3$}$};
 \node[red] at (2.7,-1.7) {$\scalebox{.75}{$3$}$};
 \node[red] at (.8,1.2) {$\scalebox{.75}{$6$}$};
  \node[red] at (1.7,0) {$\scalebox{.75}{$8$}$};
  \node[red] at (3.3,0) {$\scalebox{.75}{$7$}$};
  \node[red] at (2.5,-.8) {$\scalebox{.75}{$4$}$};
 \node[red] at (3.2,-1.2) {$\scalebox{.75}{$5$}$};
 \node[red] at (1.5,.8) {$\scalebox{.75}{$8$}$};
 \node[red] at (.7,0) {$\scalebox{.75}{$4$}$};

\node at (0,-1.2) {$\;$};
\end{tikzpicture}
\qquad 
\begin{tikzpicture}[x=.75cm, y=.75cm,
    every edge/.style={
        draw,
      postaction={decorate,
                    decoration={markings}
                   }
        }
]

\node at (-1.5,-.25) {$\scalebox{1}{$w_1=$}$};

\draw[thick] (0,0)--(2,2)--(4,0);
\draw[thick] (0,0)--(2,-2)--(4,0);

\draw[thick] (1,0)--(1.5,0.5)--(2,0);
\draw[thick] (2,0)--(3,-1);
\draw[thick] (1,0)--(.5,-.5);
\draw[thick] (1.5,.5)--(2,1)--(3,0);
\draw[thick] (3,0)--(3.5,-.5);
\draw[thick] (2,1)--(1.5,1.5);

\draw[thick] (2,2)--(2,2.25);
\draw[thick] (2,-2)--(2,-2.25);

\node at (0,-1.2) {$\;$};

\node at (1.75,1.65) {\rotatebox[origin=tr]{-135}{$\scalebox{.75}{$>$}$}};
\node at (1.75,.65) {\rotatebox[origin=tr]{-135}{$\scalebox{.75}{$>$}$}};
\node at (1.25,1.15) {\rotatebox[origin=tr]{-135}{$\scalebox{.75}{$>$}$}};
\node at (1.75,-1.75) {\rotatebox[origin=tr]{135}{$\scalebox{.75}{$>$}$}};
\node at (2.75,-.75) {\rotatebox[origin=tr]{135}{$\scalebox{.75}{$>$}$}};
\node at (3.25,-.25) {\rotatebox[origin=tr]{135}{$\scalebox{.75}{$>$}$}};
\node at (.25,-.25) {\rotatebox[origin=tr]{135}{$\scalebox{.75}{$>$}$}};
\node at (3.25,-.75) {\rotatebox[origin=tr]{45}{$\scalebox{.75}{$>$}$}};
\node at (3.75,-.25) {\rotatebox[origin=tr]{45}{$\scalebox{.75}{$>$}$}};
\node at (2.75,-.75) {\rotatebox[origin=tr]{135}{$\scalebox{.75}{$>$}$}};
\node at (.75,-.25) {\rotatebox[origin=tr]{45}{$\scalebox{.75}{$>$}$}};
\node at (2.25,-1.75) {\rotatebox[origin=tr]{45}{$\scalebox{.75}{$>$}$}};
\node at (2.25,1.75) {\rotatebox[origin=tr]{-45}{$\scalebox{.75}{$>$}$}};
\node at (2.25,.75) {\rotatebox[origin=tr]{-45}{$\scalebox{.75}{$>$}$}};
\node at (1.75,.25) {\rotatebox[origin=tr]{-45}{$\scalebox{.75}{$>$}$}};
\node at (1.25,.15) {\rotatebox[origin=tr]{-135}{$\scalebox{.75}{$>$}$}};
\node at (1.75,1.25) {\rotatebox[origin=tr]{-45}{$\scalebox{.75}{$>$}$}};

  \fill (0,0)  circle[radius=1.5pt];
  \fill (1,0)  circle[radius=1.5pt];  
  \fill (2,0)  circle[radius=1.5pt];
  \fill (3,0)  circle[radius=1.5pt];
  \fill (4,0)  circle[radius=1.5pt];

 \node[red] at (1.8,2.1) {$\scalebox{.75}{$1$}$};
 \node[red] at (1.8,-2.1) {$\scalebox{.75}{$1$}$};
 \node[red] at (1.3,1.6) {$\scalebox{.75}{$2$}$};
 \node[red] at (-.3,0) {$\scalebox{.75}{$6$}$};
 \node[red] at (4.3,0) {$\scalebox{.75}{$3$}$};
 \node[red] at (.3,-.7) {$\scalebox{.75}{$2$}$};
 \node[red] at (2.2,1.2) {$\scalebox{.75}{$7$}$};
  \node[red] at (1.7,0) {$\scalebox{.75}{$4$}$};
  \node[red] at (3.3,0) {$\scalebox{.75}{$4$}$};
  \node[red] at (3.7,-.75) {$\scalebox{.75}{$5$}$};
  \node[red] at (3.2,-1.2) {$\scalebox{.75}{$3$}$};
 \node[red] at (1.5,.8) {$\scalebox{.75}{$7$}$};
 \node[red] at (.7,0) {$\scalebox{.75}{$7$}$};
  
\end{tikzpicture}
\]

\[
\begin{tikzpicture}[x=.75cm, y=.75cm,
    every edge/.style={
        draw,
      postaction={decorate,
                    decoration={markings}
                   }
        }
]

\node at (-2.5,-.25) {$\scalebox{1}{$w_2=$}$};
 
\draw[thick] (-1,0)--(1.5,2.5)--(2,2);
\draw[thick] (-1,0)--(1.5,-2.5)--(2,-2);

\draw[thick] (0,0)--(2,2)--(4,0);
\draw[thick] (0,0)--(2,-2)--(4,0);

\draw[thick] (2,0)--(2.5,.5)--(3,0);

\draw[thick] (1,0)--(.5,.5);
\draw[thick] (2.5,.5)--(1.5,1.5);

\draw[thick] (3,0)--(3.5,-.5);
\draw[thick] (2,0)--(1.5,-.5)--(1,0);
\draw[thick] (1.5,-.5)--(2.5,-1.5);

\draw[thick] (1.5,2.5)--(1.5,2.75);
\draw[thick] (1.5,-2.5)--(1.5,-2.75);

\node at (1.75,1.65) {\rotatebox[origin=tr]{-135}{$\scalebox{.75}{$>$}$}};
\node at (2.25,.15) {\rotatebox[origin=tr]{-135}{$\scalebox{.75}{$>$}$}};
\node at (2.25,1.75) {\rotatebox[origin=tr]{-45}{$\scalebox{.75}{$>$}$}};
\node at (1.25,2.15) {\rotatebox[origin=tr]{-135}{$\scalebox{.75}{$>$}$}};
\node at (1.75,2.25) {\rotatebox[origin=tr]{-45}{$\scalebox{.75}{$>$}$}};
\node at (.75,.25) {\rotatebox[origin=tr]{-45}{$\scalebox{.75}{$>$}$}};
\node at (2.75,.25) {\rotatebox[origin=tr]{-45}{$\scalebox{.75}{$>$}$}};
\node at (.25,.15) {\rotatebox[origin=tr]{-135}{$\scalebox{.75}{$>$}$}};
\node at (1.75,1.25) {\rotatebox[origin=tr]{-45}{$\scalebox{.75}{$>$}$}};
\node at (1.25,1.15) {\rotatebox[origin=tr]{-135}{$\scalebox{.75}{$>$}$}};

\node at (1.75,-1.75) {\rotatebox[origin=tr]{135}{$\scalebox{.75}{$>$}$}};
\node at (2.25,-1.25) {\rotatebox[origin=tr]{135}{$\scalebox{.75}{$>$}$}};
\node at (3.25,-.25) {\rotatebox[origin=tr]{135}{$\scalebox{.75}{$>$}$}};
\node at (2.75,-1.25) {\rotatebox[origin=tr]{45}{$\scalebox{.75}{$>$}$}};
\node at (3.75,-.25) {\rotatebox[origin=tr]{45}{$\scalebox{.75}{$>$}$}};
\node at (1.25,-.25) {\rotatebox[origin=tr]{135}{$\scalebox{.75}{$>$}$}};
\node at (1.25,-2.25) {\rotatebox[origin=tr]{135}{$\scalebox{.75}{$>$}$}};
\node at (1.75,-2.25) {\rotatebox[origin=tr]{45}{$\scalebox{.75}{$>$}$}};
\node at (1.75,-.25) {\rotatebox[origin=tr]{45}{$\scalebox{.75}{$>$}$}};
\node at (2.25,-1.75) {\rotatebox[origin=tr]{45}{$\scalebox{.75}{$>$}$}};

  \fill (0,0)  circle[radius=1.5pt];
  \fill (1,0)  circle[radius=1.5pt];  
  \fill (2,0)  circle[radius=1.5pt];
  \fill (3,0)  circle[radius=1.5pt];
  \fill (4,0)  circle[radius=1.5pt];
  \fill (-1,0)  circle[radius=1.5pt];

 \node[red] at (1.2,2.6) {$\scalebox{.75}{$1$}$};
 \node[red] at (1.2,-2.6) {$\scalebox{.75}{$1$}$};
  \node[red] at (2.2,2.2) {$\scalebox{.75}{$3$}$};
 \node[red] at (1.3,1.7) {$\scalebox{.75}{$4$}$};
 \node[red] at (.4,.8) {$\scalebox{.75}{$8$}$};
 \node[red] at (-.3,0) {$\scalebox{.75}{$4$}$};
   \node[red] at (1.7,0) {$\scalebox{.75}{$7$}$};
  \node[red] at (3.3,0) {$\scalebox{.75}{$4$}$};
 \node[red] at (.7,0) {$\scalebox{.75}{$8$}$};
 \node[red] at (-1.4,0) {$\scalebox{.75}{$2$}$};
 \node[red] at (4.3,0) {$\scalebox{.75}{$5$}$};
 \node[red] at (2.6,.8) {$\scalebox{.75}{$7$}$};
  \node[red] at (2.7,-1.7) {$\scalebox{.75}{$5$}$};
  \node[red] at (2.2,-2.2) {$\scalebox{.75}{$3$}$};
  \node[red] at (3.7,-.75) {$\scalebox{.75}{$3$}$};
  \node[red] at (1.3,-.75) {$\scalebox{.75}{$4$}$};

\node at (0,-1.2) {$\;$};
\end{tikzpicture}
\qquad 
\begin{tikzpicture}[x=.75cm, y=.75cm,
    every edge/.style={
        draw,
      postaction={decorate,
                    decoration={markings}
                   }
        }
]

\node at (-3.5,-.25) {$\scalebox{1}{$w_3=$}$};
 
\draw[thick] (-1,0)--(1.5,2.5)--(2,2);
\draw[thick] (-1,0)--(1.5,-2.5)--(2,-2);

\draw[thick] (-2,0)--(1,3)--(1.5,2.5);
\draw[thick] (-2,0)--(1,-3)--(1.5,-2.5);

\draw[thick] (0,0)--(2,2)--(4,0);
\draw[thick] (0,0)--(2,-2)--(4,0);
\draw[thick] (1,0)--(1.5,0.5)--(2,0);
\draw[thick] (1.5,0.5)--(1,1);
\draw[thick] (3,0)--(1.5,1.5);
\draw[thick] (3,0)--(2.5,-.5); 
\draw[thick] (4,0)--(2,2);
\draw[thick] (4,0)--(3.5,-.5);
\draw[thick] (1,0)--(2.5,-1.5);
\draw[thick] (2,0)--(3,-1);

\draw[thick] (1,3)--(1,3.25);
\draw[thick] (1,-3)--(1,-3.25);

\node at (1.75,1.65) {\rotatebox[origin=tr]{-135}{$\scalebox{.75}{$>$}$}};
\node at (2.25,1.75) {\rotatebox[origin=tr]{-45}{$\scalebox{.75}{$>$}$}};
\node at (1.25,2.15) {\rotatebox[origin=tr]{-135}{$\scalebox{.75}{$>$}$}};
\node at (1.75,2.25) {\rotatebox[origin=tr]{-45}{$\scalebox{.75}{$>$}$}};
\node at (.75,2.65) {\rotatebox[origin=tr]{-135}{$\scalebox{.75}{$>$}$}};
\node at (1.25,2.75) {\rotatebox[origin=tr]{-45}{$\scalebox{.75}{$>$}$}};
\node at (.75,.65) {\rotatebox[origin=tr]{-135}{$\scalebox{.75}{$>$}$}};
\node at (1.25,.15) {\rotatebox[origin=tr]{-135}{$\scalebox{.75}{$>$}$}};
\node at (1.75,.25) {\rotatebox[origin=tr]{-45}{$\scalebox{.75}{$>$}$}};
\node at (1.25,.75) {\rotatebox[origin=tr]{-45}{$\scalebox{.75}{$>$}$}};
\node at (1.75,1.25) {\rotatebox[origin=tr]{-45}{$\scalebox{.75}{$>$}$}};
\node at (1.25,1.15) {\rotatebox[origin=tr]{-135}{$\scalebox{.75}{$>$}$}};

\node at (2.25,-.25) {\rotatebox[origin=tr]{135}{$\scalebox{.75}{$>$}$}};
\node at (1.75,-1.75) {\rotatebox[origin=tr]{135}{$\scalebox{.75}{$>$}$}};
\node at (2.25,-1.25) {\rotatebox[origin=tr]{135}{$\scalebox{.75}{$>$}$}};
\node at (2.75,-.75) {\rotatebox[origin=tr]{135}{$\scalebox{.75}{$>$}$}};
\node at (.75,-2.75) {\rotatebox[origin=tr]{135}{$\scalebox{.75}{$>$}$}};
\node at (2.75,-1.25) {\rotatebox[origin=tr]{45}{$\scalebox{.75}{$>$}$}};
\node at (3.25,-.75) {\rotatebox[origin=tr]{45}{$\scalebox{.75}{$>$}$}};
\node at (1.25,-2.25) {\rotatebox[origin=tr]{135}{$\scalebox{.75}{$>$}$}};
\node at (1.25,-2.75) {\rotatebox[origin=tr]{45}{$\scalebox{.75}{$>$}$}};
\node at (1.75,-2.25) {\rotatebox[origin=tr]{45}{$\scalebox{.75}{$>$}$}};
\node at (2.75,-.25) {\rotatebox[origin=tr]{45}{$\scalebox{.75}{$>$}$}};
\node at (2.25,-1.75) {\rotatebox[origin=tr]{45}{$\scalebox{.75}{$>$}$}};

  \fill (0,0)  circle[radius=1.5pt];
  \fill (1,0)  circle[radius=1.5pt];  
  \fill (2,0)  circle[radius=1.5pt];
  \fill (3,0)  circle[radius=1.5pt];
  \fill (4,0)  circle[radius=1.5pt];
  \fill (-1,0)  circle[radius=1.5pt];
  \fill (-2,0)  circle[radius=1.5pt];

 \node[red] at (2.1,2.2) {$\scalebox{.75}{$5$}$};
 \node[red] at (2.1,-2.3) {$\scalebox{.75}{$5$}$};
 \node[red] at (1.3,1.6) {$\scalebox{.75}{$4$}$};
 \node[red] at (-.3,0) {$\scalebox{.75}{$4$}$};
 \node[red] at (4.3,0) {$\scalebox{.75}{$3$}$};
 \node[red] at (2.7,-1.7) {$\scalebox{.75}{$3$}$};
 \node[red] at (.8,1.2) {$\scalebox{.75}{$8$}$};
  \node[red] at (1.7,0) {$\scalebox{.75}{$8$}$};
  \node[red] at (3.3,0) {$\scalebox{.75}{$7$}$};
  \node[red] at (2.5,-.8) {$\scalebox{.75}{$4$}$};
 \node[red] at (3.2,-1.2) {$\scalebox{.75}{$5$}$};
 \node[red] at (1.5,.8) {$\scalebox{.75}{$8$}$};
 \node[red] at (.7,0) {$\scalebox{.75}{$4$}$};
 \node[red] at (-1.3,0) {$\scalebox{.75}{$4$}$};
 \node[red] at (-2.3,0) {$\scalebox{.75}{$2$}$};
 \node[red] at (.7,-3.1) {$\scalebox{.75}{$1$}$};
 \node[red] at (.7,3.1) {$\scalebox{.75}{$1$}$};
 \node[red] at (1.7,2.7) {$\scalebox{.75}{$3$}$};
 \node[red] at (1.7,-2.7) {$\scalebox{.75}{$3$}$};

\node at (0,-1.2) {$\;$};

\end{tikzpicture}
\]
\[
\begin{tikzpicture}[x=.75cm, y=.75cm,
    every edge/.style={
        draw,
      postaction={decorate,
                    decoration={markings}
                   }
        }
]

\node at (-2.5,-.25) {$\scalebox{1}{$\sigma(x_1^2)=$}$};
 

\draw[thick] (-1,0)--(1,2)--(3,0)--(1,-2)--(-1,0);
\draw[thick] (0.5,-1.5)--(2,0)--(0.5,1.5);
\draw[thick] (-.5,-.5)--(1,1);
\draw[thick] (0,-1)--(1.5,.5);

\draw[thick] (1,-2)--(1,-2.5);
\draw[thick] (1,2)--(1,2.5);

\node at (.75,1.65) {\rotatebox[origin=tr]{-135}{$\scalebox{.75}{$>$}$}};
\node at (1.25,1.75) {\rotatebox[origin=tr]{-45}{$\scalebox{.75}{$>$}$}};
\node at (.75,-1.75) {\rotatebox[origin=tr]{135}{$\scalebox{.75}{$>$}$}};
\node at (1.25,-1.75) {\rotatebox[origin=tr]{45}{$\scalebox{.75}{$>$}$}};
\node at (.25,1.15) {\rotatebox[origin=tr]{-135}{$\scalebox{.75}{$>$}$}};
\node at (1.25,.15) {\rotatebox[origin=tr]{-135}{$\scalebox{.75}{$>$}$}};
\node at (.75,.65) {\rotatebox[origin=tr]{-135}{$\scalebox{.75}{$>$}$}};
 \node at (.75,1.25) {\rotatebox[origin=tr]{-45}{$\scalebox{.75}{$>$}$}};
\node at (1.25,.75) {\rotatebox[origin=tr]{-45}{$\scalebox{.75}{$>$}$}};
\node at (1.75,.25) {\rotatebox[origin=tr]{-45}{$\scalebox{.75}{$>$}$}};

\node at (.25,-1.25) {\rotatebox[origin=tr]{135}{$\scalebox{.75}{$>$}$}};
\node at (-.25,-.75) {\rotatebox[origin=tr]{135}{$\scalebox{.75}{$>$}$}};
\node at (-.75,-.25) {\rotatebox[origin=tr]{135}{$\scalebox{.75}{$>$}$}};
\node at (-.25,-.25) {\rotatebox[origin=tr]{45}{$\scalebox{.75}{$>$}$}};
\node at (.25,-.75) {\rotatebox[origin=tr]{45}{$\scalebox{.75}{$>$}$}};
\node at (.75,-1.25) {\rotatebox[origin=tr]{45}{$\scalebox{.75}{$>$}$}};

  \fill (-1,0)  circle[radius=1.5pt];
  \fill (0,0)  circle[radius=1.5pt];
  \fill (1,0)  circle[radius=1.5pt];  
  \fill (2,0)  circle[radius=1.5pt];
  \fill (3,0)  circle[radius=1.5pt];
 
   \node[red] at (.7,-2.2) {$\scalebox{.75}{$1$}$};
   \node[red] at (.7,2.2) {$\scalebox{.75}{$1$}$};
   \node[red] at (-1.3,0) {$\scalebox{.75}{$6$}$};

 \node[red] at (.3,-1.7) {$\scalebox{.75}{$2$}$};
 \node[red] at (.25,1.7) {$\scalebox{.75}{$2$}$};
 \node[red] at (-.3,0) {$\scalebox{.75}{$7$}$};
 \node[red] at (-.3,-1.2) {$\scalebox{.75}{$6$}$};
 \node[red] at (1.2,1.2) {$\scalebox{.75}{$7$}$};
  \node[red] at (1.7,0) {$\scalebox{.75}{$7$}$};
  \node[red] at (3.3,0) {$\scalebox{.75}{$3$}$};
  \node[red] at (-.7,-.8) {$\scalebox{.75}{$2$}$};
 \node[red] at (1.7,.8) {$\scalebox{.75}{$4$}$};
 \node[red] at (.7,0) {$\scalebox{.75}{$8$}$};

\node at (0,-1.2) {$\;$};
\end{tikzpicture}
\]
 \caption{The labellings of the generators of $M_2$.}\label{labelsGENM2}
\end{figure}

As in the case of $M_0$, we begin proving that adding a positive element of length $2$ to $M_1$ either doesn't change it (Lemma \ref{lemma-squared-bis-1}) or gives the whole $K_{(2,2)}$ (Lemma \ref{lemma-length-two}).

\begin{lemma}\label{lemma-squared-bis-1}
The groups $\langle x_{2k+1}^2,\CF\rangle$ are all equal to $M_1$.
\end{lemma}
\begin{proof}
Denote by $R_{2i+1}$ the group $\langle x_{2i+1}^2,\CF\rangle$. Since $\varphi(\CF)\subset \CF$, it holds $\varphi^{2i}(R_1)=\varphi^{2i}(M_1)=\varphi^{2i}(\langle x_{1}^2,\CF\rangle)=\langle x_{2i+1}^2,\varphi^{2i}(\CF)\rangle\subset R_{2i+1}$. As $x_2x_3$, $x_3^2\in R_1=M_1$, we have $x_{2i+2}x_{2i+3}$, $x_{2i+3}^2\in R_{2i+1}$. In particular, $R_{2i+3}\leq R_{2i+1}$. 
We want to prove that the converse inclusion holds. 
First, notice that $w_1x_{2i+2}x_{2i+3}w_1^{-1}=x_{2i}x_{2i+1}\in R_{2i+1}$ for all $i\geq 1$,
where $w_1=x_0x_1^2x_0^{-1}$.
Therefore, we have $\varphi^{2i-1}(w_0)x_{2i+1}^2(x_{2i}x_{2i+1})^{-1}=x_{2i-1}^2\in R_{2i+1}$
where 
$w_0=x_0^2x_1x_2^{-1}$.
This means that $R_{2i-1}\leq R_{2i+1}$ and we are done.
\end{proof}

\begin{lemma}\label{lemma-length-two} 
For any $g=x_ix_j\in K_{(2,2)}\setminus M_1$, the subgroup $\langle g,M_1\rangle$ is equal to $K_{(2,2)}$.
\end{lemma}
\begin{proof}
We divide the proof into a series of cases.

\noindent
\textbf{Case 1}: $g=x_{2i}^2$ for $i\geq 0$.\\
In this case the claim follows at once from Lemmas \ref{M0even} and \ref{lemmaH1}.

Now we observe that
if $K_{(2,2)}\subseteq\langle g, M_0\rangle$, then $K_{(2,2)}\subseteq\langle \varphi(g), M_1\rangle$. Indeed, we have
 $\varphi(K_{(2,2)})\subseteq\varphi(\langle g, M_0\rangle)=\langle \varphi(g), \varphi(M_0)\rangle\leq \langle \varphi(g), M_1\rangle$
by Lemma \ref{lemma-M}.
In particular, $x_2^2\in \langle \varphi(g), M_1\rangle$ and by Lemmas \ref{M0even} and \ref{lemmaH1} we are done.

\noindent
\textbf{Case 2}: $g=x_{2i+1}x_{2i+2}$ for any $i\geq 0$.\\
The claim  follows from Lemmas 
\ref{lemmaH1} and  \ref{lemma-zero-j}.

\noindent
\textbf{Case 3}:  $g=x_{2i+1}x_{2j+2}$	for any $j\geq i+1\geq 0$, $i\geq 0$.\\
The claim follows from Lemmas 
\ref{lemmaH1} 
and  \ref{lemma-zero-j}.

\noindent
\textbf{Case 4}: $g=x_{2i+1}x_{2j+1}$ for any $j\geq i+1$, $i\geq 0$.\\
The claim follows from Lemmas 
\ref{lemmaH1}  and  \ref{lemma-zero-j}.

\noindent
\textbf{Case 5}: $g=x_{2i+2}x_{2j+1}$		for any $j> i+1$, $i\geq 0$.\\
The claim follows from Lemma \ref{lemmaH1}. 

\noindent
\textbf{Case 6}: $g=x_{2i+2}x_{2j+2}$		for any $j\geq i\geq 0$.\\
The claim follows from Lemma 
\ref{lemmaH1}. 

\end{proof}

The following lemma is instrumental in the proof of Theorem \ref{thmM1} for $M_1$ and $M_2$.
\begin{lemma}
\label{positivity2}
For any $g\in 
K_{(2,2)}\setminus M_1$, the double coset $M_1gM_1$ contains a positive element $w$ such that for every $w'\in M_1gM_1$, it holds $|w|\leq |w'|$.
Moreover,  $w$ may be chosen such that it does not contain any block
and that it lies in $\varphi(F_+)$. 
\end{lemma}
\begin{proof}
Consider the normal form of $g=x_0^{a_0}x_1^{a_1}\cdots x_1^{-b_1}x_0^{-b_0}$. 
In the first step of this proof we want to obtain an element in $\varphi(F)\cap M_1gM_1$. 
There are two cases to deal with: $a_0$ and $b_0$ are both even or odd.
In the first case take the element $h^{-a_0/2}gh^{b_0/2}$, where $h:=w_0x_2x_3=x_0^2x_1x_3\in M_1$, where $w_0=x_0^2x_1x_2^{-1}$. This element has the same length as $g$ and is in $\varphi(F)$.
In the second case, we observe that $g=x_0^{a_0}\varphi(\tilde{g})x_0^{-b_0}$, where $\tilde{g}:=x_0^{a_1}x_1^{a_2}\cdots x_1^{-b_1}x_0^{-b_1}$. Now take the element
\begin{align*}
h^{-[a_0/2]-1}gh^{[b_0/2]+1}&=x_3^{-1}x_1^{-1}(x_0^{-1}\varphi(\tilde{g})x_0)x_1x_3\\
&=x_3^{-1}x_1^{-1}\varphi^2(\tilde{g})x_1x_3\\
&=x_3^{-1}(\varphi(x_0^{-1}\varphi(\tilde{g})x_0))x_3\\
&=x_3^{-1}\varphi^3(\tilde{g})x_3
\end{align*}
where we used that $x_0^{-1}\varphi(x)x_0=\varphi^2(x)$ for all $x\in F$, \cite[p. 29]{B}.
The generators $x_3^{\pm 1}$ may or may not appear in the normal form of $\varphi^3(\tilde{g})$. 
If they appear, the element $x_3^{-1}\varphi^3(\tilde{g})x_3$ is shorter than $g$ and in $\varphi(F)$.
Otherwise, we have $h^{-[a_0/2]-1}gh^{[b_0/2]+1}=\varphi^4(\tilde{g})\in \varphi(F)$.
 
Now take $w\in M_1gM_1$ of minimal length. By the previous discussion we may assume that $w\in \varphi(F)$.
If it is positive, we are done. Otherwise, we consider its normal form 
$$
w=x_{0}^{a_0}\cdots x_{n}^{a_n}x_{n}^{-b_n}\cdots x_{0}^{-b_0}\; .
$$ 
Let $x_{i_0}^{-1}$ be the last non-zero factor. If $i_0\in 2\IN_0+1$, the element $wx_{i_0}^2\in M_1gM_1$ admits a normal form of the same length, but with less negative factors.
Similarly, when $i_0\in 2\IN$, one considers the element $wx_{i_0}x_{i_0+1}\in M_1gM_1$. The claim follows by iteration.

Now we want to show that such an element $w$ 
 in $\varphi(F_+)$ does not contain a block. 
Suppose instead that the normal form of $w$ contains a block, that is $w=z_1Bz_2$, with $B$ being a minimal block.

We now show that we can replace $w$ with another element in $M_1gM_1\cap \varphi(F)$ of the same length and such that $w'=B'z_2'$, with $B'$ being a translation of $B$. 
If $z_1=\emptyset$, there is nothing to do.
If $z_1$ is non-empty, then $z_1=x_jz''_1$.  
If $j$ is even, consider $z'':=(x_jx_{j+1})^{-1}w=x_{j+1}^{-1}z''_1Bz_2$ (recall that $x_jx_{j+1}\in M_1$).
By Lemma \ref{lemma-block-prop} (5) and the minimality of $w$ we get $z''=z_1'B'z_2'$, where $|z_1'|=|z_1''|=|z_1|-1$, $|z_2'|=|z_2|+1$, $B'$ is a translation of $B$. If $j$ is odd, consider 
$x_j^{-2}w=x_{j}^{-1}z''_1Bz_2$, $x_j^{2}\in M_1$ and argue as before. By iteration we may assume that $z_1$ is empty. 

By the previous discussion we may assume that $w=Bz_2$, where $B=x_{i_1}\cdots x_{i_n}$. 
There are two cases depending on whether $i_1$ is even or odd. 

In the first case consider $t_1=(x_{i_1}x_{i_1+1})^{-1}w=x_{i_1+1}^{-1}x_{i_2}\cdots x_{i_n}$.
By Lemma \ref{lemma-block-prop} (6) 
 there exists a $j$ such that $i_j=i_1+j-1$. This means that $x_{i_1+1}^{-1}$
cancel the first occurrence of $x_{i_j}$ in $t$. This is in contradiction with our hypothesis of $w$ being of minimal length and we are done.

We want to show that the second case ($i_1$ odd) cannot occur. 
Since $w$ is of minimal length, $i_2\geq i_1+1$ (if $i_2=i_1$, then the element $x_{i_1}^{-2}w\in M_1gM_1\cap \varphi(F)$ is shorter than $w$).
We assumed that $B$ is a minimal block, however $\tilde{B}:=x_{i_2}\cdots x_{i_n}$ is a block. Indeed, set $i'_j:=i_{j+1}$ for $j=1, \ldots , n-1$ and consider $\tilde{B}=x_{i'_1}\cdots x_{i'_{n-1}}$. We   have to check that $i'_j<i'_1+j$ for all $j=1, \ldots , n-1$.
By definition we have $i'_j:=i_{j+1}<i_1+j+1=i_2+j=i'_1+j$.
If we show that $\tilde B$ contains at least two different letters, then we reached a contradiction.
As $B$ is a block 
 $i_3\geq i_2=i_1+1$ and $i_3<i_1+3$. 
Therefore, we have two sub-cases: $i_3=i_1+2$ and $i_3=i_2=i_1+1$. In the first we 
found that there are at least two different letters and, thus, $\tilde B$ is a block.
In the second sub-case 
the element 
\begin{align*}
\varphi^{i_1}(w_1)^{-1}w&=x_{i_1}x_{i_1+1}^{-2}x_{i_1}^{-1}Bz_2=x_{i_1}x_{i_1+1}^{-2}x_{i_1}^{-1}x_{i_1}x_{i_1+1}x_{i_1+1}x_{i_4}\cdots x_{i_n}z_2\\
&=x_{i_1}x_{i_4}\cdots x_{i_n}w_2\in M_1gM_1
\end{align*} 
 is shorter than $w$ (recall that $w_1=x_0x_1^{2}x_0^{-1}$) and we are done. 
\end{proof}
We are now ready to prove that $M_1$ and $M_2$ are maximal subgroups of the rectangular subgroup $K_{(2,2)}$.
\begin{proof}[Proof of Theorem \ref{thmM1} for $M_1$ and $M_2$]
Let $g\in K_{(2,2)}\setminus M_1$. We need to show that $\langle g, M_1\rangle =K_{(2,2)}$.
Note that $\langle g, M_1\rangle =\langle g', M_1\rangle$ for any $g'\in M_1 gM_1$.
Therefore, we may replace $g$ with any element in $M_1gM_1$. By Lemma \ref{positivity2} we may suppose that $g$ 
 does not contain any block and is an element of minimal length in $M_1gM_1\cap\varphi(F_+)$.

We  give a proof by induction on the length of the normal form of $g$ (which is even because $g$ is in $K_{(2,2)}$).
If the length is $2$, then the claim is exactly the content of
 Lemma \ref{lemma-length-two}. 

Suppose that the length is bigger than $2$. 
 We have $g=w'x_j^k$ with $j\in \IN$ and $w'=x_{i_1}\cdots x_{i_m}$ is either empty or its last letter is not $x_j$. If $w'$ is empty (in this case this implies that $k\in 2\IN$, $j\in 2\IN_0$), then $x_j^{-k}(x_j x_{j+1}) x_j^k=x_jx_{j+1+k}\in \langle M_1, x_j^k\rangle$ and by Lemma \ref{lemma-length-two} we are done.

Suppose that $w'$ is non-empty and let $m=|w'|$. 
Now we have two cases: (1) $j$ is even; (2) $j$ is odd and $k=1$. Without loss of generality, we may suppose that $j\geq m=|w'|$ (it suffices to replace $w$ by $x_1^{-2l}wx_1^{2l}$ with $l$ big enough). In this case it holds $x_{j-m}w'=w'x_j$ if $j\geq m$ (\cite[Formula ($\star$) in proof of  Theorem 3.12]{GS2}).
For case (1), there are two sub-cases: $k$ and $m$ are odd, $k$ and $m$ are even.
If $k$ and $m$ are odd ($j$ is even) take the element 
\begin{align*}
g^{-1}x_{j-m}^2g& = x_j^{-k}w'^{-1}	x_{j-m}^2w'x_j^k	\\
&= x_j^{-k}w'^{-1}w'x_j^2x_j^k\\
&=x_{j}^2 \in \langle g, M_1\rangle
\end{align*}
which contains $K_{(2,2)}$ by 
Lemma \ref{lemma-length-two}.\\
If $k$ and $m$ are even  ($j$ is even) take the element 
\begin{align*}
g^{-1}(x_{j-m}x_{j-m+1})g& = x_j^{-k}w'^{-1}	x_{j-m}x_{j-m+1}w'x_j^k	\\
&= x_j^{-k}w'^{-1}w'x_{j}x_{j+1}x_j^k\\
&= x_j^{-k}x_{j}x_{j+1}x_j^k\\
&= x_{j}x_{j+1+k} \in \langle g, M_1\rangle
\end{align*}
which contains $K_{(2,2)}$ by 
Lemma \ref{lemma-length-two}.\\

For case (2), that is  $j$ is odd and $k=1$, we observe that $m$ is odd and consider the element  
\begin{align*}
g^{-1}(x_{j-m}x_{j-m+1})g&=x_j^{-1}w'^{-1}(x_{j-m}x_{j-m+1})w'x_j\\
&=x_j^{-1}w'^{-1}w'x_jx_{j+1}x_j\\
&=x_{j+1}x_j=x_jx_{j+2}\in M_1gM_1
\end{align*}
Now the claim follows from Lemma \ref{lemma-length-two}.

\end{proof}
\begin{corollary}\label{cor42}
The subgroups $\theta^{-1}(M_1)$ and  $\theta^{-1}(M_2)$  are maximal infinite index subgroups of $F$. 
\end{corollary}
\begin{remark} 
Like $M_0$, the subgroups $M_1$ and $M_2$ are closed, that is $Cl(M_1)=M_1$ and $Cl(M_2)=M_2$. 
\end{remark}

\begin{theorem}\label{teo43}
The index of $\CF$ in $M_0$, $M_1$ and $M_2$ is infinite.
\end{theorem}
In order to prove this theorem we need the following lemma.
\begin{lemma}\label{lemma601}
The groups $\langle x_j^{2k},\CF\rangle$ are all equal to $M_0$ or $M_1$.
\end{lemma}
\begin{proof}
 Clearly, $\langle x_j^{2k},\CF\rangle$ is contained in $M_0$ (if $j$ is even) or $M_1$ (if $j$ is odd). 
Since $x_j^{-2k}\varphi^j(w_1)x_j^{2k}=x_{j+2k}^2$, we have that $\langle x_{2k}^2,\CF\rangle\leq \langle x_j^{2k},\CF\rangle$. By Lemmas \ref{M0even} and \ref{lemma-squared-bis-1} we are done.
\end{proof}

\begin{proof}[Proof of Theorem \ref{teo43}]
First of all, we observe that it is enough to calculate $|\CF:M_0|$ and $|\CF: M_1|$.
Suppose that $|\CF:M_0|=n<\infty$, that is $M_0=\cup_{k=1}^n g_k\CF$ for some $g_1=1,\ldots , g_n\in M_0$.
We claim that this implies that for every $g\in M_0$ there exist infinitely many $m\in \IN$ such that $g^m\in\CF$. Indeed, there are at least two distinct indices $i$, $j\in\IN$ such that $g^i, g^j\in g_k\CF$ for some $k\in \{1, \ldots , n\}$. This means that $g^{i-j}\in \CF$ and $g^{(i-j)m}\in\CF$ for all $m\in\IN$.
Take $g=x_0^2\in M_0\setminus \CF$. It follows from Lemma \ref{lemma601} that the element $g^k$ does not belong to $\CF$ for all $k\in\IN$.  Therefore, the index of $\CF$ in $M_0$ is infinite.

For $M_1$ use the same argument with $g=x_1^2$ and Lemma \ref{lemma601}.
\end{proof}

\section{On maximal infinite index subgroups of $F$}\label{sec5}
We recall that the \textbf{oriented subgroup} $\vec{F}$ is the subgroup of $F$ generated by $x_0x_1$, $x_1x_2$, $x_2x_3$. 
It can be easily seen that $\vec{F}$ is a subgroup of $K_{(1,2)}$. 
As mentioned in Section \ref{sec4}, the subgroup $K_{(1,2)}$ is isomorphic to $F$ and   an isomorphism is provided by the map
$\beta: \;  F\to K_{(1,2)}$, which is defined as
$\beta(x_i):=x_ix_2$ for $i=0, 1$.
The subgroup $\beta^{-1}(\vec{F})$ is the first example of a maximal subgroup of infinite index in $F$ without fixed points in the open unit interval $(0,1)$. 
For further information on $\vec{F}$, we refer to \cite{Jo14,GS, GS2, Ren, A, RS, AB1}.

In this section we compare the subgroups $\theta^{-1}(M_0)$, $\theta^{-1}(M_1)$, $\theta^{-1}(M_2)$ 
to   maximal infinite index subgroups of $F$ that have been identified before: the oriented subgroup $\vec{F}$; the parabolic subgroups $\stab(t)$ for $t\in (0,1)$;
 Golan's examples \cite[Examples 10.12 and 10.13, Section 10.3.B]{G} $K_1:=\langle H, x_1^2x_2^{-1} \rangle$, $K_2:=\langle H, x_1^2x_2x_1^{-3},x_1^3x_2x_1^{-4} \rangle$, where $H:=\langle x_0, x_1x_2x_1^{-1}\rangle$, and \\ 
$
K_3:=\langle x_0,x_1x_2x_1^{-3},x_1x_2x_3x_2^{-3}x_1^{-1}\rangle\; .
$ 
Note that $K_3$ is the first known example of a maximal  infinite index subgroup of $F$ which acts transitively on the set of dyadic rationals. 

We mention that $K_1$ and $K_2$ are concrete realisations of Golan and Sapir's implicit example of maximal infinite index subgroup containing $H$ described in \cite{GS2}. 
In fact, the subgroups $M_0$, $M_1$ and $M_2$ are distinct from all the   maximal infinite index subgroups containing $H$.
\begin{theorem}\label{distinct}
The subgroup $\theta^{-1}(M_0)$, $\theta^{-1}(M_1)$, $\theta^{-1}(M_2)$ are distinct from the parabolic subgroups, from $\beta^{-1}(\vec{F})$,  from
$K_1$, $K_2$, and $K_3$.
\end{theorem}
\begin{proof}
First of all we show that $\theta^{-1}(M_0)$ does not stabilise any number in $[1/2,1)$. The element $\sigma(x_1)=x_0x_1x_0^{-2}\in \stab(t)$ for $t\in [1/2,1)$. 
The following computations show that $\sigma(x_1)\not\in \theta^{-1}(M_0)$ and, thus, $\theta^{-1}(M_0)\neq \stab(t)$ for all $t\in [1/2,1)$
\begin{align*}
\theta(\sigma(x_1))&=(x_0x_1x_4^{-1}x_0^{-3})(x_0x_1^2x_0^{-3})(x_0^3x_4x_1^{-1}x_0^{-1})(x_0^3x_4x_1^{-1}x_0^{-1})\\
&=x_0x_1x_4^{-1}x_0^{-2}(x_1^2x_4)x_1^{-1}x_0^2x_4x_1^{-1}x_0^{-1}\\
&=x_0x_1x_4^{-1}x_0^{-2}(x_1x_3x_0^2)x_4x_1^{-1}x_0^{-1}\\
&=x_0x_1x_4^{-1}(x_3x_5)x_4x_1^{-1}x_0^{-1}\\
&=x_0(x_1x_3x_4)x_1^{-1}x_0^{-1}\\
&=x_0x_2x_3x_0^{-1}=x_1x_2\not\in M_0
\end{align*}

The element $x_2=x_0^{-1}x_1x_0\in \stab(t)$ for all $t\in (0,3/4]$. Since
\begin{align*}
\theta(x_2)&=(x_0^3x_4x_1^{-1}x_0^{-1})(x_0x_1^2x_0^{-3})(x_0x_1x_4^{-1}x_0^{-3})\\
&=x_0^3x_4x_1(x_0^{-2}x_1)x_4^{-1}x_0^{-3}\\
&=x_0^3x_4x_1x_3x_0^{-2}x_4^{-1}x_0^{-3}\\
&=x_0^3x_4x_1x_3x_6^{-1}x_0^{-5}\\
&=x_0^3x_1x_5x_3x_6^{-1}x_0^{-5}\\
&=x_0^3x_1x_3x_6x_6^{-1}x_0^{-5}\\
&=x_0^3x_1x_3x_0^{-5}
\end{align*}
we see that $\theta(x_2)\in M_0$ if and only if $x_0^{-4}\theta(x_2)x_0^6=x_2x_4\in M_0$. By Lemma 
\ref{lemma-length-two} we know that $x_2x_4\not\in M_0$ and, thus, $\theta^{-1}(M_0)$ does not stabilise any number in $(0,3/4]$. In particular, $\theta^{-1}(M_0)$ is not a parabolic subgroup.

We now show that $\theta^{-1}(M_0)$ does not coincide with the maximal subgroup $\beta^{-1}(\vec{F})$ exhibited in \cite{GS2}. 
Recall that $x_0x_1$ is one of the generators of $\vec{F}$, \cite{GS}. 
It was shown in the proof of \cite[Theorem 3.15]{GS2} that $\beta^{-1}(x_0x_1)=x_0x_1x_2^{-1}=x_0x_1x_0^{-1}x_1x_0$.
We have that 
 \begin{align*}
\theta(x_0x_1x_2^{-1})&=(x_0x_1x_4^{-1}x_0^{-3})(x_0x_1^2x_0^{-3})(x_0^3x_4x_1^{-1}x_0^{-1})(x_0x_1^2x_0^{-3})(x_0x_1x_4^{-1}x_0^{-3})\\
&=x_0x_1x_4^{-1}(x_0^{-2}x_1^2)x_4x_1x_0^{-2}x_1x_4^{-1}x_0^{-3}\\
&=x_0x_1x_4^{-1}x_3^2(x_0^{-2}x_4)x_1x_0^{-2}x_1x_4^{-1}x_0^{-3}\\
&=x_0x_1x_4^{-1}x_3^2x_6(x_0^{-2}x_1)x_0^{-2}x_1x_4^{-1}x_0^{-3}\\
&=x_0x_1x_4^{-1}x_3^2x_6x_3(x_0^{-4}x_1)x_4^{-1}x_0^{-3}\\
&=x_0x_1x_4^{-1}x_3^2x_6x_3x_5(x_0^{-4}x_4^{-1})x_0^{-3}\\
&=x_0x_1x_4^{-1}x_3^2x_6x_3x_5x_8^{-1}x_0^{-7}\\
&=x_0x_1x_4^{-1}x_3^3x_7x_5x_8^{-1}x_0^{-7}\\
&=x_0x_1x_4^{-1}x_3^3x_5x_8x_8^{-1}x_0^{-7}\\
&=x_0x_1x_4^{-1}x_3^3x_5x_0^{-7}\\
&=x_0x_1x_3^3x_5x_8^{-1}x_0^{-7}
\end{align*}
where we used that $x_n^{-1}x_k=x_kx_{n+1}^{-1}$ and $x_k^{-1}x_n=x_{n+1}x_k^{-1}$ for all $k<n$.
The following figure shows that $\theta(x_0x_1x_2^{-1})$ does not belong to $M_0$ and, therefore, $\theta^{-1}(M_0)\neq \beta^{-1}(\vec{F})$
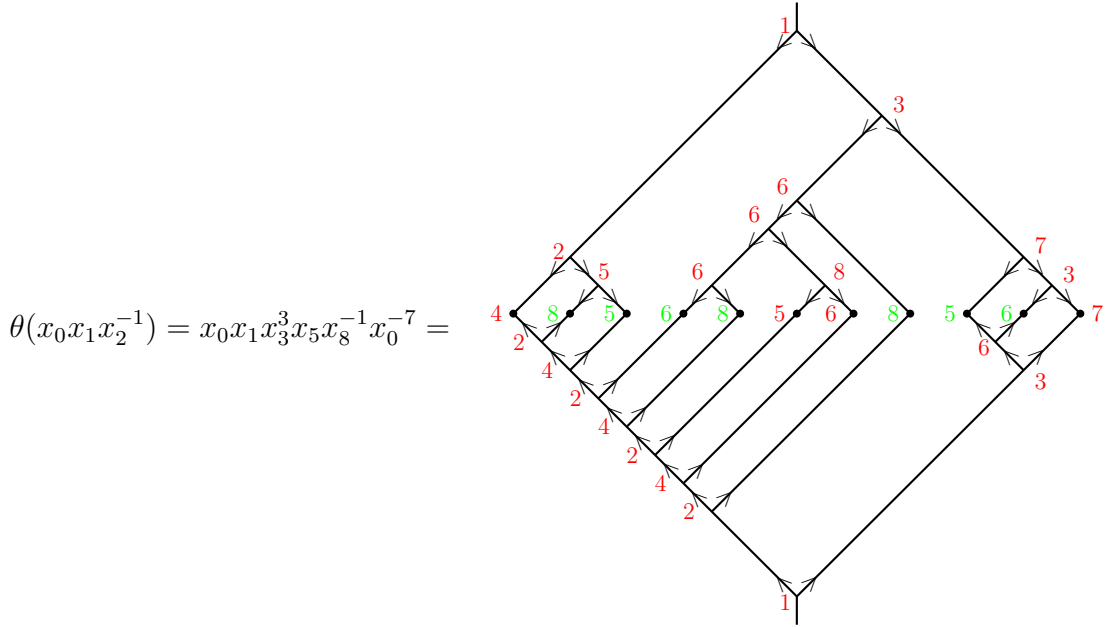
\[
\begin{tikzpicture}[x=.75cm, y=.75cm,
    every edge/.style={
        draw,
      postaction={decorate,
                    decoration={markings}
                   }
        }
]

\node at (-5,-.25) {$\scalebox{1}{$\theta(x_0x_1x_2^{-1})=x_0x_1x_3^3x_5x_8^{-1}x_0^{-7}=$}$};
 
\draw[thick] (0,0)--(5,5)--(10,0)--(5,-5)--(0,0); 
\draw[thick] (5,5)--(5,5.5);
\draw[thick] (5,-5)--(5,-5.5);
\draw[thick] (1,1)--(2,0)--(1,-1);
\draw[thick] (.5,-.5)--(1.5,.5);
\draw[thick] (2,-2)--(4,0)--(3.5,.5);
\draw[thick] (2.5,-2.5)--(5.5,.5);
\draw[thick] (3,-3)--(6,0)--(4.5,1.5);
\draw[thick] (6.5,3.5)--(1.5,-1.5);
\draw[thick] (5,2)--(7,0)--(3.5,-3.5);
\draw[thick] (8.5,-.5)--(9.5,.5);
\draw[thick] (9,-1)--(8,0)--(9,1);


\node at (.75,.65) {\rotatebox[origin=tr]{-135}{$\scalebox{.75}{$>$}$}};
\node at (1.25,.75) {\rotatebox[origin=tr]{-45}{$\scalebox{.75}{$>$}$}}; 

\node at (1.25,.15) {\rotatebox[origin=tr]{-135}{$\scalebox{.75}{$>$}$}};
\node at (1.75,.25) {\rotatebox[origin=tr]{-45}{$\scalebox{.75}{$>$}$}}; 

\node at (3.25,.15) {\rotatebox[origin=tr]{-135}{$\scalebox{.75}{$>$}$}};
\node at (3.75,.25) {\rotatebox[origin=tr]{-45}{$\scalebox{.75}{$>$}$}}; 

\node at (5.25,.15) {\rotatebox[origin=tr]{-135}{$\scalebox{.75}{$>$}$}};
\node at (5.75,.25) {\rotatebox[origin=tr]{-45}{$\scalebox{.75}{$>$}$}}; 

\node at (9.25,.15) {\rotatebox[origin=tr]{-135}{$\scalebox{.75}{$>$}$}};
\node at (9.75,.25) {\rotatebox[origin=tr]{-45}{$\scalebox{.75}{$>$}$}}; 
 
\node at (8.75,.65) {\rotatebox[origin=tr]{-135}{$\scalebox{.75}{$>$}$}};
\node at (9.25,.75) {\rotatebox[origin=tr]{-45}{$\scalebox{.75}{$>$}$}}; 

\node at (4.75,1.65) {\rotatebox[origin=tr]{-135}{$\scalebox{.75}{$>$}$}};
\node at (5.25,1.75) {\rotatebox[origin=tr]{-45}{$\scalebox{.75}{$>$}$}}; 

\node at (4.25,1.15) {\rotatebox[origin=tr]{-135}{$\scalebox{.75}{$>$}$}};
\node at (4.75,1.25) {\rotatebox[origin=tr]{-45}{$\scalebox{.75}{$>$}$}}; 

\node at (6.25,3.15) {\rotatebox[origin=tr]{-135}{$\scalebox{.75}{$>$}$}};
\node at (6.75,3.25) {\rotatebox[origin=tr]{-45}{$\scalebox{.75}{$>$}$}}; 

\node at (4.75,4.65) {\rotatebox[origin=tr]{-135}{$\scalebox{.75}{$>$}$}};
\node at (5.25,4.75) {\rotatebox[origin=tr]{-45}{$\scalebox{.75}{$>$}$}}; 


\node at (.25,-.25) {\rotatebox[origin=tr]{135}{$\scalebox{.75}{$>$}$}};
\node at (.75,-.25) {\rotatebox[origin=tr]{45}{$\scalebox{.75}{$>$}$}}; 

\node at (.75,-.75) {\rotatebox[origin=tr]{135}{$\scalebox{.75}{$>$}$}};
\node at (1.25,-.75) {\rotatebox[origin=tr]{45}{$\scalebox{.75}{$>$}$}}; 

\node at (1.25,-1.25) {\rotatebox[origin=tr]{135}{$\scalebox{.75}{$>$}$}};
\node at (1.75,-1.25) {\rotatebox[origin=tr]{45}{$\scalebox{.75}{$>$}$}}; 

\node at (1.75,-1.75) {\rotatebox[origin=tr]{135}{$\scalebox{.75}{$>$}$}};
\node at (2.25,-1.75) {\rotatebox[origin=tr]{45}{$\scalebox{.75}{$>$}$}}; 

\node at (2.25,-2.25) {\rotatebox[origin=tr]{135}{$\scalebox{.75}{$>$}$}};
\node at (2.75,-2.25) {\rotatebox[origin=tr]{45}{$\scalebox{.75}{$>$}$}}; 

\node at (2.75,-2.75) {\rotatebox[origin=tr]{135}{$\scalebox{.75}{$>$}$}};
\node at (3.25,-2.75) {\rotatebox[origin=tr]{45}{$\scalebox{.75}{$>$}$}}; 

\node at (3.25,-3.25) {\rotatebox[origin=tr]{135}{$\scalebox{.75}{$>$}$}};
\node at (3.75,-3.25) {\rotatebox[origin=tr]{45}{$\scalebox{.75}{$>$}$}}; 

\node at (4.75,-4.75) {\rotatebox[origin=tr]{135}{$\scalebox{.75}{$>$}$}};
\node at (5.25,-4.75) {\rotatebox[origin=tr]{45}{$\scalebox{.75}{$>$}$}}; 

\node at (8.25,-.25) {\rotatebox[origin=tr]{135}{$\scalebox{.75}{$>$}$}};
\node at (8.75,-.25) {\rotatebox[origin=tr]{45}{$\scalebox{.75}{$>$}$}}; 

\node at (8.75,-.75) {\rotatebox[origin=tr]{135}{$\scalebox{.75}{$>$}$}};
\node at (9.25,-.75) {\rotatebox[origin=tr]{45}{$\scalebox{.75}{$>$}$}};

  \fill (0,0)  circle[radius=1.5pt];
  \fill (1,0)  circle[radius=1.5pt];  
  \fill (2,0)  circle[radius=1.5pt];
  \fill (3,0)  circle[radius=1.5pt];
  \fill (4,0)  circle[radius=1.5pt]; 
  \fill (5,0)  circle[radius=1.5pt];
  \fill (6,0)  circle[radius=1.5pt];  
  \fill (7,0)  circle[radius=1.5pt];
  \fill (8,0)  circle[radius=1.5pt];
  \fill (9,0)  circle[radius=1.5pt]; 
  \fill (10,0)  circle[radius=1.5pt];

 \node[red] at (4.8,5.1) {$\scalebox{.75}{$1$}$};
 \node[red] at (.8,1.1) {$\scalebox{.75}{$2$}$};
 \node[red] at (-.3,0) {$\scalebox{.75}{$4$}$};
 \node[red] at (6.8,3.7) {$\scalebox{.75}{$3$}$};
 \node[red] at (9.8,.7) {$\scalebox{.75}{$3$}$};
 \node[red] at (9.3,1.2) {$\scalebox{.75}{$7$}$};
 \node[red] at (10.3,0) {$\scalebox{.75}{$7$}$};
 \node[green] at (8.7,0) {$\scalebox{.75}{$6$}$};
 \node[green] at (7.7,0) {$\scalebox{.75}{$5$}$};
 \node[green] at (6.7,0) {$\scalebox{.75}{$8$}$};
 \node[red] at (5.6,0) {$\scalebox{.75}{$6$}$};
 \node[red] at (4.7,0) {$\scalebox{.75}{$5$}$};
 \node[green] at (3.7,0) {$\scalebox{.75}{$8$}$};
 \node[green] at (2.7,0) {$\scalebox{.75}{$6$}$};
 \node[green] at (1.7,0) {$\scalebox{.75}{$5$}$};
 \node[green] at (.7,0) {$\scalebox{.75}{$8$}$};
 \node[red] at (1.6,.75) {$\scalebox{.75}{$5$}$};
 \node[red] at (3.25,.75) {$\scalebox{.75}{$6$}$};
 \node[red] at (5.75,.75) {$\scalebox{.75}{$8$}$};
 \node[red] at (4.25,1.75) {$\scalebox{.75}{$6$}$};
 \node[red] at (4.75,2.25) {$\scalebox{.75}{$6$}$};

 \node[red] at (4.8,-5.1) {$\scalebox{.75}{$1$}$};
 \node[red] at (.1,-.5) {$\scalebox{.75}{$2$}$};
 \node[red] at (.6,-1) {$\scalebox{.75}{$4$}$};
 \node[red] at (1.1,-1.5) {$\scalebox{.75}{$2$}$};
 \node[red] at (1.6,-2) {$\scalebox{.75}{$4$}$};
 \node[red] at (2.1,-2.5) {$\scalebox{.75}{$2$}$};
 \node[red] at (2.6,-3) {$\scalebox{.75}{$4$}$};
 \node[red] at (3.1,-3.5) {$\scalebox{.75}{$2$}$};
 \node[red] at (9.3,-1.2) {$\scalebox{.75}{$3$}$};
 \node[red] at (8.3,-.6) {$\scalebox{.75}{$6$}$};

\node at (0,-1.2) {$\;$};

\end{tikzpicture}
\]
Now we take care of $\theta^{-1}(M_1)$. First, we show that it does not stabilise any number in $(0,3/4]$. 
The following computations show that $\theta(x_2)=x_0^3x_1x_3x_0^{-5}\not\in M_1$ and, thus, $\theta^{-1}(M_1)\neq \stab(t)$ for all $t\in (0,3/4]$

 
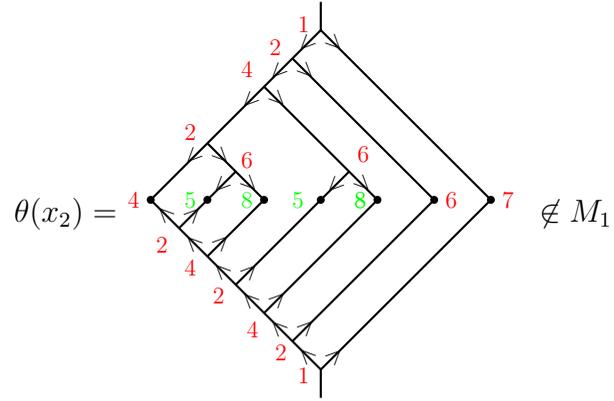
\[
 \begin{tikzpicture}[x=.75cm, y=.75cm,
    every edge/.style={
        draw,
      postaction={decorate,
                    decoration={markings}
                   }
        }
]

\node at (-3.5,-.25) {$\scalebox{1}{$\theta(x_2)=$}$};
\node at (5.5,-.25) {$\scalebox{1}{$\not\in M_1$}$};
 
\draw[thick] (1,3)--(1,3.5); 
\draw[thick] (1,-3)--(1,-3.5); 
\draw[thick] (-2,0)--(1,3)--(4,0)--(1,-3)--(-2,0);

\draw[thick] (-1,1)--(0,0)--(-1,-1); 
\draw[thick] (-.5,.5)--(-1.5,-.5); 
\draw[thick] (0,2)--(2,0)--(0,-2); 
\draw[thick] (1.5,.5)--(-.5,-1.5); 
\draw[thick] (.5,2.5)--(3,0)--(.5,-2.5);

\node at (.25,2.15) {\rotatebox[origin=tr]{-135}{$\scalebox{.75}{$>$}$}};%
\node at (.75,2.25) {\rotatebox[origin=tr]{-45}{$\scalebox{.75}{$>$}$}};%
\node at (-.75,.15) {\rotatebox[origin=tr]{-135}{$\scalebox{.75}{$>$}$}};%
\node at (-.25,.25) {\rotatebox[origin=tr]{-45}{$\scalebox{.75}{$>$}$}};%
\node at (.75,2.65) {\rotatebox[origin=tr]{-135}{$\scalebox{.75}{$>$}$}};%
\node at (1.25,2.75) {\rotatebox[origin=tr]{-45}{$\scalebox{.75}{$>$}$}};
\node at (1.25,.15) {\rotatebox[origin=tr]{-135}{$\scalebox{.75}{$>$}$}};
\node at (-.25,1.65) {\rotatebox[origin=tr]{-135}{$\scalebox{.75}{$>$}$}};%
\node at (1.75,.25) {\rotatebox[origin=tr]{-45}{$\scalebox{.75}{$>$}$}};%
\node at (.25,1.75) {\rotatebox[origin=tr]{-45}{$\scalebox{.75}{$>$}$}};
\node at (-.75,.75) {\rotatebox[origin=tr]{-45}{$\scalebox{.75}{$>$}$}};%
\node at (-1.25,.65) {\rotatebox[origin=tr]{-135}{$\scalebox{.75}{$>$}$}};%

\node at (-.25,-1.75) {\rotatebox[origin=tr]{135}{$\scalebox{.75}{$>$}$}};
\node at (-.75,-1.25) {\rotatebox[origin=tr]{135}{$\scalebox{.75}{$>$}$}};
\node at (-1.25,-.75) {\rotatebox[origin=tr]{135}{$\scalebox{.75}{$>$}$}};
\node at (-1.75,-.25) {\rotatebox[origin=tr]{135}{$\scalebox{.75}{$>$}$}};
\node at (.75,-2.75) {\rotatebox[origin=tr]{135}{$\scalebox{.75}{$>$}$}};
\node at (-.25,-1.25) {\rotatebox[origin=tr]{45}{$\scalebox{.75}{$>$}$}};
\node at (-.75,-.75) {\rotatebox[origin=tr]{45}{$\scalebox{.75}{$>$}$}};
\node at (.25,-2.25) {\rotatebox[origin=tr]{135}{$\scalebox{.75}{$>$}$}};
\node at (1.25,-2.75) {\rotatebox[origin=tr]{45}{$\scalebox{.75}{$>$}$}};
\node at (.75,-2.25) {\rotatebox[origin=tr]{45}{$\scalebox{.75}{$>$}$}};
\node at (-1.25,-.25) {\rotatebox[origin=tr]{45}{$\scalebox{.75}{$>$}$}};
\node at (.25,-1.75) {\rotatebox[origin=tr]{45}{$\scalebox{.75}{$>$}$}};

  \fill (0,0)  circle[radius=1.5pt];
  \fill (1,0)  circle[radius=1.5pt];  
  \fill (2,0)  circle[radius=1.5pt];
  \fill (3,0)  circle[radius=1.5pt];
  \fill (4,0)  circle[radius=1.5pt];
  \fill (-1,0)  circle[radius=1.5pt];
  \fill (-2,0)  circle[radius=1.5pt];

 \node[red] at (-.3,.7) {$\scalebox{.75}{$6$}$};
 \node[red] at (-.2,-2.3) {$\scalebox{.75}{$4$}$};
 \node[red] at (-.3,2.3) {$\scalebox{.75}{$4$}$};
 \node[green] at (-.3,0) {$\scalebox{.75}{$8$}$};
 \node[red] at (4.3,0) {$\scalebox{.75}{$7$}$};
 \node[red] at (-.8,-1.7) {$\scalebox{.75}{$2$}$};
 \node[red] at (.2,2.7) {$\scalebox{.75}{$2$}$};
  \node[green] at (1.7,0) {$\scalebox{.75}{$8$}$};
  \node[green] at (.6,0) {$\scalebox{.75}{$5$}$};
  \node[red] at (3.3,0) {$\scalebox{.75}{$6$}$};
  \node[red] at (-1.8,-.8) {$\scalebox{.75}{$2$}$};
 \node[red] at (-1.3,-1.2) {$\scalebox{.75}{$4$}$};
 \node[red] at (1.75,.8) {$\scalebox{.75}{$6$}$};
 \node[green] at (1.7,0) {$\scalebox{.75}{$8$}$};
 \node[green] at (-1.3,0) {$\scalebox{.75}{$5$}$};
 \node[red] at (-2.3,0) {$\scalebox{.75}{$4$}$};
 \node[red] at (.7,-3.1) {$\scalebox{.75}{$1$}$};
 \node[red] at (.7,3.1) {$\scalebox{.75}{$1$}$};
 \node[red] at (-1.3,1.2) {$\scalebox{.75}{$2$}$};
 \node[red] at (.3,-2.7) {$\scalebox{.75}{$2$}$};

\node at (0,-1.2) {$\;$};

\end{tikzpicture}
 \]
 Similarly, $\theta^{-1}(M_1)$ is different from $\stab(t)$ for all $t\in [1/4,1)$ because
 \begin{align*}
 \theta(\sigma(x_2))&= \theta(\sigma(x_0^{-1}))  \theta(\sigma(x_1))  \theta(\sigma(x_0))\\
 &=\theta(x_0) x_1x_2 \theta(x_0)^{-1}\\
 &=(x_0x_1x_4^{-1}x_0^{-3})(x_1x_2)(x_0^{3}x_4x_1^{-1}x_0^{-1})\\
 &=x_0x_1x_4^{-1}x_4x_5x_4x_1^{-1}x_0^{-1}\\
 &=x_0x_1x_5x_4x_1^{-1}x_0^{-1}\\
 &=x_0x_4x_3x_1x_1^{-1}x_0^{-1}\\
 &=x_0x_4x_3x_0^{-1}=x_3x_2=x_2x_4\not\in M_1\\
 \end{align*}
 The following figure shows that $\theta(x_0x_1x_2^{-1})$ does not belong to $M_1$
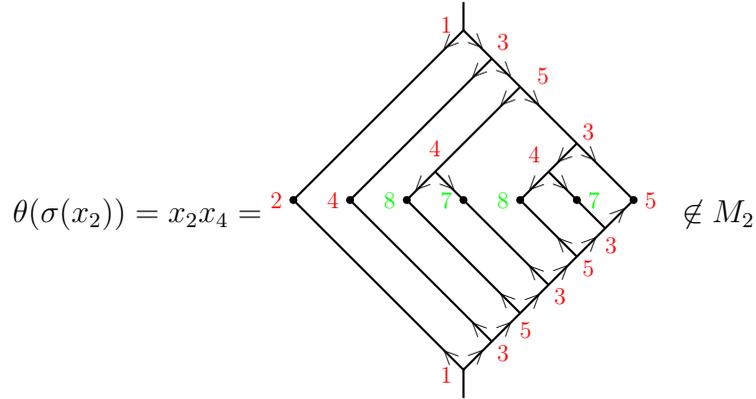
\[
\begin{tikzpicture}[x=.75cm, y=.75cm,
    every edge/.style={
        draw,
      postaction={decorate,
                    decoration={markings}
                   }
        }
]

\node at (-5,-.25) {$\scalebox{1}{$\theta(x_0x_1x_2^{-1})=x_0x_1x_3^3x_5x_8^{-1}x_0^{-7}=$}$};
 
\draw[thick] (0,0)--(5,5)--(10,0)--(5,-5)--(0,0); 
\draw[thick] (5,5)--(5,5.5);
\draw[thick] (5,-5)--(5,-5.5);
\draw[thick] (1,1)--(2,0)--(1,-1);
\draw[thick] (.5,-.5)--(1.5,.5);
\draw[thick] (2,-2)--(4,0)--(3.5,.5);
\draw[thick] (2.5,-2.5)--(5.5,.5);
\draw[thick] (3,-3)--(6,0)--(4.5,1.5);
\draw[thick] (6.5,3.5)--(1.5,-1.5);
\draw[thick] (5,2)--(7,0)--(3.5,-3.5);
\draw[thick] (8.5,-.5)--(9.5,.5);
\draw[thick] (9,-1)--(8,0)--(9,1);


\node at (.75,.65) {\rotatebox[origin=tr]{-135}{$\scalebox{.75}{$>$}$}};
\node at (1.25,.75) {\rotatebox[origin=tr]{-45}{$\scalebox{.75}{$>$}$}}; 

\node at (1.25,.15) {\rotatebox[origin=tr]{-135}{$\scalebox{.75}{$>$}$}};
\node at (1.75,.25) {\rotatebox[origin=tr]{-45}{$\scalebox{.75}{$>$}$}}; 

\node at (3.25,.15) {\rotatebox[origin=tr]{-135}{$\scalebox{.75}{$>$}$}};
\node at (3.75,.25) {\rotatebox[origin=tr]{-45}{$\scalebox{.75}{$>$}$}}; 

\node at (5.25,.15) {\rotatebox[origin=tr]{-135}{$\scalebox{.75}{$>$}$}};
\node at (5.75,.25) {\rotatebox[origin=tr]{-45}{$\scalebox{.75}{$>$}$}}; 

\node at (9.25,.15) {\rotatebox[origin=tr]{-135}{$\scalebox{.75}{$>$}$}};
\node at (9.75,.25) {\rotatebox[origin=tr]{-45}{$\scalebox{.75}{$>$}$}}; 
 
\node at (8.75,.65) {\rotatebox[origin=tr]{-135}{$\scalebox{.75}{$>$}$}};
\node at (9.25,.75) {\rotatebox[origin=tr]{-45}{$\scalebox{.75}{$>$}$}}; 

\node at (4.75,1.65) {\rotatebox[origin=tr]{-135}{$\scalebox{.75}{$>$}$}};
\node at (5.25,1.75) {\rotatebox[origin=tr]{-45}{$\scalebox{.75}{$>$}$}}; 

\node at (4.25,1.15) {\rotatebox[origin=tr]{-135}{$\scalebox{.75}{$>$}$}};
\node at (4.75,1.25) {\rotatebox[origin=tr]{-45}{$\scalebox{.75}{$>$}$}}; 

\node at (6.25,3.15) {\rotatebox[origin=tr]{-135}{$\scalebox{.75}{$>$}$}};
\node at (6.75,3.25) {\rotatebox[origin=tr]{-45}{$\scalebox{.75}{$>$}$}}; 

\node at (4.75,4.65) {\rotatebox[origin=tr]{-135}{$\scalebox{.75}{$>$}$}};
\node at (5.25,4.75) {\rotatebox[origin=tr]{-45}{$\scalebox{.75}{$>$}$}}; 


\node at (.25,-.25) {\rotatebox[origin=tr]{135}{$\scalebox{.75}{$>$}$}};
\node at (.75,-.25) {\rotatebox[origin=tr]{45}{$\scalebox{.75}{$>$}$}}; 

\node at (.75,-.75) {\rotatebox[origin=tr]{135}{$\scalebox{.75}{$>$}$}};
\node at (1.25,-.75) {\rotatebox[origin=tr]{45}{$\scalebox{.75}{$>$}$}}; 

\node at (1.25,-1.25) {\rotatebox[origin=tr]{135}{$\scalebox{.75}{$>$}$}};
\node at (1.75,-1.25) {\rotatebox[origin=tr]{45}{$\scalebox{.75}{$>$}$}}; 

\node at (1.75,-1.75) {\rotatebox[origin=tr]{135}{$\scalebox{.75}{$>$}$}};
\node at (2.25,-1.75) {\rotatebox[origin=tr]{45}{$\scalebox{.75}{$>$}$}}; 

\node at (2.25,-2.25) {\rotatebox[origin=tr]{135}{$\scalebox{.75}{$>$}$}};
\node at (2.75,-2.25) {\rotatebox[origin=tr]{45}{$\scalebox{.75}{$>$}$}}; 

\node at (2.75,-2.75) {\rotatebox[origin=tr]{135}{$\scalebox{.75}{$>$}$}};
\node at (3.25,-2.75) {\rotatebox[origin=tr]{45}{$\scalebox{.75}{$>$}$}}; 

\node at (3.25,-3.25) {\rotatebox[origin=tr]{135}{$\scalebox{.75}{$>$}$}};
\node at (3.75,-3.25) {\rotatebox[origin=tr]{45}{$\scalebox{.75}{$>$}$}}; 

\node at (4.75,-4.75) {\rotatebox[origin=tr]{135}{$\scalebox{.75}{$>$}$}};
\node at (5.25,-4.75) {\rotatebox[origin=tr]{45}{$\scalebox{.75}{$>$}$}}; 

\node at (8.25,-.25) {\rotatebox[origin=tr]{135}{$\scalebox{.75}{$>$}$}};
\node at (8.75,-.25) {\rotatebox[origin=tr]{45}{$\scalebox{.75}{$>$}$}}; 

\node at (8.75,-.75) {\rotatebox[origin=tr]{135}{$\scalebox{.75}{$>$}$}};
\node at (9.25,-.75) {\rotatebox[origin=tr]{45}{$\scalebox{.75}{$>$}$}};

  \fill (0,0)  circle[radius=1.5pt];
  \fill (1,0)  circle[radius=1.5pt];  
  \fill (2,0)  circle[radius=1.5pt];
  \fill (3,0)  circle[radius=1.5pt];
  \fill (4,0)  circle[radius=1.5pt]; 
  \fill (5,0)  circle[radius=1.5pt];
  \fill (6,0)  circle[radius=1.5pt];  
  \fill (7,0)  circle[radius=1.5pt];
  \fill (8,0)  circle[radius=1.5pt];
  \fill (9,0)  circle[radius=1.5pt]; 
  \fill (10,0)  circle[radius=1.5pt];

 \node[red] at (4.8,5.1) {$\scalebox{.75}{$1$}$};
 \node[red] at (.8,1.1) {$\scalebox{.75}{$2$}$};
 \node[red] at (-.3,0) {$\scalebox{.75}{$4$}$};
 \node[red] at (6.8,3.7) {$\scalebox{.75}{$3$}$};
 \node[red] at (9.8,.7) {$\scalebox{.75}{$3$}$};
 \node[red] at (9.3,1.2) {$\scalebox{.75}{$7$}$};
 \node[red] at (10.3,0) {$\scalebox{.75}{$7$}$};
 \node[red] at (8.7,0) {$\scalebox{.75}{$5$}$};
 \node[green] at (7.7,0) {$\scalebox{.75}{$8$}$};
 \node[green] at (6.7,0) {$\scalebox{.75}{$5$}$};
 \node[red] at (5.6,0) {$\scalebox{.75}{$6$}$};
 \node[green] at (4.7,0) {$\scalebox{.75}{$8$}$};
 \node[green] at (3.7,0) {$\scalebox{.75}{$5$}$};
 \node[red] at (2.7,0) {$\scalebox{.75}{$6$}$};
 \node[green] at (1.7,0) {$\scalebox{.75}{$8$}$};
 \node[green] at (.7,0) {$\scalebox{.75}{$5$}$};
 \node[red] at (1.6,.75) {$\scalebox{.75}{$6$}$};
 \node[red] at (3.25,.75) {$\scalebox{.75}{$5$}$};
 \node[red] at (5.75,.75) {$\scalebox{.75}{$8$}$};
 \node[red] at (4.25,1.75) {$\scalebox{.75}{$6$}$};
 \node[red] at (4.75,2.25) {$\scalebox{.75}{$5$}$};

 \node[red] at (4.8,-5.1) {$\scalebox{.75}{$1$}$};
 \node[red] at (.1,-.5) {$\scalebox{.75}{$2$}$};
 \node[red] at (.6,-1) {$\scalebox{.75}{$4$}$};
 \node[red] at (1.1,-1.5) {$\scalebox{.75}{$2$}$};
 \node[red] at (1.6,-2) {$\scalebox{.75}{$4$}$};
 \node[red] at (2.1,-2.5) {$\scalebox{.75}{$2$}$};
 \node[red] at (2.6,-3) {$\scalebox{.75}{$4$}$};
 \node[red] at (3.1,-3.5) {$\scalebox{.75}{$2$}$};
 \node[red] at (9.3,-1.2) {$\scalebox{.75}{$3$}$};
 \node[red] at (8.3,-.6) {$\scalebox{.75}{$5$}$};

\node at (0,-1.2) {$\;$};

\end{tikzpicture}
\]
This means that $\theta^{-1}(M_1)\neq \beta^{-1}(\vec{F})$.

Now, we deal with $\theta^{-1}(M_2)$.
The element $x_2\in \stab(t)$ for $t\in (0,3/4]$. 
The following computations show that 
$\theta^{-1}(M_2)\neq \stab(t)$ for all $t\in (0,3/4]$
 \[
 \begin{tikzpicture}[x=.75cm, y=.75cm,
    every edge/.style={
        draw,
      postaction={decorate,
                    decoration={markings}
                   }
        }
]

\node at (-3.5,-.25) {$\scalebox{1}{$\theta(x_2)=$}$};
\node at (5.5,-.25) {$\scalebox{1}{$\not\in M_2$}$};
 
\draw[thick] (1,3)--(1,3.5); 
\draw[thick] (1,-3)--(1,-3.5); 
\draw[thick] (-2,0)--(1,3)--(4,0)--(1,-3)--(-2,0);

\draw[thick] (-1,1)--(0,0)--(-1,-1); 
\draw[thick] (-.5,.5)--(-1.5,-.5); 
\draw[thick] (0,2)--(2,0)--(0,-2); 
\draw[thick] (1.5,.5)--(-.5,-1.5); 
\draw[thick] (.5,2.5)--(3,0)--(.5,-2.5);

\node at (.25,2.15) {\rotatebox[origin=tr]{-135}{$\scalebox{.75}{$>$}$}};%
\node at (.75,2.25) {\rotatebox[origin=tr]{-45}{$\scalebox{.75}{$>$}$}};%
\node at (-.75,.15) {\rotatebox[origin=tr]{-135}{$\scalebox{.75}{$>$}$}};%
\node at (-.25,.25) {\rotatebox[origin=tr]{-45}{$\scalebox{.75}{$>$}$}};%
\node at (.75,2.65) {\rotatebox[origin=tr]{-135}{$\scalebox{.75}{$>$}$}};%
\node at (1.25,2.75) {\rotatebox[origin=tr]{-45}{$\scalebox{.75}{$>$}$}};
\node at (1.25,.15) {\rotatebox[origin=tr]{-135}{$\scalebox{.75}{$>$}$}};
\node at (-.25,1.65) {\rotatebox[origin=tr]{-135}{$\scalebox{.75}{$>$}$}};%
\node at (1.75,.25) {\rotatebox[origin=tr]{-45}{$\scalebox{.75}{$>$}$}};%
\node at (.25,1.75) {\rotatebox[origin=tr]{-45}{$\scalebox{.75}{$>$}$}};
\node at (-.75,.75) {\rotatebox[origin=tr]{-45}{$\scalebox{.75}{$>$}$}};%
\node at (-1.25,.65) {\rotatebox[origin=tr]{-135}{$\scalebox{.75}{$>$}$}};%

\node at (-.25,-1.75) {\rotatebox[origin=tr]{135}{$\scalebox{.75}{$>$}$}};
\node at (-.75,-1.25) {\rotatebox[origin=tr]{135}{$\scalebox{.75}{$>$}$}};
\node at (-1.25,-.75) {\rotatebox[origin=tr]{135}{$\scalebox{.75}{$>$}$}};
\node at (-1.75,-.25) {\rotatebox[origin=tr]{135}{$\scalebox{.75}{$>$}$}};
\node at (.75,-2.75) {\rotatebox[origin=tr]{135}{$\scalebox{.75}{$>$}$}};
\node at (-.25,-1.25) {\rotatebox[origin=tr]{45}{$\scalebox{.75}{$>$}$}};
\node at (-.75,-.75) {\rotatebox[origin=tr]{45}{$\scalebox{.75}{$>$}$}};
\node at (.25,-2.25) {\rotatebox[origin=tr]{135}{$\scalebox{.75}{$>$}$}};
\node at (1.25,-2.75) {\rotatebox[origin=tr]{45}{$\scalebox{.75}{$>$}$}};
\node at (.75,-2.25) {\rotatebox[origin=tr]{45}{$\scalebox{.75}{$>$}$}};
\node at (-1.25,-.25) {\rotatebox[origin=tr]{45}{$\scalebox{.75}{$>$}$}};
\node at (.25,-1.75) {\rotatebox[origin=tr]{45}{$\scalebox{.75}{$>$}$}};

  \fill (0,0)  circle[radius=1.5pt];
  \fill (1,0)  circle[radius=1.5pt];  
  \fill (2,0)  circle[radius=1.5pt];
  \fill (3,0)  circle[radius=1.5pt];
  \fill (4,0)  circle[radius=1.5pt];
  \fill (-1,0)  circle[radius=1.5pt];
  \fill (-2,0)  circle[radius=1.5pt];

 \node[red] at (-.3,.7) {$\scalebox{.75}{$7$}$};
 \node[red] at (-.2,-2.3) {$\scalebox{.75}{$6$}$};
 \node[red] at (-.3,2.3) {$\scalebox{.75}{$6$}$};
 \node[green] at (-.3,0) {$\scalebox{.75}{$4$}$};
 \node[red] at (4.3,0) {$\scalebox{.75}{$3$}$};
 \node[red] at (-.8,-1.7) {$\scalebox{.75}{$2$}$};
 \node[red] at (.2,2.7) {$\scalebox{.75}{$2$}$};
  \node[red] at (1.7,0) {$\scalebox{.75}{$8$}$};
  \node[green] at (.6,0) {$\scalebox{.75}{$4$}$};
  \node[red] at (3.3,0) {$\scalebox{.75}{$7$}$};
  \node[red] at (-1.8,-.8) {$\scalebox{.75}{$2$}$};
 \node[red] at (-1.3,-1.2) {$\scalebox{.75}{$6$}$};
 \node[red] at (1.75,.8) {$\scalebox{.75}{$8$}$};
  \node[red] at (-1.3,0) {$\scalebox{.75}{$7$}$};
 \node[red] at (-2.3,0) {$\scalebox{.75}{$6$}$};
 \node[red] at (.7,-3.1) {$\scalebox{.75}{$1$}$};
 \node[red] at (.7,3.1) {$\scalebox{.75}{$1$}$};
 \node[red] at (-1.3,1.2) {$\scalebox{.75}{$2$}$};
 \node[red] at (.3,-2.7) {$\scalebox{.75}{$2$}$};

\node at (0,-1.2) {$\;$};

\end{tikzpicture}
 \]
Since
 \[
 \begin{tikzpicture}[x=.75cm, y=.75cm,
    every edge/.style={
        draw,
      postaction={decorate,
                    decoration={markings}
                   }
        }
]

\node at (-4.75,-.25) {$\scalebox{1}{$\theta(\sigma(x_2))=x_2x_4=$}$};
\node at (5.5,-.25) {$\scalebox{1}{$\not\in M_2$}$};
 
\draw[thick] (1,3)--(1,3.5); 
\draw[thick] (1,-3)--(1,-3.5); 
\draw[thick] (-2,0)--(1,3)--(4,0)--(1,-3)--(-2,0);

\draw[thick] (1.5,2.5)--(-1,0)--(1.5,-2.5); 
\draw[thick] (2,2)--(.5,.5)--(2.5,-1.5); 
\draw[thick] (.5,.5)--(0,0)--(2,-2); 
\draw[thick] (3,-1)--(2,0)--(3,1); 
\draw[thick] (2.5,.5)--(3.5,-.5);

\node at (2.75,.65) {\rotatebox[origin=tr]{-135}{$\scalebox{.75}{$>$}$}};%
\node at (3.25,.75) {\rotatebox[origin=tr]{-45}{$\scalebox{.75}{$>$}$}};%
\node at (1.25,2.15) {\rotatebox[origin=tr]{-135}{$\scalebox{.75}{$>$}$}};%
\node at (1.75,2.25) {\rotatebox[origin=tr]{-45}{$\scalebox{.75}{$>$}$}};%
\node at (.75,2.65) {\rotatebox[origin=tr]{-135}{$\scalebox{.75}{$>$}$}};%
\node at (1.25,2.75) {\rotatebox[origin=tr]{-45}{$\scalebox{.75}{$>$}$}};
\node at (2.25,.15) {\rotatebox[origin=tr]{-135}{$\scalebox{.75}{$>$}$}};
\node at (.25,.15) {\rotatebox[origin=tr]{-135}{$\scalebox{.75}{$>$}$}};%
\node at (.75,.25) {\rotatebox[origin=tr]{-45}{$\scalebox{.75}{$>$}$}};%
\node at (2.75,.25) {\rotatebox[origin=tr]{-45}{$\scalebox{.75}{$>$}$}};
\node at (2.25,1.75) {\rotatebox[origin=tr]{-45}{$\scalebox{.75}{$>$}$}};%
\node at (1.75,1.65) {\rotatebox[origin=tr]{-135}{$\scalebox{.75}{$>$}$}};%

\node at (1.75,-1.75) {\rotatebox[origin=tr]{135}{$\scalebox{.75}{$>$}$}};
\node at (2.25,-1.25) {\rotatebox[origin=tr]{135}{$\scalebox{.75}{$>$}$}};
\node at (2.75,-.75) {\rotatebox[origin=tr]{135}{$\scalebox{.75}{$>$}$}};
\node at (.75,-2.75) {\rotatebox[origin=tr]{135}{$\scalebox{.75}{$>$}$}};
\node at (2.75,-1.25) {\rotatebox[origin=tr]{45}{$\scalebox{.75}{$>$}$}};
\node at (3.25,-.75) {\rotatebox[origin=tr]{45}{$\scalebox{.75}{$>$}$}};
\node at (1.25,-2.25) {\rotatebox[origin=tr]{135}{$\scalebox{.75}{$>$}$}};
\node at (1.25,-2.75) {\rotatebox[origin=tr]{45}{$\scalebox{.75}{$>$}$}};
\node at (1.75,-2.25) {\rotatebox[origin=tr]{45}{$\scalebox{.75}{$>$}$}};
\node at (3.75,-.25) {\rotatebox[origin=tr]{45}{$\scalebox{.75}{$>$}$}};
\node at (2.25,-1.75) {\rotatebox[origin=tr]{45}{$\scalebox{.75}{$>$}$}};

  \fill (0,0)  circle[radius=1.5pt];
  \fill (1,0)  circle[radius=1.5pt];  
  \fill (2,0)  circle[radius=1.5pt];
  \fill (3,0)  circle[radius=1.5pt];
  \fill (4,0)  circle[radius=1.5pt];
  \fill (-1,0)  circle[radius=1.5pt];
  \fill (-2,0)  circle[radius=1.5pt];

 \node[red] at (1.7,2.8) {$\scalebox{.75}{$3$}$};
 \node[red] at (2.1,-2.3) {$\scalebox{.75}{$5$}$};
 \node[red] at (2.4,2.2) {$\scalebox{.75}{$5$}$};
 \node[green] at (-.3,0) {$\scalebox{.75}{$8$}$};
 \node[red] at (4.3,0) {$\scalebox{.75}{$5$}$};
 \node[red] at (2.7,-1.7) {$\scalebox{.75}{$3$}$};
 \node[red] at (3.2,1.2) {$\scalebox{.75}{$3$}$};
  \node[green] at (1.7,0) {$\scalebox{.75}{$8$}$};
  \node[green] at (3.3,0) {$\scalebox{.75}{$7$}$};
  \node[red] at (3.6,-.8) {$\scalebox{.75}{$3$}$};
 \node[red] at (3.2,-1.2) {$\scalebox{.75}{$5$}$};
 \node[red] at (2.25,.8) {$\scalebox{.75}{$4$}$};
 \node[green] at (.7,0) {$\scalebox{.75}{$7$}$};
 \node[red] at (-1.3,0) {$\scalebox{.75}{$4$}$};
 \node[red] at (-2.3,0) {$\scalebox{.75}{$2$}$};
 \node[red] at (.7,-3.1) {$\scalebox{.75}{$1$}$};
 \node[red] at (.7,3.1) {$\scalebox{.75}{$1$}$};
 \node[red] at (.5,.9) {$\scalebox{.75}{$4$}$};
 \node[red] at (1.7,-2.7) {$\scalebox{.75}{$3$}$};

\node at (0,-1.2) {$\;$};

\end{tikzpicture}
 \]
 we see that $\theta^{-1}(M_2)\neq \stab(t)$ for all $t\in (0,1)$.
 
 In order to prove that $\theta^{-1}(M_2)\neq \beta^{-1}(\vec{F})$, it suffices to show that
$\theta(x_0x_1x_2^{-1})$ does not belong to $M_2$ (recall that $x_0x_1x_2^{-1}\in\beta^{-1}(\vec{F})$)
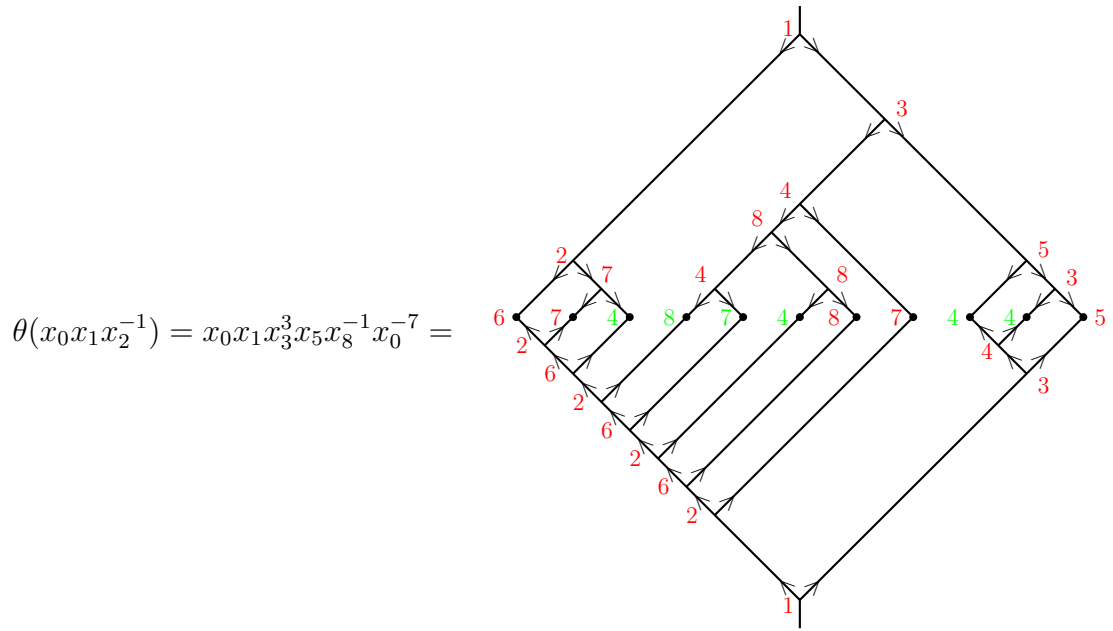
\[
\begin{tikzpicture}[x=.75cm, y=.75cm,
    every edge/.style={
        draw,
      postaction={decorate,
                    decoration={markings}
                   }
        }
]

\node at (-5,-.25) {$\scalebox{1}{$\theta(x_0x_1x_2^{-1})=x_0x_1x_3^3x_5x_8^{-1}x_0^{-7}=$}$};
 
\draw[thick] (0,0)--(5,5)--(10,0)--(5,-5)--(0,0); 
\draw[thick] (5,5)--(5,5.5);
\draw[thick] (5,-5)--(5,-5.5);
\draw[thick] (1,1)--(2,0)--(1,-1);
\draw[thick] (.5,-.5)--(1.5,.5);
\draw[thick] (2,-2)--(4,0)--(3.5,.5);
\draw[thick] (2.5,-2.5)--(5.5,.5);
\draw[thick] (3,-3)--(6,0)--(4.5,1.5);
\draw[thick] (6.5,3.5)--(1.5,-1.5);
\draw[thick] (5,2)--(7,0)--(3.5,-3.5);
\draw[thick] (8.5,-.5)--(9.5,.5);
\draw[thick] (9,-1)--(8,0)--(9,1);


\node at (.75,.65) {\rotatebox[origin=tr]{-135}{$\scalebox{.75}{$>$}$}};
\node at (1.25,.75) {\rotatebox[origin=tr]{-45}{$\scalebox{.75}{$>$}$}}; 

\node at (1.25,.15) {\rotatebox[origin=tr]{-135}{$\scalebox{.75}{$>$}$}};
\node at (1.75,.25) {\rotatebox[origin=tr]{-45}{$\scalebox{.75}{$>$}$}}; 

\node at (3.25,.15) {\rotatebox[origin=tr]{-135}{$\scalebox{.75}{$>$}$}};
\node at (3.75,.25) {\rotatebox[origin=tr]{-45}{$\scalebox{.75}{$>$}$}}; 

\node at (5.25,.15) {\rotatebox[origin=tr]{-135}{$\scalebox{.75}{$>$}$}};
\node at (5.75,.25) {\rotatebox[origin=tr]{-45}{$\scalebox{.75}{$>$}$}}; 

\node at (9.25,.15) {\rotatebox[origin=tr]{-135}{$\scalebox{.75}{$>$}$}};
\node at (9.75,.25) {\rotatebox[origin=tr]{-45}{$\scalebox{.75}{$>$}$}}; 
 
\node at (8.75,.65) {\rotatebox[origin=tr]{-135}{$\scalebox{.75}{$>$}$}};
\node at (9.25,.75) {\rotatebox[origin=tr]{-45}{$\scalebox{.75}{$>$}$}}; 

\node at (4.75,1.65) {\rotatebox[origin=tr]{-135}{$\scalebox{.75}{$>$}$}};
\node at (5.25,1.75) {\rotatebox[origin=tr]{-45}{$\scalebox{.75}{$>$}$}}; 

\node at (4.25,1.15) {\rotatebox[origin=tr]{-135}{$\scalebox{.75}{$>$}$}};
\node at (4.75,1.25) {\rotatebox[origin=tr]{-45}{$\scalebox{.75}{$>$}$}}; 

\node at (6.25,3.15) {\rotatebox[origin=tr]{-135}{$\scalebox{.75}{$>$}$}};
\node at (6.75,3.25) {\rotatebox[origin=tr]{-45}{$\scalebox{.75}{$>$}$}}; 

\node at (4.75,4.65) {\rotatebox[origin=tr]{-135}{$\scalebox{.75}{$>$}$}};
\node at (5.25,4.75) {\rotatebox[origin=tr]{-45}{$\scalebox{.75}{$>$}$}}; 


\node at (.25,-.25) {\rotatebox[origin=tr]{135}{$\scalebox{.75}{$>$}$}};
\node at (.75,-.25) {\rotatebox[origin=tr]{45}{$\scalebox{.75}{$>$}$}}; 

\node at (.75,-.75) {\rotatebox[origin=tr]{135}{$\scalebox{.75}{$>$}$}};
\node at (1.25,-.75) {\rotatebox[origin=tr]{45}{$\scalebox{.75}{$>$}$}}; 

\node at (1.25,-1.25) {\rotatebox[origin=tr]{135}{$\scalebox{.75}{$>$}$}};
\node at (1.75,-1.25) {\rotatebox[origin=tr]{45}{$\scalebox{.75}{$>$}$}}; 

\node at (1.75,-1.75) {\rotatebox[origin=tr]{135}{$\scalebox{.75}{$>$}$}};
\node at (2.25,-1.75) {\rotatebox[origin=tr]{45}{$\scalebox{.75}{$>$}$}}; 

\node at (2.25,-2.25) {\rotatebox[origin=tr]{135}{$\scalebox{.75}{$>$}$}};
\node at (2.75,-2.25) {\rotatebox[origin=tr]{45}{$\scalebox{.75}{$>$}$}}; 

\node at (2.75,-2.75) {\rotatebox[origin=tr]{135}{$\scalebox{.75}{$>$}$}};
\node at (3.25,-2.75) {\rotatebox[origin=tr]{45}{$\scalebox{.75}{$>$}$}}; 

\node at (3.25,-3.25) {\rotatebox[origin=tr]{135}{$\scalebox{.75}{$>$}$}};
\node at (3.75,-3.25) {\rotatebox[origin=tr]{45}{$\scalebox{.75}{$>$}$}}; 

\node at (4.75,-4.75) {\rotatebox[origin=tr]{135}{$\scalebox{.75}{$>$}$}};
\node at (5.25,-4.75) {\rotatebox[origin=tr]{45}{$\scalebox{.75}{$>$}$}}; 

\node at (8.25,-.25) {\rotatebox[origin=tr]{135}{$\scalebox{.75}{$>$}$}};
\node at (8.75,-.25) {\rotatebox[origin=tr]{45}{$\scalebox{.75}{$>$}$}}; 

\node at (8.75,-.75) {\rotatebox[origin=tr]{135}{$\scalebox{.75}{$>$}$}};
\node at (9.25,-.75) {\rotatebox[origin=tr]{45}{$\scalebox{.75}{$>$}$}};

  \fill (0,0)  circle[radius=1.5pt];
  \fill (1,0)  circle[radius=1.5pt];  
  \fill (2,0)  circle[radius=1.5pt];
  \fill (3,0)  circle[radius=1.5pt];
  \fill (4,0)  circle[radius=1.5pt]; 
  \fill (5,0)  circle[radius=1.5pt];
  \fill (6,0)  circle[radius=1.5pt];  
  \fill (7,0)  circle[radius=1.5pt];
  \fill (8,0)  circle[radius=1.5pt];
  \fill (9,0)  circle[radius=1.5pt]; 
  \fill (10,0)  circle[radius=1.5pt];

 \node[red] at (4.8,5.1) {$\scalebox{.75}{$1$}$};
 \node[red] at (.8,1.1) {$\scalebox{.75}{$2$}$};
 \node[red] at (-.3,0) {$\scalebox{.75}{$6$}$};
 \node[red] at (6.8,3.7) {$\scalebox{.75}{$3$}$};
 \node[red] at (9.8,.7) {$\scalebox{.75}{$3$}$};
 \node[red] at (9.3,1.2) {$\scalebox{.75}{$5$}$};
 \node[red] at (10.3,0) {$\scalebox{.75}{$5$}$};
 \node[green] at (8.7,0) {$\scalebox{.75}{$4$}$};
 \node[green] at (7.7,0) {$\scalebox{.75}{$4$}$};
 \node[red] at (6.7,0) {$\scalebox{.75}{$7$}$};
 \node[red] at (5.6,0) {$\scalebox{.75}{$8$}$};
 \node[green] at (4.7,0) {$\scalebox{.75}{$4$}$};
 \node[green] at (3.7,0) {$\scalebox{.75}{$7$}$};
 \node[green] at (2.7,0) {$\scalebox{.75}{$8$}$};
 \node[green] at (1.7,0) {$\scalebox{.75}{$4$}$};
 \node[red] at (.7,0) {$\scalebox{.75}{$7$}$};
 \node[red] at (1.6,.75) {$\scalebox{.75}{$7$}$};
 \node[red] at (3.25,.75) {$\scalebox{.75}{$4$}$};
 \node[red] at (5.75,.75) {$\scalebox{.75}{$8$}$};
 \node[red] at (4.25,1.75) {$\scalebox{.75}{$8$}$};
 \node[red] at (4.75,2.25) {$\scalebox{.75}{$4$}$};

 \node[red] at (4.8,-5.1) {$\scalebox{.75}{$1$}$};
 \node[red] at (.1,-.5) {$\scalebox{.75}{$2$}$};
 \node[red] at (.6,-1) {$\scalebox{.75}{$6$}$};
 \node[red] at (1.1,-1.5) {$\scalebox{.75}{$2$}$};
 \node[red] at (1.6,-2) {$\scalebox{.75}{$6$}$};
 \node[red] at (2.1,-2.5) {$\scalebox{.75}{$2$}$};
 \node[red] at (2.6,-3) {$\scalebox{.75}{$6$}$};
 \node[red] at (3.1,-3.5) {$\scalebox{.75}{$2$}$};
 \node[red] at (9.3,-1.2) {$\scalebox{.75}{$3$}$};
 \node[red] at (8.3,-.6) {$\scalebox{.75}{$4$}$};

\node at (0,-1.2) {$\;$};

\end{tikzpicture}
\]
 Finally, we deal with  $K_1$, $K_2$, and $K_3$. 
 $H$ is also not contained in $\theta^{-1}(M_0)$ because $\theta(x_1)\in M_0$ (if also $\theta(x_0)\in M_0$, then $M_0=K_{(2,2)}$).
 For the same reason $K_3$ is distinct from $\theta^{-1}(M_0)$.
 As for $M_1$ and $M_2$, it suffices to check that $\theta(x_0)=x_0x_1x_4^{-1}x_0^{-3}\not \in M_1$ and $\theta(x_0)\not \in M_2=\sigma(M_1)$ (which is equivalent to showing that $\sigma(\theta(x_0))\not \in M_1=\sigma(M_2)$). This is done in the following figure.   
 
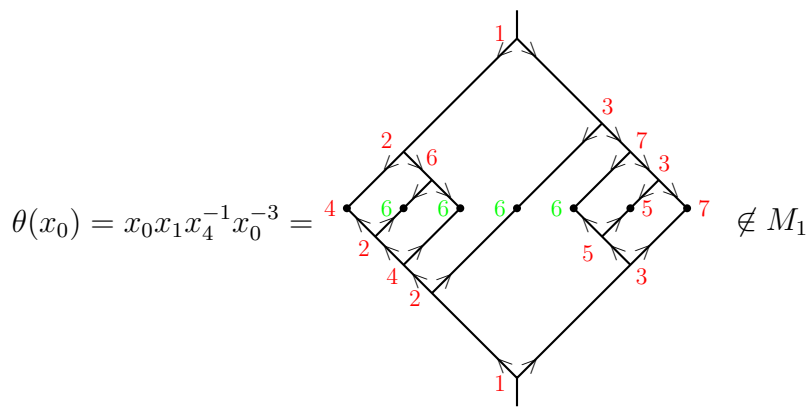
\[
 \begin{tikzpicture}[x=.75cm, y=.75cm,
    every edge/.style={
        draw,
      postaction={decorate,
                    decoration={markings}
                   }
        }
]

\node at (-5.25,-.25) {$\scalebox{1}{$\theta(x_0)=x_0x_1x_4^{-1}x_0^{-3}=$}$};
\node at (5.5,-.25) {$\scalebox{1}{$\not\in M_1$}$};
 
\draw[thick] (1,3)--(1,3.5); 
\draw[thick] (1,-3)--(1,-3.5); 
\draw[thick] (-2,0)--(1,3)--(4,0)--(1,-3)--(-2,0);

\draw[thick] (-1,1)--(0,0)--(-1,-1); 
\draw[thick] (2.5,1.5)--(-.5,-1.5); 
\draw[thick] (-1.5,-.5)--(-0.5,0.5);
\draw[thick] (3,-1)--(2,0)--(3,1); 
\draw[thick] (2.5,-.5)--(3.5,.5);

\node at (2.75,.65) {\rotatebox[origin=tr]{-135}{$\scalebox{.75}{$>$}$}};%
\node at (3.25,.75) {\rotatebox[origin=tr]{-45}{$\scalebox{.75}{$>$}$}};%
\node at (-.75,.15) {\rotatebox[origin=tr]{-135}{$\scalebox{.75}{$>$}$}};%
\node at (-.75,.75) {\rotatebox[origin=tr]{-45}{$\scalebox{.75}{$>$}$}};%
\node at (-1.25,.65) {\rotatebox[origin=tr]{-135}{$\scalebox{.75}{$>$}$}};%
\node at (.75,2.65) {\rotatebox[origin=tr]{-135}{$\scalebox{.75}{$>$}$}};%
\node at (1.25,2.75) {\rotatebox[origin=tr]{-45}{$\scalebox{.75}{$>$}$}};
\node at (3.25,.15) {\rotatebox[origin=tr]{-135}{$\scalebox{.75}{$>$}$}};
\node at (-.25,.25) {\rotatebox[origin=tr]{-45}{$\scalebox{.75}{$>$}$}};%
\node at (3.75,.25) {\rotatebox[origin=tr]{-45}{$\scalebox{.75}{$>$}$}};
\node at (2.75,1.25) {\rotatebox[origin=tr]{-45}{$\scalebox{.75}{$>$}$}};%
\node at (2.25,1.15) {\rotatebox[origin=tr]{-135}{$\scalebox{.75}{$>$}$}};%

\node at (-1.75,-.25) {\rotatebox[origin=tr]{135}{$\scalebox{.75}{$>$}$}};
\node at (2.25,-.25) {\rotatebox[origin=tr]{135}{$\scalebox{.75}{$>$}$}};
\node at (2.75,-.75) {\rotatebox[origin=tr]{135}{$\scalebox{.75}{$>$}$}};
\node at (.75,-2.75) {\rotatebox[origin=tr]{135}{$\scalebox{.75}{$>$}$}};
\node at (2.75,-.25) {\rotatebox[origin=tr]{45}{$\scalebox{.75}{$>$}$}};
\node at (3.25,-.75) {\rotatebox[origin=tr]{45}{$\scalebox{.75}{$>$}$}};
\node at (-1.25,-.75) {\rotatebox[origin=tr]{135}{$\scalebox{.75}{$>$}$}};
\node at (1.25,-2.75) {\rotatebox[origin=tr]{45}{$\scalebox{.75}{$>$}$}};
\node at (-.75,-1.25) {\rotatebox[origin=tr]{135}{$\scalebox{.75}{$>$}$}};
\node at (-.75,-.75) {\rotatebox[origin=tr]{45}{$\scalebox{.75}{$>$}$}};
\node at (-1.25,-.25) {\rotatebox[origin=tr]{45}{$\scalebox{.75}{$>$}$}};
\node at (-.25,-1.25) {\rotatebox[origin=tr]{45}{$\scalebox{.75}{$>$}$}};

  \fill (0,0)  circle[radius=1.5pt];
  \fill (1,0)  circle[radius=1.5pt];  
  \fill (2,0)  circle[radius=1.5pt];
  \fill (3,0)  circle[radius=1.5pt];
  \fill (4,0)  circle[radius=1.5pt];
  \fill (-1,0)  circle[radius=1.5pt];
  \fill (-2,0)  circle[radius=1.5pt];

 \node[red] at (-1.3,1.2) {$\scalebox{.75}{$2$}$};
 \node[red] at (-1.2,-1.2) {$\scalebox{.75}{$4$}$};
 \node[red] at (-.8,-1.6) {$\scalebox{.75}{$2$}$};
 \node[red] at (-1.7,-.7) {$\scalebox{.75}{$2$}$};
  \node[red] at (2.6,1.8) {$\scalebox{.75}{$3$}$};
 \node[green] at (-.3,0) {$\scalebox{.75}{$6$}$};
 \node[red] at (4.3,0) {$\scalebox{.75}{$7$}$};
  \node[red] at (3.2,1.2) {$\scalebox{.75}{$7$}$};
  \node[green] at (1.7,0) {$\scalebox{.75}{$6$}$};
  \node[red] at (3.3,0) {$\scalebox{.75}{$5$}$};
  \node[red] at (3.6,.8) {$\scalebox{.75}{$3$}$};
 \node[red] at (3.2,-1.2) {$\scalebox{.75}{$3$}$};
 \node[red] at (2.25,-.8) {$\scalebox{.75}{$5$}$};
 \node[green] at (.7,0) {$\scalebox{.75}{$6$}$};
 \node[green] at (-1.3,0) {$\scalebox{.75}{$6$}$};
 \node[red] at (-2.3,0) {$\scalebox{.75}{$4$}$};
 \node[red] at (.7,-3.1) {$\scalebox{.75}{$1$}$};
 \node[red] at (.7,3.1) {$\scalebox{.75}{$1$}$};
 \node[red] at (-.5,.9) {$\scalebox{.75}{$6$}$};
 
\node at (0,-1.2) {$\;$};

\end{tikzpicture}
 \]
 
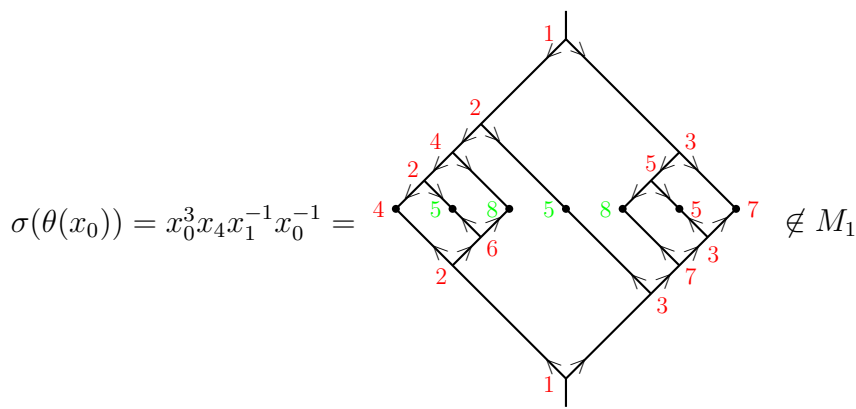
\[
 \begin{tikzpicture}[x=.75cm, y=.75cm,
    every edge/.style={
        draw,
      postaction={decorate,
                    decoration={markings}
                   }
        }
]

\node at (-5.75,-.25) {$\scalebox{1}{$\sigma(\theta(x_0))=x_0^3x_4x_1^{-1}x_0^{-1}=$}$};
\node at (5.5,-.25) {$\scalebox{1}{$\not\in M_1$}$};
 
\draw[thick] (1,3)--(1,3.5); 
\draw[thick] (1,-3)--(1,-3.5); 
\draw[thick] (-2,0)--(1,3)--(4,0)--(1,-3)--(-2,0);

\draw[thick] (-1,1)--(0,0)--(-1,-1); 
\draw[thick] (2.5,-1.5)--(-.5,1.5); 
\draw[thick] (-1.5,.5)--(-0.5,-0.5);
\draw[thick] (3,-1)--(2,0)--(3,1); 
\draw[thick] (2.5,.5)--(3.5,-.5);

\node at (2.75,.65) {\rotatebox[origin=tr]{-135}{$\scalebox{.75}{$>$}$}};%
\node at (3.25,.75) {\rotatebox[origin=tr]{-45}{$\scalebox{.75}{$>$}$}};%
\node at (-1.75,.15) {\rotatebox[origin=tr]{-135}{$\scalebox{.75}{$>$}$}};%
\node at (-.75,.75) {\rotatebox[origin=tr]{-45}{$\scalebox{.75}{$>$}$}};%
\node at (-1.25,.65) {\rotatebox[origin=tr]{-135}{$\scalebox{.75}{$>$}$}};%
\node at (.75,2.65) {\rotatebox[origin=tr]{-135}{$\scalebox{.75}{$>$}$}};%
\node at (1.25,2.75) {\rotatebox[origin=tr]{-45}{$\scalebox{.75}{$>$}$}};
\node at (2.25,.15) {\rotatebox[origin=tr]{-135}{$\scalebox{.75}{$>$}$}};
\node at (-1.25,.25) {\rotatebox[origin=tr]{-45}{$\scalebox{.75}{$>$}$}};%
\node at (2.75,.25) {\rotatebox[origin=tr]{-45}{$\scalebox{.75}{$>$}$}};
\node at (-.25,1.25) {\rotatebox[origin=tr]{-45}{$\scalebox{.75}{$>$}$}};%
\node at (-.75,1.15) {\rotatebox[origin=tr]{-135}{$\scalebox{.75}{$>$}$}};%

\node at (-.75,-.25) {\rotatebox[origin=tr]{135}{$\scalebox{.75}{$>$}$}};
\node at (3.25,-.25) {\rotatebox[origin=tr]{135}{$\scalebox{.75}{$>$}$}};
\node at (2.75,-.75) {\rotatebox[origin=tr]{135}{$\scalebox{.75}{$>$}$}};
\node at (.75,-2.75) {\rotatebox[origin=tr]{135}{$\scalebox{.75}{$>$}$}};
\node at (3.75,-.25) {\rotatebox[origin=tr]{45}{$\scalebox{.75}{$>$}$}};
\node at (3.25,-.75) {\rotatebox[origin=tr]{45}{$\scalebox{.75}{$>$}$}};
\node at (-1.25,-.75) {\rotatebox[origin=tr]{135}{$\scalebox{.75}{$>$}$}};
\node at (1.25,-2.75) {\rotatebox[origin=tr]{45}{$\scalebox{.75}{$>$}$}};
\node at (2.25,-1.25) {\rotatebox[origin=tr]{135}{$\scalebox{.75}{$>$}$}};
\node at (-.75,-.75) {\rotatebox[origin=tr]{45}{$\scalebox{.75}{$>$}$}};
\node at (-.25,-.25) {\rotatebox[origin=tr]{45}{$\scalebox{.75}{$>$}$}};
\node at (2.75,-1.25) {\rotatebox[origin=tr]{45}{$\scalebox{.75}{$>$}$}};

  \fill (0,0)  circle[radius=1.5pt];
  \fill (1,0)  circle[radius=1.5pt];  
  \fill (2,0)  circle[radius=1.5pt];
  \fill (3,0)  circle[radius=1.5pt];
  \fill (4,0)  circle[radius=1.5pt];
  \fill (-1,0)  circle[radius=1.5pt];
  \fill (-2,0)  circle[radius=1.5pt];

 \node[red] at (-1.3,1.2) {$\scalebox{.75}{$4$}$};
 \node[red] at (-1.2,-1.2) {$\scalebox{.75}{$2$}$};
 \node[red] at (-.6,1.8) {$\scalebox{.75}{$2$}$};
 \node[red] at (-1.7,.7) {$\scalebox{.75}{$2$}$};
  \node[red] at (2.5,.8) {$\scalebox{.75}{$5$}$};
 \node[green] at (-.3,0) {$\scalebox{.75}{$8$}$};
 \node[red] at (4.3,0) {$\scalebox{.75}{$7$}$};
  \node[red] at (3.2,1.2) {$\scalebox{.75}{$3$}$};
  \node[green] at (1.7,0) {$\scalebox{.75}{$8$}$};
  \node[red] at (3.3,0) {$\scalebox{.75}{$5$}$};
  \node[red] at (3.6,-.8) {$\scalebox{.75}{$3$}$};
 \node[red] at (3.2,-1.2) {$\scalebox{.75}{$7$}$};
 \node[red] at (2.7,-1.7) {$\scalebox{.75}{$3$}$};
 \node[green] at (.7,0) {$\scalebox{.75}{$5$}$};
 \node[green] at (-1.3,0) {$\scalebox{.75}{$5$}$};
 \node[red] at (-2.3,0) {$\scalebox{.75}{$4$}$};
 \node[red] at (.7,-3.1) {$\scalebox{.75}{$1$}$};
 \node[red] at (.7,3.1) {$\scalebox{.75}{$1$}$};
 \node[red] at (-.3,-.7) {$\scalebox{.75}{$6$}$};
 
\node at (0,-1.2) {$\;$};

\end{tikzpicture}
 \]
\end{proof}
%
  
Up to isomorphism, there are at least four maximal infinite index subgroups of $F$: Stab($1/2$), Stab($1/3$), Stab($\sqrt{2}/2$), $\beta^{-1}(\vec{F})$ 
\begin{question}
Are $M_0$, $M_1$, and $M_2$ isomorphic to the other known maximal infinite index subgroups of $F$? What about Golan's subgroups $K_1$, $K_2$, $K_3$?
\end{question}
So far we have exhibited three subgroups containing $\CF$
: $M_0$, $M_1$ and $M_2$. 
It is natural to wonder if $\CF$ is of quasi-finite index in $K_{(2,2)}$, that is, if there are only finitely many subgroups of $K_{(2,2)}$ containing $\CF$. 
Here we prove some partial results suggesting that $M_0$, $M_1$ and $M_2$ might be the only subgroups between $\CF$ and $K_{(2,2)}$.

\begin{lemma}\label{exist-block}
Let $g$ be a normal form  in $F_+\cap K_{(2,2)}$ containing at least two letters, then there exists a block $h$ in $\langle g, \CF\rangle$.
\end{lemma}
\begin{proof}
Consider the normal form of $g=x_{i_1}^ax_{i_2}\cdots x_{i_n}$, $a\in \IN$.
 For every $j$, there exists a $k_j\in\IN$ such that $i_j<i_1+a+k_j+1$. Let $k$ be $\sum_j k_j$ and consider 
\begin{align*}
h:=\varphi^{i_1}(w_1^k)g&=\varphi^{i_1}(x_0x_1^{2k}x_0^{-1})g=x_{i_1}x_{i_1+1}^{2k}x_{i_1}^{a-1}x_{i_2}\cdots x_{i_n}\\
&=x_{i_1}^ax_{i_1+a}^{2k}x_{i_2}\cdots x_{i_n}\; .
\end{align*}
Now $h$ is a block. Indeed, 
\begin{align*}
& i_1+a<i_1+a+1, \\
& i_2<i_1+a+k_2+1<i_1+a+2k+1\\
& i_3<i_1+a+k_3+1<i_1+a+2k+2\\
& i_j<i_1+a+k_j+1<i_1+a+2k+(j-1) \qquad j\geq 3
\end{align*}
\end{proof}

\begin{proposition}
Let  $g\in F_+\cap K_{(2,2)}$, then $\langle g, \CF\rangle$ is equal to $M_0$, $M_1$, or $K_{(2,2)}$.
\end{proposition}
\begin{proof}
Thanks to Lemma \ref{lemma601} we may assume that the normal form of $g$ contains at least two letters.
By Lemma \ref{exist-block} there exists a block $h\in\langle g, \CF\rangle$.
Consider its normal form $h=x_{i_1}\cdots x_{i_n}$, $w_1=x_0x_1^2x_0^{-1}$.
Recall that if $h$ is a block, then $x_{k}^{\pm 1}h=hx_{k+|h|}^{\pm 1}$ for all $k>i_1$ (here $|h|$ is the length of the normal form of $h$).
Then, 
\begin{align*}
h^{-1}\varphi^{i_1}(w_1)h&=h^{-1}x_{i_1}x_{i_1+1}^2x_{i_1}^{-1}x_{i_1}x_{i_2}\cdots x_{i_n}\\
&=h^{-1}x_{i_1}x_{i_2}\cdots x_{i_n}x_{i_1+n}^2=h^{-1}hx_{i_1+n}^2=x_{i_1+n}^2\in \langle g, \CF\rangle
\end{align*}
where we used that $x_{i_1+1}$ skips $x_{i_2}\cdots x_{i_n}$.
Now if $i_1+n\equiv_2 0$, then $M_0\leq\langle g, \CF\rangle\leq K_{(2,2)}$. Otherwise, we have $M_1\leq \langle g, \CF\rangle\leq K_{(2,2)}$. In both cases now the claim follows from the maximality of $M_0$ and $M_1$ in $K_{(2,2)}$.
\end{proof}

\section*{Acknowledgements}
We would like to thank the referee for their attentive perusal of the manuscript, which resulted in many improvements in the presentation of the results of this paper.
V.A. acknowledges the support from the Swiss National Science foundation through the SNF project no. 178756 (Fibred links, L-space covers and algorithmic knot theory) and from the Department of Mathematics of the University of Geneva.  T.N. acknowledges support of  Swiss NSF grants 200020-178828 and 200020-200400.

\end{document}